\documentclass[reqno,12pt,letterpaper]{amsart}
\usepackage[proof]{sdmacros}

\DeclareMathOperator{\Sp}{Sp}

\title[Resonances for open quantum maps]%
{Resonances for open quantum maps\\
and a fractal uncertainty principle}
\author{Semyon Dyatlov}
\email{dyatlov@math.mit.edu}
\address{Department of Mathematics, Massachusetts Institute of Technology,
77 Massachusetts Ave, Cambridge, MA 02139}
\author{Long Jin}
\email{long249@purdue.edu}
\address{Department of Mathematics,
Purdue University,
150 N University Street, West Lafayette, IN 47907}

\begin{document}

\begin{abstract}
We study eigenvalues of quantum open baker's maps with trapped sets
given by linear arithmetic Cantor sets of dimensions $\delta\in (0,1)$.
We show that the size of the spectral gap is strictly greater than the standard
bound $\max(0,{1\over 2}-\delta)$ for all values of $\delta$, which is the first result of this kind.
The size of the improvement is determined from
a fractal uncertainty principle and can be computed for any given Cantor set.
We next show a fractal Weyl upper bound for the number of eigenvalues in annuli, with exponent
which depends on the inner radius of the annulus.
\end{abstract}

\maketitle

\addtocounter{section}{1}
\addcontentsline{toc}{section}{1. Introduction}

Open quantum maps are useful models in the study of scattering
phenomena and in particular
scattering resonances. They quantize canonical relations on compact
symplectic manifolds, giving families of operators defined on
finite dimensional Hilbert spaces. This makes them attractive
models for numerical experimentation.
See~\S\ref{s:history} for an overview of some of the
previous results in physics and mathematics literature.

The present paper investigates eigenvalues (related to resonances
by~\eqref{e:eig-res} below) for a family of open quantum maps known as
\emph{quantum open baker's maps}. The corresponding trapped
orbits form Cantor sets. The combinatorial
and number theoretic properties of these sets
make it possible to prove
results on spectral gaps
(Theorems~\ref{t:gap-improves}, \ref{t:gap-detailed})
which lie well beyond what is known for
other models. We also obtain a fractal Weyl upper bound
(Theorem~\ref{t:weyl}) and provide numerical results (see~\S\ref{s:numerics}).

The quantum open baker's maps we study are determined by triples
\begin{equation}
  \label{e:oqm-triple}
(M,\mathcal A,\chi),\quad
M\in\mathbb N,\quad
\mathcal A\subset\{0,\dots,M-1\},\quad
\chi\in C_0^\infty\big((0,1);[0,1]\big).
\end{equation}
We call $M$ the \emph{base}, $\mathcal A$ the \emph{alphabet}, and
$\chi$ the \emph{cutoff function}. For each $k\in\mathbb N$, the corresponding
open quantum baker's map is the
operator on
$$
\ell^2_N=\ell^2(\mathbb Z_N),\quad
\mathbb Z_N=\mathbb Z/(N\mathbb Z),\quad
N:=M^k
$$
defined as follows (see~\S\ref{s:setup} for details):
\begin{equation}
  \label{e:B-N-def}
B_N=B_{N,\chi}=\mathcal F_N^*\begin{pmatrix}
\chi_{N/M}\mathcal F_{N/M}\chi_{N/M} & \\
&\ddots&\\
 & &\chi_{N/M}\mathcal F_{N/M}\chi_{N/M}
\end{pmatrix}I_{\mathcal A, M}
\end{equation}
where $\mathcal F_N$ is the unitary Fourier transform,
$\chi_{N/M}$ is the multiplication operator
on~$\ell^2_{N/M}$ discretizing $\chi$, and
$I_{\mathcal A,M}$ is the diagonal matrix
with $\ell$-th diagonal entry equal to~1 if $\lfloor {\ell\over N/M}\rfloor\in\mathcal A$
and 0 otherwise. A basic example is $M=3$, $\mathcal A=\{0,2\}$, giving
$$
N=3^k,\quad
B_N=\mathcal F_N^*\begin{pmatrix}
\chi_{N/3}\mathcal F_{N/3}\chi_{N/3}&0&0\\
0&0&0\\
0&0&\chi_{N/3}\mathcal F_{N/3}\chi_{N/3}\end{pmatrix}.
$$
The operator $B_{N}$ is the discrete analog of a Fourier integral
operator corresponding to the classical open baker's map, which
is the following symplectic relation on the torus~$\mathbb T^2_{x,\xi}$
(see \S\ref{s:setup} and Figure~\ref{f:basic-baker})
\begin{equation}
  \label{e:kappa-def}
\begin{gathered}
\varkappa_{M,\mathcal A}:(y,\eta)\mapsto (x,\xi)=\Big(My-a,{\eta+a\over M}\Big),\\
(y,\eta)\in \Big({a\over M},{a+1\over M}\Big)\times (0,1),\quad
a\in\mathcal A.
\end{gathered}
\end{equation}
\begin{figure}
\includegraphics{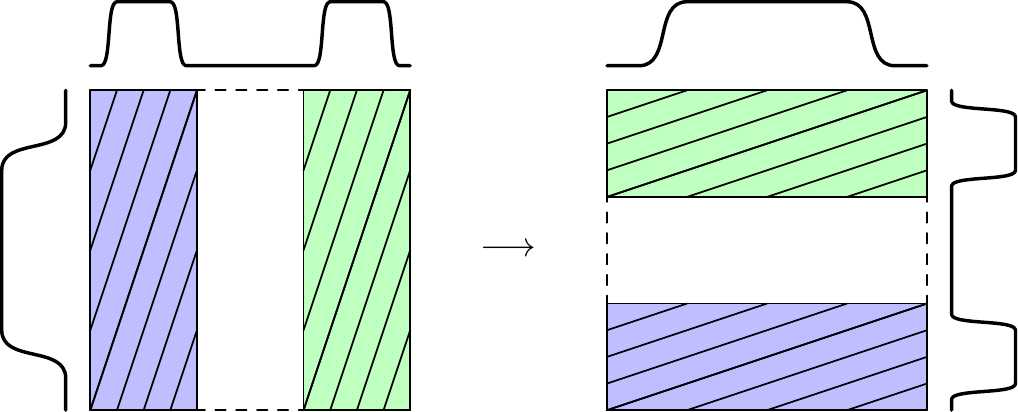}
\caption{The classical open baker's map $\varkappa_{M,\mathcal A}$ for
$M=3,\mathcal A=\{0,2\}$, picturing the cutoffs forming the symbol
of the Fourier integral operator.}
\label{f:basic-baker}
\end{figure}%
The symbol of the Fourier integral operator has the form $\chi(x)\chi(\eta)$
and cuts off from the boundaries of the rectangles where the transformation
$\varkappa_{M,\mathcal A}$ is defined. Without the cutoff $\chi$,
the operator $B_N$ would have additional singularities
which would change the spectrum~-- see~\S\ref{s:cutoff-dep} below.

The linearization of $\varkappa_{M,\mathcal A}$ has eigenvalues $M,1/M$.
The operator $B_{N}$ is then a toy model for the time $t=\log M$
propagator of a quantum system which has classical expansion rate $1$
(such as a convex co-compact hyperbolic surface, see for instance~\cite{hgap}),
at frequencies $\sim N$. In particular, if $\omega$, $\Im\omega\leq 0$, is a scattering resonance
for the quantum system, then the corresponding eigenvalue
of $B_{N}$ is
\begin{equation}
  \label{e:eig-res}
\lambda=e^{-it\omega}=M^{-i\omega},\quad
|\lambda|=M^{\Im\omega}\leq 1.
\end{equation}
This formula provides an analogy between gap/counting results for
scattering resonances and those for eigenvalues of open quantum maps.

We assume that
$1<|\mathcal A|<M$ and define the parameter
\begin{equation}
  \label{e:delta}
\delta:={\log |\mathcal A|\over\log M}\in (0,1)
\end{equation}
which is the dimension of the corresponding Cantor set~-- see~\eqref{e:C-k}, \eqref{e:C-infty} below. We remark that the topological pressure of the time $t=\log M$ suspension of the
map $\varkappa_{M,\mathcal A}$ on the trapped set~\eqref{e:trapped-set} is given by~\cite[\S8.2.2]{Nonnenmacher}
\begin{equation}
  \label{e:under-pressure}
P(s)=\delta-s,\quad s\in\mathbb R.
\end{equation}
Here we choose time $\log M$ suspension so that it has expansion rate 1,
see the discussion preceding~\eqref{e:eig-res}.

\subsection{Spectral gaps}
  \label{s:intro-gaps}

The matrix $B_N$ has norm bounded by 1, therefore its spectrum $\Sp(B_N)$ is contained
in the unit disk. Our first result shows in particular that
the spectral radius of $B_N$ is less than 1,
uniformly as $N\to\infty$:
\begin{theo}[Improved spectral gap]
  \label{t:gap-improves}
There exists
\begin{equation}
  \label{e:improved-beta}
\beta=\beta(M,\mathcal A)>\max\Big(0,{1\over 2}-\delta\Big)
\end{equation}
such that
\begin{equation}
  \label{e:gap-improves}
\limsup_{N\to\infty}\max\{|\lambda|\colon \lambda\in\Sp(B_N)\}\ \leq\ M^{-\beta}.
\end{equation}
\end{theo}
The size of the gap~\eqref{e:improved-beta} improves over both the trivial bound
and the pressure bound $-P({1\over 2})={1\over 2}-\delta$
(see~\cite[\S8]{Nonnenmacher})
 for the entire range $\delta\in (0,1)$.
See~\S\ref{s:history} below for an overview of previously known spectral gap results.
We also provide a polynomial resolvent bound for
$|\lambda|>M^{-\beta}$, see Proposition~\ref{l:resolvent-bound}.

\begin{figure}
\includegraphics[scale=0.5]{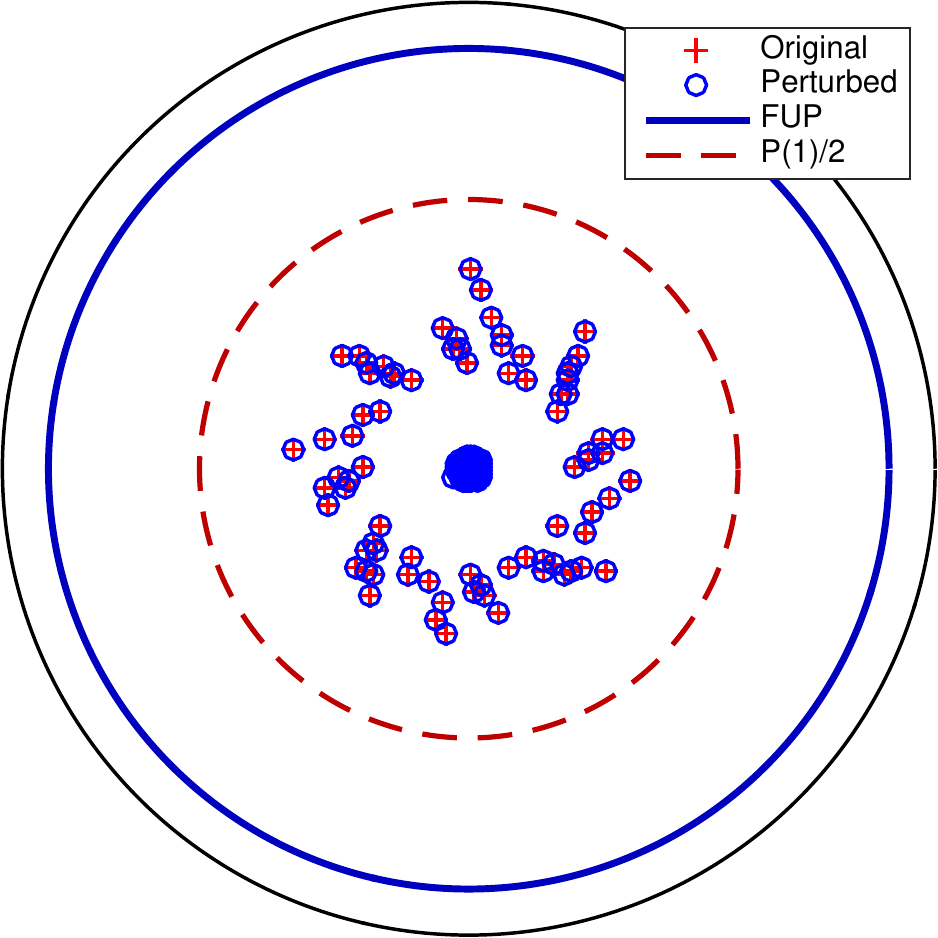}\quad
\includegraphics[scale=0.5]{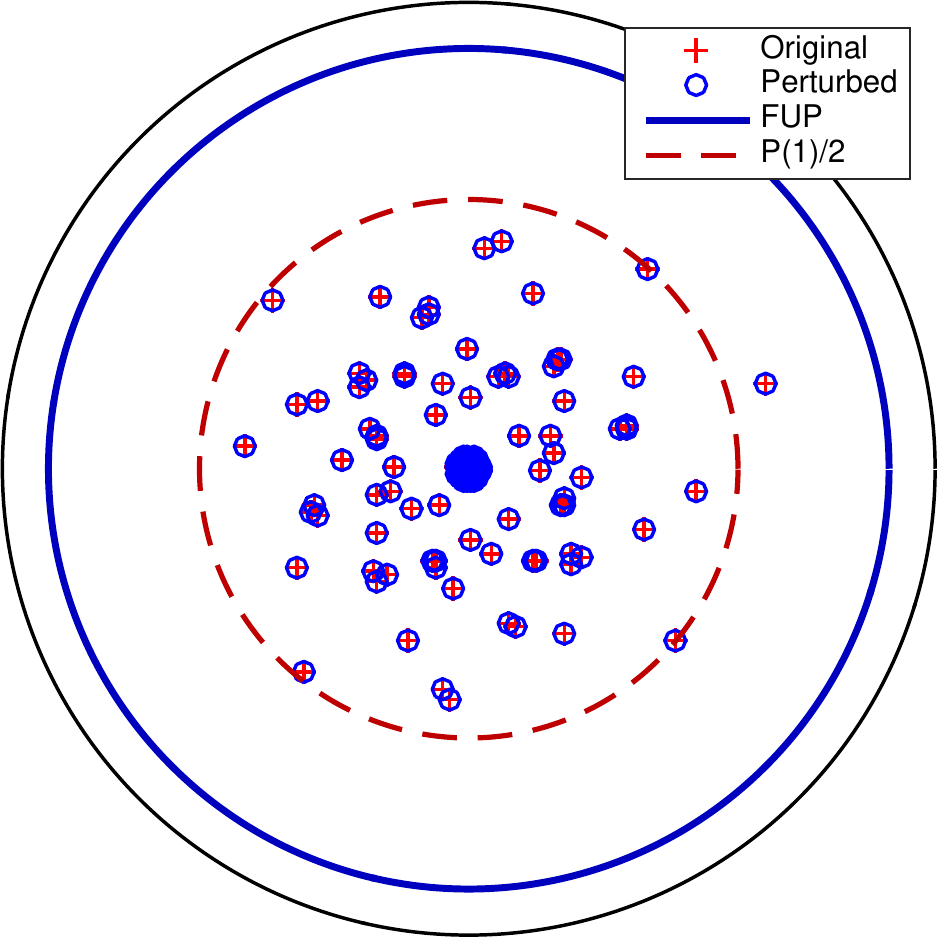}\quad
\includegraphics[scale=0.5]{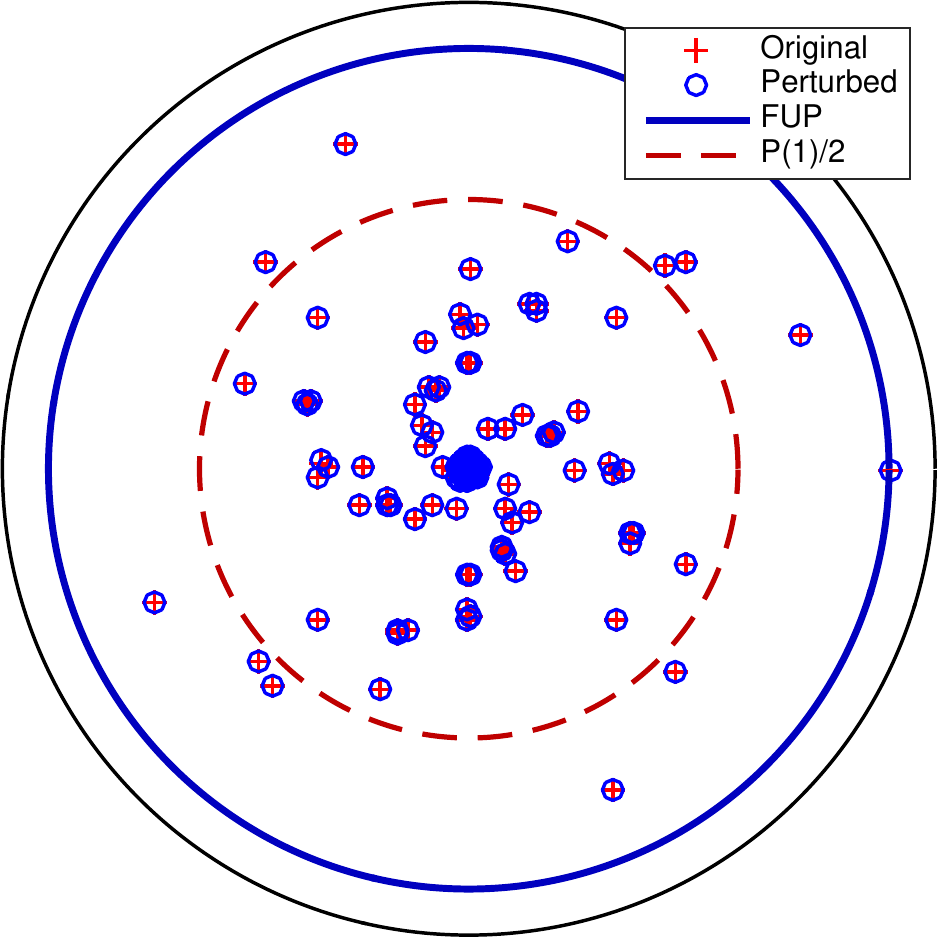}
\hbox to\hsize{\hss $\mathcal A=\{1,2,3\}$
\hss\hss $\mathcal A=\{2,3,4\}$
\hss\hss $\mathcal A=\{3,4,5\}$\hss}
\caption{The spectrum of $B_N$ for $M=9$, $k=4$, and three
different alphabets with $\delta=1/2$. The spectral radius
bound~\eqref{e:gap-improves} (labeled `FUP')
is close to being sharp for $\mathcal A=\{3,4,5\}$
but not for the other two alphabets.
We use a small random perturbation
to test the stability of the numerics.
See~\S\ref{s:numerics}
for the notation used in the eigenvalue plots.}
\label{f:fup-howgood}
\end{figure}

The constant in~\eqref{e:improved-beta} can be computed
as follows. Define the $k$-th order Cantor set
\begin{equation}
  \label{e:C-k}
\mathcal C_k=\mathcal C_k(M,\mathcal A)=\Big\{\sum_{j=0}^{k-1} a_j M^j\,\Big|\, a_0,\dots,a_{k-1}\in\mathcal A\Big\}\subset \mathbb Z_N.
\end{equation}
Denoting by $\indic_{\mathcal C_k}$ the multiplication operator by the indicator function of $\mathcal C_k$,
define
\begin{equation}
  \label{e:r-k}
r_k=\|\indic_{\mathcal C_k}\mathcal F_N\indic_{\mathcal C_k}\|_{\ell^2_N\to\ell^2_N}.
\end{equation}
Then Theorem~\ref{t:gap-improves} is a corollary of the more precise
\begin{theo}
  \label{t:gap-detailed}
There exists a limit, called the \textbf{fractal uncertainty exponent},
\begin{equation}
  \label{e:beta-fup}
\beta=\beta(M,\mathcal A)=-\lim_{k\to\infty}{\log r_k\over k\log M} >\max\Big(0,{1\over 2}-\delta\Big)
\end{equation}
and~\eqref{e:gap-improves} holds for this choice of $\beta$.
\end{theo}
The definition~\eqref{e:beta-fup} implies the following bound,
which we call \emph{fractal uncertainty principle} since it says that
no function can be supported on $\mathcal C_k$ in both position and frequency:
\begin{equation}
  \label{e:fup-boundd}
\|\indic_{\mathcal C_k}\mathcal F_N\indic_{\mathcal C_k}\|_{\ell^2_N\to \ell^2_N}\leq C_\varepsilon N^{-\beta+\varepsilon}\quad\text{for all }\varepsilon>0.
\end{equation}
The fractal uncertainty principle implies a bound on the spectral radius of $B_N$ by
the following argument which previously appeared in the setting
of hyperbolic manifolds in~\cite{hgap}:
an eigenfunction with eigenvalue~$\lambda$, $|\lambda|\geq M^{-\beta+\varepsilon}$,
would give a counterexample to~\eqref{e:fup-boundd} since (a) it is essentially supported
near $\mathcal C_k$ in frequency and (b) its mass when restricted to near $\mathcal C_k$
in position has a lower bound. 
We present the proof in~\S\ref{s:oqm}; due to the explicit nature
of open quantum maps it is greatly
simplified on the technical level compared to~\cite{hgap}.

Theorems~\ref{t:gap-improves} and~\ref{t:gap-detailed} only give
an upper bound on the spectral radius of $B_N$. Lower bounds are
difficult to prove mathematically because this would involve showing
existence of eigenvalues of nonselfadjoint operators. However,
numerical evidence suggests that there are cases for which~\eqref{e:gap-improves}
is close to being sharp~-- see Figures~\ref{f:fup-howgood}, \ref{f:improved-pressure}, \ref{f:improved-0}, and~\ref{f:Weyl-great}.
We also remark that the constant $\beta(M,\mathcal A)$ does
not depend on the cutoff~$\chi$. The spectral
radius of $B_N$ may be much smaller than $\beta(M,\mathcal A)$ for
some choices of $\chi$ (for instance $\chi\equiv 0$)
however the size of a spectral gap with polynomial resolvent bound
is independent of $\chi$ as long as $\chi=1$ near $\mathcal C_\infty$,
see Theorem~\ref{t:cutoff-dependence} below.

It is easy to show~\eqref{e:fup-boundd} with $\beta=\max(0,{1\over 2}-\delta)$
using only the size of $\mathcal C_k$, see~\eqref{e:this-is-trivial}.
The proof that~\eqref{e:fup-boundd} holds for some $\beta>\max(0,{1\over 2}-\delta)$,
presented in~\S\S\ref{s:submul}--\ref{s:improve-2},
is more complicated and uses the algebraic structure of the Cantor sets
$\mathcal C_k$.
In particular it relies on a submultiplicative
inequality~\eqref{e:submultiplicative}, which uses that $N$ is a power of~$M$
and does not seem to extend to more general situations.
More recent results 
of Bourgain--Dyatlov~\cite{fullgap}
and Dyatlov--Jin~\cite{regfup} give a fractal uncertainty principle
for the much more general class of Ahlfors--David regular sets.
They in particular imply that Theorem~\ref{t:gap-improves}
holds without the assumption $N=M^k$ (though with less information
on the size of $\beta$)~-- see~\cite[\S5]{regfup}.

The value of the exponent $\beta$ in~\eqref{e:beta-fup} varies
with the choice of the alphabet, even for fixed $M,\delta$~-- see Figure~\ref{f:fupcloud}.
We summarize several quantitative results regarding this dependence,
valid for large $M$ and proved in~\S\ref{s:fup}:
\begin{figure}
\includegraphics[scale=0.5125]{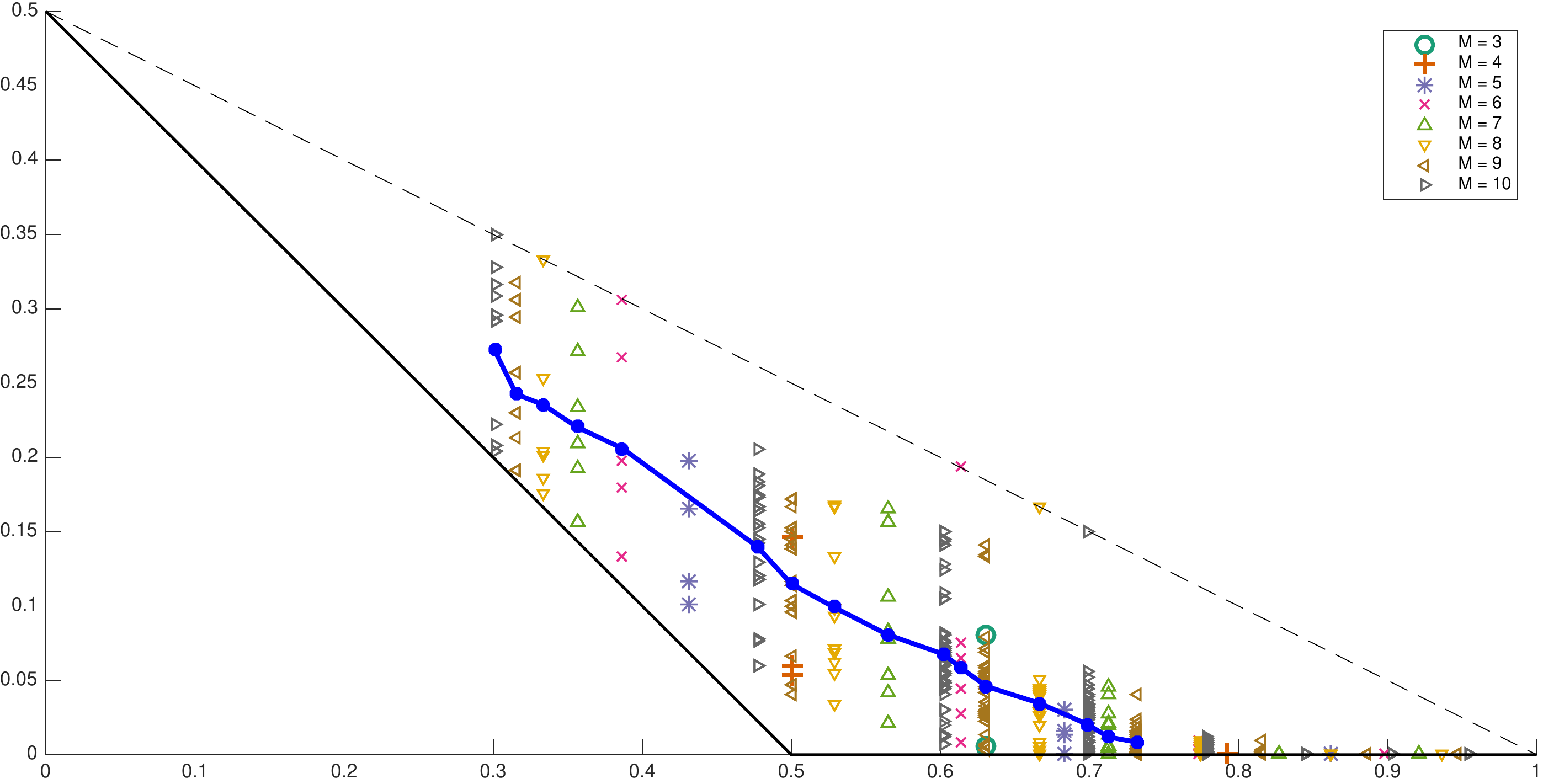}
\caption{Numerically approximated fractal uncertainty exponents for all possible
alphabets with $M\leq 10$. Here the $x$ axis represents $\delta$
and the $y$ axis represents $\beta$.
The solid black line is $\beta=\max(0,{1\over 2}-\delta)$
and the dashed line is $\beta={1-\delta\over 2}$.
The solid blue line is the average value of $\beta$ over all alphabets
with given $M\in [6,10]$ and $\delta\leq 0.75$.
See~\S\ref{s:numerics} for details.}
\label{f:fupcloud}
\end{figure}

\begin{enumerate}
\item For $\delta\leq {1\over 2}$, the value $\beta-(1/2-\delta)$ is bounded below
by a negative power of $M$~-- see Corollary~\ref{l:improve-pressure-cantor}.
For $\delta<{1\over 2}$, there exist alphabets for which $\beta-(1/2-\delta)$ is also
bounded above by a (different) negative power of $M$~-- see Proposition~\ref{l:lower-1/2}.
\item For $\delta={1\over 2}+o({1\over\log M})$, $\beta$ is bounded below
by ${1\over \mathbf K\log M}$ (here $\mathbf K$ denotes a global constant)~--
see Proposition~\ref{l:improved-ae-ultimate}. In fact,
$\beta$ can be estimated in terms of the additive energy of $\mathcal C_k$. 
There exist alphabets with $\delta={1\over 2}$ for which $\beta$ is bounded
above by $\mathbf K\over \log M$~-- see Proposition~\ref{l:lower-1/2}.
\item
For $\delta>{1\over 2}$, $\beta$ is bounded below by
$$
\beta\geq \exp\big(-M^{{\delta\over 1-\delta}+o(1)}\big),
$$
see Corollary~\ref{l:improve-0-cantor}.
We do not prove matching upper bounds but numerical evidence
in Table~\ref{b:worst-FUP} suggests that there exists
alphabets with $\beta$ exponentially small in $M$.
\item We always have
\begin{equation}
  \label{e:jnn}
\beta\leq {1-\delta\over 2}=-{P(1)\over 2},
\end{equation}
corresponding to the classical escape rate (see~\S\ref{s:history}
and~\eqref{e:JN-2}),
and for a generic alphabet the inequality in~\eqref{e:jnn}
is strict~-- see Proposition~\ref{l:not-the-best}.
However, there exist infinitely many pairs $(M,\mathcal A)$ for which
$\beta={1-\delta\over 2}$, see~\S\ref{s:special-alphabets}. 
Numerical evidence suggests
that for these special alphabets the spectrum of $B_N$ has a \emph{band structure},
making it the most feasible case for proving a fractal Weyl asymptotic
for the number of eigenvalues~-- see Conjecture~\ref{l:Weyl-great}.
\item Finally, the expected value of $\beta$ for large $M$ and a randomly
chosen $\mathcal A$ of fixed size appears to be much larger than $\max(0,{1\over 2}-\delta)$,
see the solid blue line on Figure~\ref{f:fupcloud}. A related
question of $L^p$ Fourier restriction bounds for random sets
was investigated by Bourgain~\cite{Bourgain}.
Examples of random multiscale Cantor sets satisfying $L^p$ restriction bounds and Fourier decay estimates
were constructed by Chen--Seeger~\cite{chen-seeger}, Shmerkin--Suomala~\cite{shmerkin},
and \L aba--Wang~\cite{laba-wang}.
\end{enumerate}

\subsection{Weyl bounds}

Our next result concerns the counting function
\begin{equation}
  \label{e:N-k}
\mathcal N_k(\nu)=\big|\Sp(B_N)\cap \{|\lambda|\geq M^{-\nu}\}\big|,\quad
\nu\geq 0,
\end{equation}
where eigenvalues of $B_N$ are counted with multiplicities.
We obtain a Weyl upper bound on $\mathcal N_k(\nu)$ (see~\cite[\S6.1]{Nonnenmacher}):
\begin{figure}
\includegraphics[scale=0.55]{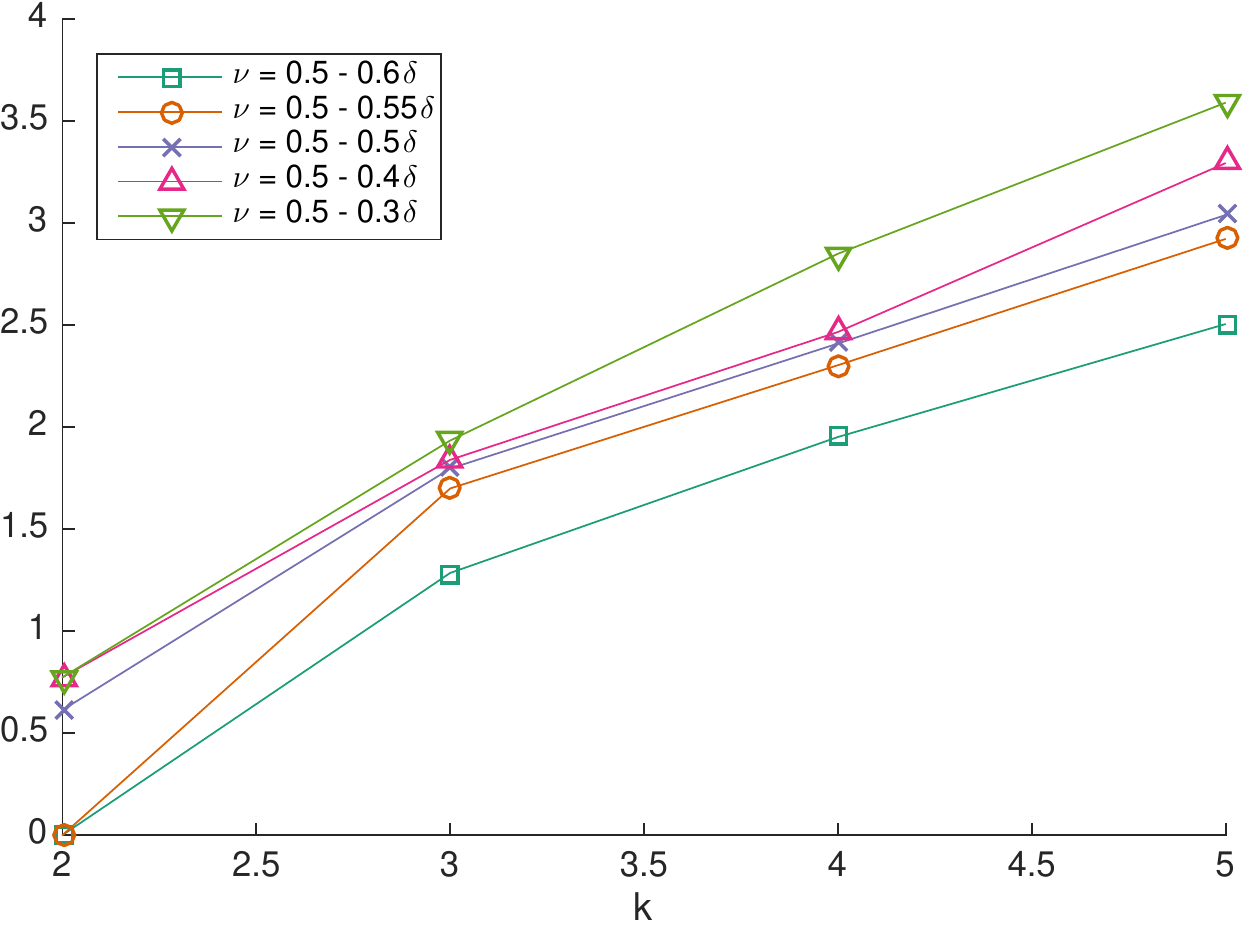}\quad
\includegraphics[scale=0.55]{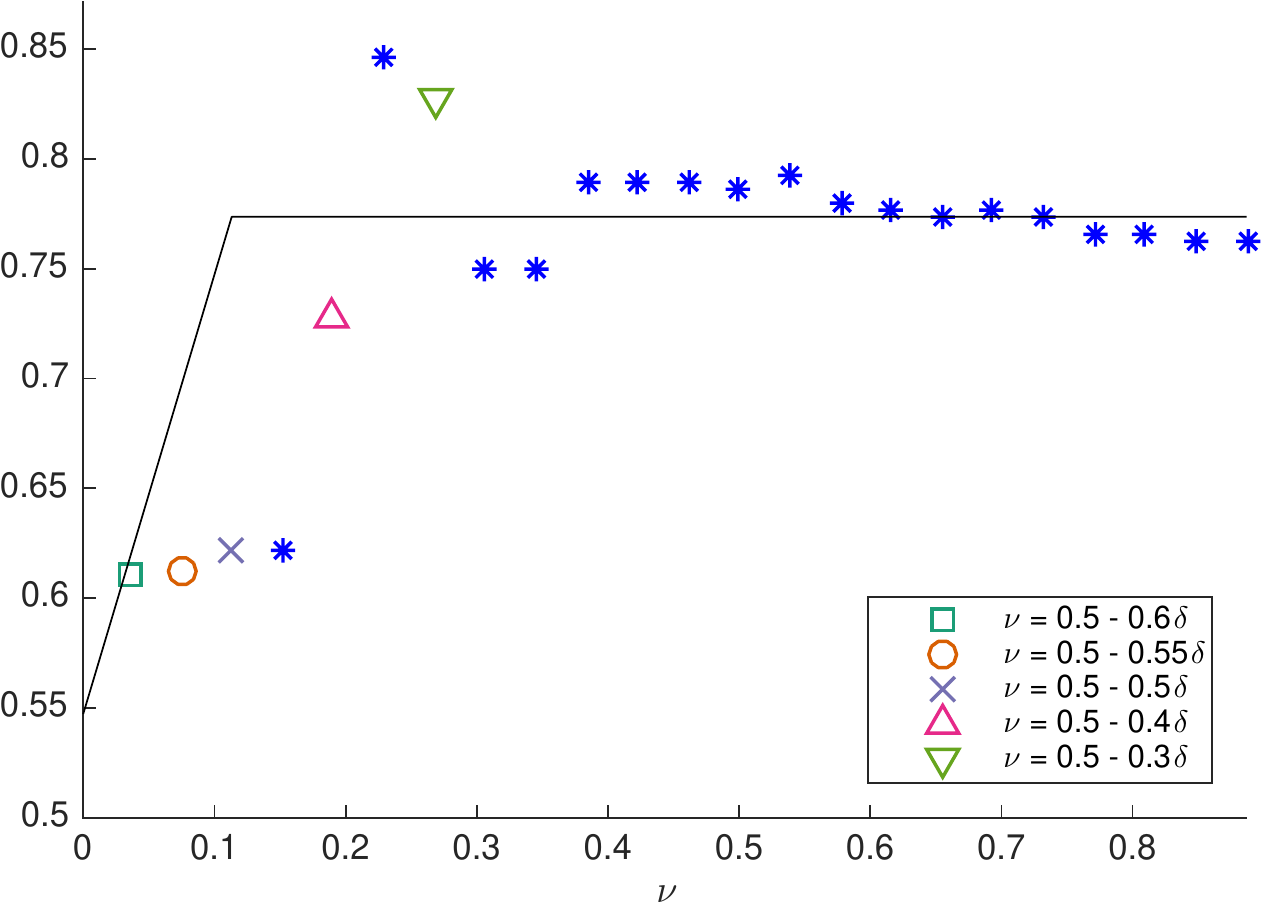}
\caption{Left: plot of $\log\mathcal N_k(\nu)/\log M$
for $M=6$, $\mathcal A=\{1, 2, 3, 4\}$, as a function of $k=2,\dots,5$
and for various values of $\nu$.
We have $\delta=\log 4 / \log 6 \approx 0.7737$.
Right: the slope of best linear approximation to the graph on the
left restricted to $k=3,4,5$, as a function
of $\nu$, together with the bound $m(\delta,\nu)$ of Theorem~\ref{t:weyl}.}
\label{f:Weyl}
\end{figure}
%
\begin{theo}[Weyl bounds]
  \label{t:weyl}
For each $\nu>0$ and $\varepsilon>0$, we have as $k\to\infty$
\begin{equation}
  \label{e:weyl}
\mathcal N_k(\nu)=\mathcal O(N^{m(\delta,\nu)+\varepsilon}),\quad
m(\delta,\nu)=\min(2\nu+2\delta-1,\delta).
\end{equation}
\end{theo}
The proof, presented in~\S\ref{s:Weyl}, uses the argument introduced
for hyperbolic surfaces in~\cite{ifwl}.
Note that $m(\delta,\nu)=\delta$ for $\nu\geq {1-\delta\over 2}$,
corresponding to the standard Weyl law (see~\S\ref{s:history}).
For $\nu\leq{1-\delta\over 2}$, the exponent
$m(\delta,\nu)$ interpolates linearly between
$m(\delta,{1-\delta\over 2})=\delta$
and $m(\delta,{1\over 2}-\delta)=0$, the latter
corresponding to the pressure gap. 

While no matching lower bounds on $\mathcal N_k(\nu)$ are known rigorously, numerical
evidence on Figure~\ref{f:Weyl} suggests that $\mathcal N_k(\nu)\sim N^{\delta}$ for $\nu$ large enough.
However, because of the small number of data points available
(and the resulting artefacts such as rough behavior of
the exponents in the right half of Figure~\ref{f:Weyl}) we could not determine
how close~\eqref{e:weyl} is to the optimal bound.

\subsection{Dependence on cutoff}
  \label{s:cutoff-dep}

Our final result, proved in~\S\ref{s:cutoff-dependence},
concerns the dependence of the spectrum of
$B_{N,\chi}$ on the cutoff $\chi$.
Let $\mathcal C_\infty\subset[0,1]$ be the limiting Cantor set:
\begin{equation}
  \label{e:C-infty}
\mathcal C_\infty=\bigcap_k \bigcup_{j\in\mathcal C_k}\Big[{j\over M^k},{j+1\over M^k}\Big].
\end{equation}

\begin{figure}
\includegraphics[scale=0.6]{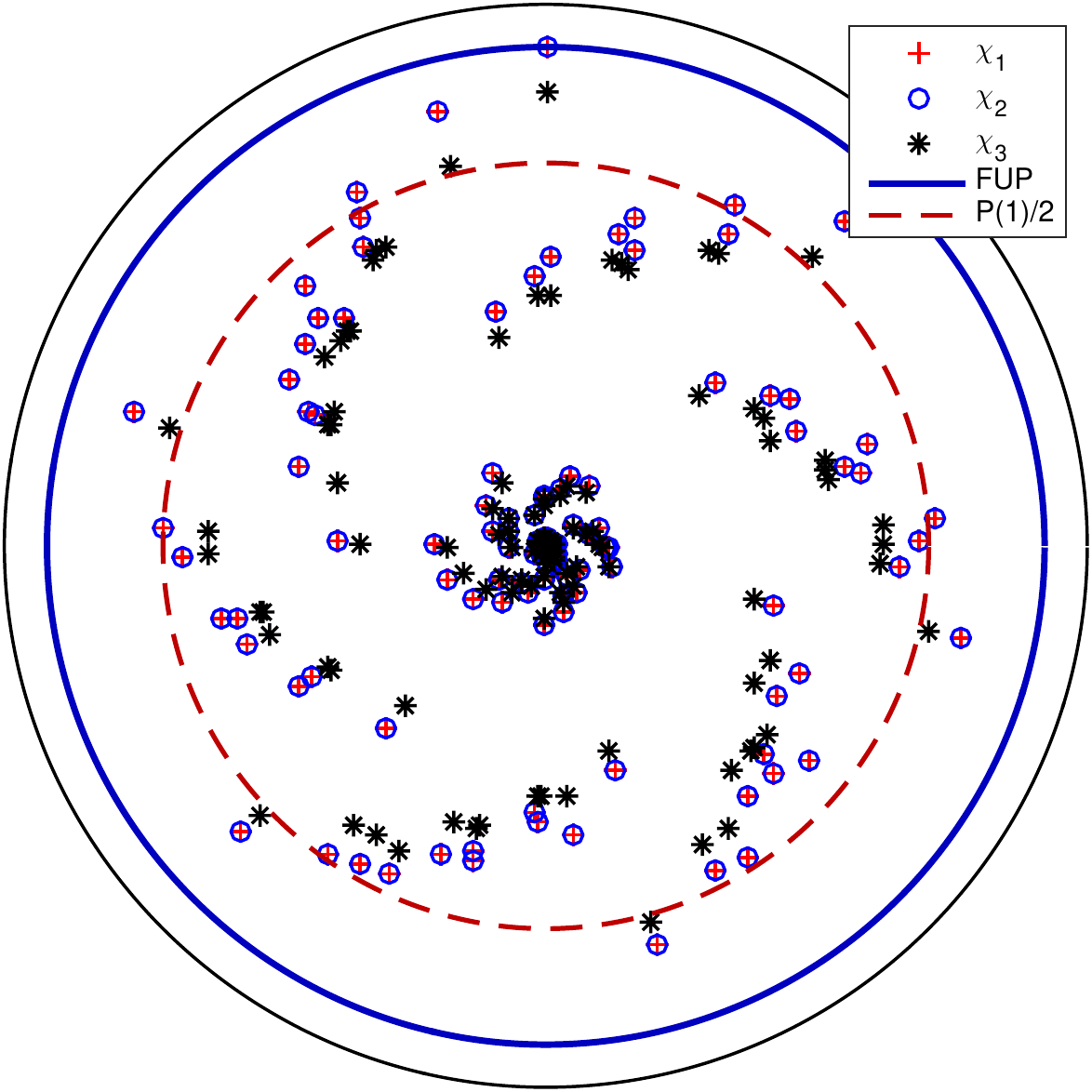}
\quad
\includegraphics[scale=0.6]{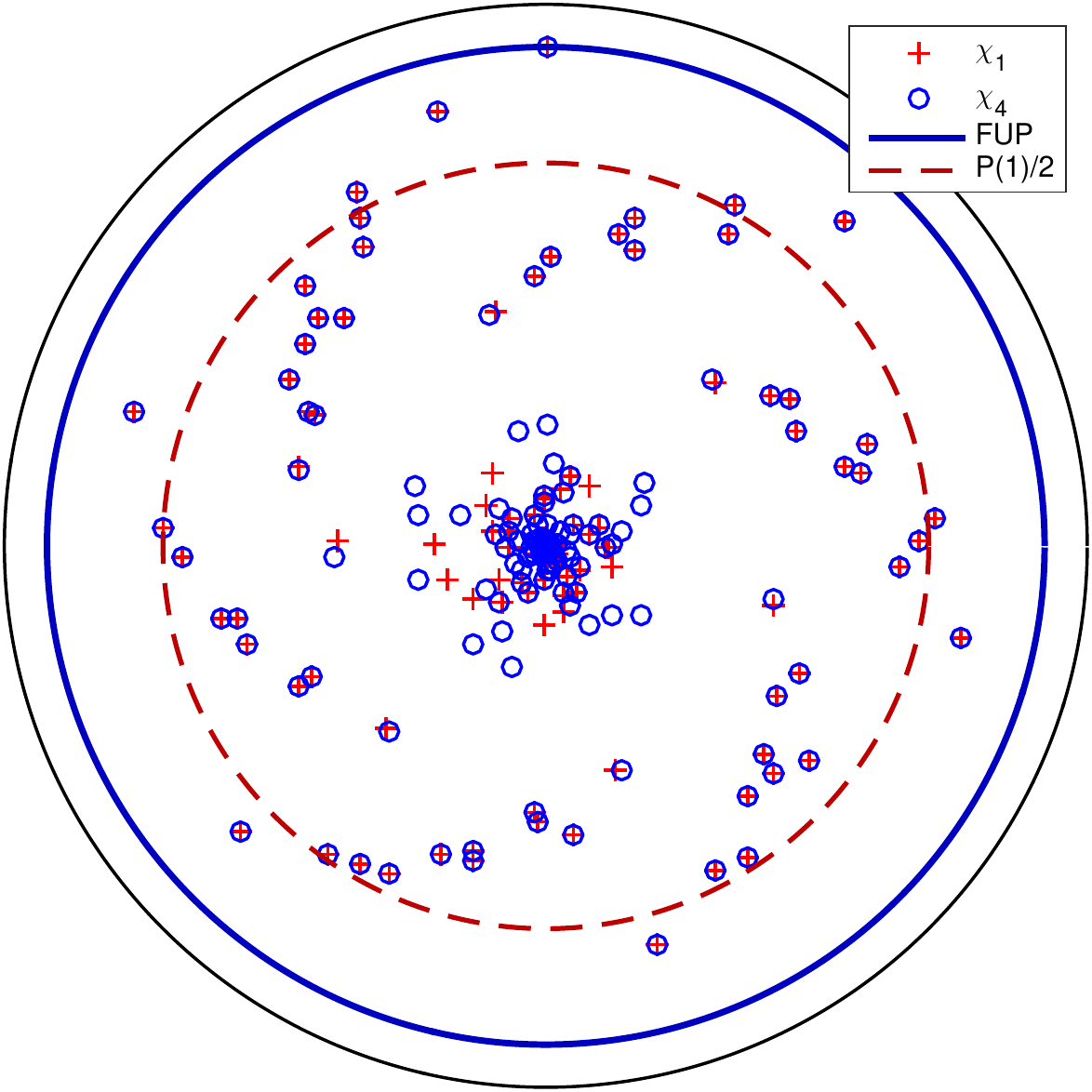}
\caption{An illustration of Theorem~\ref{t:cutoff-dependence},
with $M=4$, $\mathcal A=\{1,2\}$, $k=6$.
We take cutoffs $\chi_1,\chi_2,\chi_3,\chi_4$
such that $\chi_1,\chi_2,\chi_3\in C_0^\infty((0,1);[0,1])$,
$\chi_1=\chi_2=1$ on $\mathcal C_\infty$,
$\chi_3\not\equiv 1$ on $\mathcal C_\infty$, and
$\chi_4\equiv 1$.
On the left, we see that $\chi_1$ and $\chi_2$ produce essentially
the same eigenvalues, but $\chi_3$ does not.
On the right, we see that sufficiently small eigenvalues for $\chi_4$ are
significantly different than those for $\chi_1$; this is due to
the fact that the Fourier transform of the sharp cutoff $\chi_4=\mathbf 1_{[0,1]}$
is not rapidly decaying.}
\label{f:cutoff-dependence}
\end{figure}

\begin{theo}[Dependence on cutoff]
  \label{t:cutoff-dependence}
Assume that $\chi,\chi'\in C_0^\infty((0,1);[0,1])$ satisfy
$$
\chi=\chi'\quad\text{in a neighborhood of }\mathcal C_\infty.
$$
Fix $\nu\geq 0$ and assume that $\lambda$ is an eigenvalue
of $B_{N,\chi}$ satisfying $|\lambda|\geq M^{-\nu}$.
Then there exists an $\mathcal O(N^{-\infty})$ quasimode $v$
for $B_{N,\chi'}$ at $\lambda$, that is
$$
v\in\ell^2_N,\quad
\|v\|_{\ell^2_N}=1,\quad
\|(B_{N,\chi'}-\lambda)v\|_{\ell^2_N}=\mathcal O(N^{-\infty}),
$$
with the constants in $\mathcal O(N^{-\infty})$ depending
only on $\chi,\chi',\nu$.
\end{theo}
Theorem~\ref{t:cutoff-dependence} does not imply that the spectra of $B_{N,\chi}$
and $B_{N,\chi'}$ in annuli are $\mathcal O(N^{-\infty})$ close to each
other, due to possible pseudospectral effects. However, it shows that
for stable features of the spectrum such as eigenvalue free regions
with a polynomial resolvent bound, only the values of $\chi$ near $\mathcal C_\infty$
matter. In particular, if $0,M-1\notin \mathcal A$, then one can choose an arbitrary
$\chi$ such that $\chi=1$ near $\mathcal C_\infty$ and see the same stable properties
of the spectrum. A numerical illustration of Theorem~\ref{t:cutoff-dependence}
is shown on Figure~\ref{f:cutoff-dependence}.

\subsection{Related results}
  \label{s:history}

We now briefly review some previous results in resonance gaps and counting
and explain their relation to the present paper.
For more
information we refer the reader to Nonnenmacher~\cite{Nonnenmacher} for
mathematical results in open quantum chaos, to Novaes~\cite{Novaes} for the physics literature on open quantum maps, and to Zelditch~\cite{ZelditchReview}
for the closely related field of closed quantum chaos.

A popular class of models for open quantum chaos is given by Laplacians
(or more general Schr\"odinger operators)
on noncompact Riemannian manifolds whose geodesic
flow is hyperbolic on the trapped set.
Resonances for these operators appear in long time expansions of solutions to
wave equations. Examples include
exteriors of several convex obstacles in $\mathbb R^n$
and convex co-compact hyperbolic quotients.
We remark that~\cite{NonnenmacherSjZworski1} reduced
the study of resonances for Laplacians to the setting of open quantum maps
quantizing a Poincar\'e map of the geodesic flow.

Essential spectral gaps for Laplacians have been studied
by Patterson~\cite{PattersonGap}, Ikawa~\cite{Ikawa},
Gaspard--Rice~\cite{GaspardRice}, and Nonnenmacher--Zworski~\cite{NonnenmacherZworskiActa}.
These papers establish in various settings a gap of size
$\beta=-P(1/2)$, under the \emph{pressure condition}
$P(1/2)<0$. Here $P(s)$ is the topological pressure of the classical flow,
see~\eqref{e:under-pressure}. The pressure gap was observed in microwave scattering experiments by Barkhofen et al.~\cite{ZworskiPRL}.

Naud~\cite{NaudGap}, Stoyanov~\cite{Stoyanov1,Stoyanov2}, and Petkov--Stoyanov~\cite{PetkovStoyanov}
showed that in some cases such as hyperbolic quotients,
there exists a gap strictly larger than $-P(1/2)$, under the
condition $P(1/2)\leq 0$. These works use the method originally developed
by Dolgopyat~\cite{Dolgopyat}, exploiting in a subtle way the interference
between waves living on different trapped trajectories,
and the size of the improvement is hard to compute from the arguments.
Recently Dyatlov--Zahl~\cite{hgap} have come up with a different
interpretation of the improved gap for hyperbolic quotients
in terms of fractal uncertainty principle, in particular obtaining for $P(1/2)\approx 0$
a gap whose size depends on the additive energy of the limit set similarly to~\S\ref{s:ae-improvements}. 
The approach of~\cite{hgap} is used in the present paper as
well as in the recent papers~\cite{fullgap,regfup} discussed in~\S\ref{s:intro-gaps}.
Improved gaps were also observed
numerically for hyperbolic surfaces with $P(1/2)\approx 0$ by Borthwick and Borthwick--Weich~\cite{BorthwickNum,Borthwick-Weich}.

On the other hand very little is known
about spectral gaps for systems with $P(1/2)>0$ and Theorem~\ref{t:gap-improves}
appears to be the first general result in this case, albeit for a special
class of systems.
Examples of systems with $P(1/2)>0$ and a spectral gap were previously given
in~\cite{NonnenmacherZworskiOQM}, discussed below, and~\cite{hgap}.

Fractal Weyl upper bounds in strips for resonances of Laplacians
and Schr\"odinger operators were first proved in the
analytic category by Sj\"ostrand~\cite{SjostrandFWL} and later
in various smooth settings by Guillop\'e--Lin--Zworski~\cite{GLZ},
Zworski~\cite{ZwInventiones},
Sj\"ostrand--Zworski~\cite{SjostrandZworskiFWL},
Nonnenmacher--Sj\"ostrand--Zworski~\cite{NonnenmacherSjZworski1,NonnenmacherSjZworski2},
and Datchev--Dyatlov~\cite{fwl}. In terms of~\eqref{e:N-k}
these bounds give $\mathcal N_k(\nu)=\mathcal O(N^\delta)$,
with $\delta$ related to the Minkowski dimension of the trapped set.
Compared to these works, our bound~\eqref{e:weyl} loses
an arbitrarily small power of $N$.
The sharpness of the exponent $\delta$ has been
investigated experimentally by Potzuweit et al.~\cite{ZworskiPRE}
and numerically by Lu--Sridhar--Zworski~\cite{LSZ},
Borthwick~\cite{BorthwickNum},
Borthwick--Weich~\cite{Borthwick-Weich},
and Borthwick--Dyatlov--Weich~\cite[Appendix]{ifwl}.

Concentration of resonances near the decay rate
$P(1)/2$ has been observed numerically in~\cite{LSZ}
and experimentally in~\cite{ZworskiPRL,ZworskiPRE}.
Jakobson--Naud~\cite{JakobsonNaud} conjectured that for hyperbolic surfaces,
there is a gap of any size less than $P(1)/2$. While the
numerical investigations of~\cite{BorthwickNum,Borthwick-Weich,ifwl}
do not seem to support this conjecture for general systems,
in~\S\ref{s:special-alphabets} we provide examples
of systems which do satisfy the conjecture.
Naud~\cite{NaudCount} showed an improved Weyl bound
$\mathcal N_k(\nu)=\mathcal O(N^{m(\nu)})$ for hyperbolic surfaces,
for some $m(\nu)<\delta$ when $\nu<P(1)/2$, and Dyatlov~\cite{ifwl}
proved the bound~\eqref{e:weyl} for hyperbolic quotients. 

Quantum baker's maps have attracted a lot of attention in physics and mathematics literature.
Their study was initiated in the closed setting
$\mathcal A=\{0,\dots,M-1\}$
by Bal\'azs--Voros~\cite{Balasz-Voros},
Saraceno~\cite{Saraceno}, and
Saraceno--Voros~\cite{SaracenoVoros};
see the introduction to~\cite{DNW}
for an overview of more recent results. 
In the open setting,
Keating et al.~\cite{Keating2} observed numerically concentration of
eigenfunctions in position and momentum consistent with Proposition~\ref{l:key-eigenvalues}.
This concentration was proved for the Walsh quantization by
Keating et al.~\cite{Keating1}; 
see also the work of
Nonnenmacher--Rubin~\cite{NonnenmacherRubin} on semiclassical defect
measures.
Novaes et al.~\cite{Novaes2} and
Carlo--Benito--Borondo~\cite{Carlo} introduced an approximation
for eigenfunctions using short periodic orbits.
Experimental realizations for open baker's maps
have been proposed by Brun--Schack~\cite{Brun-Schack}
and Hannay--Keating--Ozorio de Almeida~\cite{HKO}.
Recently ideas in open quantum chaos have been applied to
analysis of computer networks, see
Ermann--Frahm--Shepelyansky~\cite{GoogleMatrix}.

The closest to the present
paper is the work of Nonnenmacher--Zworski~\cite{NonnenmacherZworskiOQM2,NonnenmacherZworskiOQM} who studied open quantum baker's maps,
in particular the Walsh quantization for the cases $M=3,4$, $\mathcal A=\{0,2\}$, $\chi\equiv 1$
in the notation of our paper (as well as obtaining numerical results for other maps
and quantizations). The Walsh quantization is obtained by replacing
$\mathcal F_N,\mathcal F_{N/M}$ in~\eqref{e:B-N-def} by the Walsh Fourier
transform, which is the Fourier transform on the group $(\mathbb Z_M)^k$.
Eigenvalues of Walsh quantizations are computed
explicitly in~\cite[\S5]{NonnenmacherZworskiOQM},
which proves fractal Weyl law and shows concentration of resonances around
decay rate $P(1)/2$. Moreover~\cite{NonnenmacherZworskiOQM}
shows that there is a spectral gap for $M=3,\mathcal A=\{0,2\}$
but not for $M=4,\mathcal A=\{0,2\}$. The latter does not contradict Theorem~\ref{t:gap-improves} because a different quantization is used, and Cantor sets
do not always satisfy the uncertainty principle under the 
Walsh Fourier transform.

\section{Open quantum maps}
  \label{s:oqm}

In this section, we study the open quantum map $B_N$. The main result is Proposition~\ref{l:fup-reduction},
giving a bound on the spectral radius of $B_N$ in terms of the
fractal uncertainty principle exponent $\beta$ defined in~\eqref{e:beta-fup}.

\subsection{Definition and basic properties}
  \label{s:setup}

For $N\in\mathbb N$, consider the abelian group
$$
\mathbb Z_N=\mathbb Z/N\mathbb Z\simeq \{0,\dots,N-1\}
$$
and the space $\ell^2_N$ of functions $u:\mathbb Z_N\to \mathbb C$ with the
Hilbert norm
$$
\|u\|_{\ell^2_N}^2=\sum_{j=0}^{N-1}|u(j)|^2.
$$
Define the unitary Fourier transform
$$
\mathcal F_N:\ell^2_N\to\ell^2_N,\quad
\mathcal F_N u(j)={1\over \sqrt{N}}\sum_{\ell=0}^{N-1}\exp\Big(-{2\pi i j\ell\over N}\Big)u(\ell).
$$
For a cutoff function $\chi\in C_0^\infty((0,1);[0,1])$, define its discretization
$\chi_N\in \ell^2_N$ by
\begin{equation}
  \label{e:chi-N}
\chi_N(j)=\chi\Big({j\over N}\Big),\quad
j\in \{0,\dots,N-1\}.
\end{equation}
We also denote by $\chi_N$ the corresponding multiplication operator on $\ell^2_N$.

Fix $(M,\mathcal A,\chi)$ as in~\eqref{e:oqm-triple} and take $k\in\mathbb N$;
put $N:=M^k$. Then the open quantum map
$B_N:\ell^2_N\to \ell^2_N$ defined in~\eqref{e:B-N-def} can be written as follows:
if $\Pi_a:\ell^2_N\to \ell^2_{N/M}$, $a\in\{0,\dots,M-1\}$, is the projection map
defined by 
\begin{equation}
  \label{e:Pi-a}
\Pi_a u(j)=u\Big(j+a{N\over M}\Big),\quad
u\in \ell^2_N,\quad
j\in \Big\{0,\dots,{N\over M}-1\Big\},
\end{equation}
then 
$$
B_N=\sum_{a\in\mathcal A}B_N^a,\quad
B_N^a:=\mathcal F_N^* \Pi_a^*\, \chi_{N/M}\,\mathcal F_{N/M}\,\chi_{N/M}\,\Pi_a.
$$
We compute for each $u\in \ell^2_N$, $j\in \{0,\dots,N-1\}$, and $a\in\mathcal A$,
\begin{equation}
  \label{e:B-N-matrix}
B_N^a u(j)={\sqrt M\over N}\sum_{m,\ell=0}^{{N\over M}-1}
e^{2\pi i\left({(j-M\ell)m\over N}+{ja\over M}\right)}\chi\Big({mM\over N}\Big)
\chi\Big({\ell M\over N}\Big)u\Big(\ell+a{N\over M}\Big).
\end{equation}
The continuous analogue of the transformation $B_N$ is obtained as follows:
put
$$
x:={j\over N},\quad
y:={\ell\over N}+{a\over M},\quad
\theta:={m\over N},\quad
h:={1\over 2\pi N}
$$
and replace the sums over $\ell, m$ by integrals over $y,\theta$ with the corresponding Jacobian factors.
Then the analogue of~\eqref{e:B-N-matrix}
is given by the operator $\mathcal U^a_h$ on $L^2(\mathbb R)$ defined as follows:
\begin{equation}
  \label{e:U-a}
\mathcal U^a_h v(x)={\sqrt{M}\over 2\pi h}\int_{\mathbb R^2}
e^{{i\over h}\left((x+a-My)\theta+xa/M\right)}\chi(M\theta)\chi(My-a)u(y)\,dyd\theta.
\end{equation}
The sum
\begin{equation}
  \label{e:U-h}
\mathcal U_h:=\sum_{a\in\mathcal A}\mathcal U^a_h
\end{equation}
is a semiclassical Fourier integral operator
(see for instance~\cite[\S3.2]{nhp}) associated to the canonical relation
$\varkappa_{M,\mathcal A}$ defined in~\eqref{e:kappa-def}, with
principal symbol equal to $\chi(x)\chi(\eta)$ with appropriate normalization.
Because of the analogy with the continuous case,
we may think of $B_N$ as a discrete Fourier integral operator quantizing the relation
$\varkappa_{M,\mathcal A}$. A rigorous justification of this analogy
can be found in the papers of Degli Esposti--Nonnenmacher--Winn~\cite{DNW}
and Nonnenmacher--Zworski~\cite{NonnenmacherZworskiOQM},
with heuristic arguments appearing in Bal\'azs--Voros~\cite{Balasz-Voros}
and Saraceno--Voros~\cite{SaracenoVoros}.

We next consider the distance function on $[0,1]$ with $0$ and $1$ identified
with each other: for $x,y\in[0,1]$,
\begin{equation*}
d(x,y)=\min_{k=-1,0,1}|x-y-k|=\min\{|x-y|,1-|x-y|\}.
\end{equation*}
In particular, $d(x,0)=\min\{x,1-x\}$ is the usual distance from $x$ to the closest integer.
For $x\in [0,1]$ and $V,W\subset [0,1]$, we put
\begin{equation}
  \label{e:distance-set}
d(x,V):=\inf_{y\in V}d(x,y),\quad
d(V,W):=\inf_{y\in V,\,z\in W}d(y,z).
\end{equation}
We define the expanding map
\begin{equation}
  \label{e:Phi}
\begin{gathered}
\Phi=\Phi_{M,\mathcal{A}}:\bigsqcup_{a\in\mathcal A}\left(\frac{a}{M},\frac{a+1}{M}\right)\to (0,1);\\
\Phi(x)=Mx-a,\quad
x\in \left(\frac{a}{M},\frac{a+1}{M}\right).
\end{gathered}
\end{equation}
In other words, $\Phi_{M,\mathcal{A}}$ is the action of the relation $\varkappa_{M,\mathcal{A}}$ on the space variable $x$. We establish the following fact regarding the interaction between the map $\Phi$ and the distance function~$d$:
\begin{lemm} Assume that $x\in[0,1]$ and $y$ is in the domain of $\Phi$. Then
\begin{equation}
\label{e:dist-phi-1}
\min\big\{d(\Phi(y),0),\,M\cdot d(y,\Phi^{-1}(x))\big\}\leq d(x,\Phi(y)).
\end{equation}
\end{lemm}
\begin{proof} We have the following equivalent expression for $d(x,y)$:
\begin{equation*}
d(x,y)=\min\{|x-y|, d(x,0)+d(y,0)\}.
\end{equation*}
Therefore
\begin{equation*}
d(x,\Phi(y))=\min\big\{|x-\Phi(y)|, d(x,0)+d(\Phi(y),0)\big\}.
\end{equation*}
Let $\tilde x$ be the unique element in $\Phi^{-1}(x)$ that
is in the same interval $({a\over M},{a+1\over M})$ as $y$.
Then
$$
|x-\Phi(y)|=M|\tilde x-y|\geq M d(y,\Phi^{-1}(x)),
$$
finishing the proof.
\end{proof}
We also use the following result on rapid decay for oscillating sums,
which is the discrete analog of rapid decay of Fourier series of smooth functions:
\begin{lemm}[Method of nonstationary phase]
\label{l:non-st}
Assume that $a\in\mathbb{Z}_N$ and
\begin{equation}
  \label{e:mnp-0}
d\Big({a\over N},0\Big)\geq cN^{-\rho}\quad\text{for some constants }
c>0,\ \rho\in [0,1).
\end{equation}
Then for all $\chi\in C_0^\infty((0,1))$, we have
\begin{equation}
  \label{e:mnp}
\sum_{m=0}^{N-1}\exp\Big({2\pi iam\over N}\Big)\chi\Big({m\over N}\Big)=\mathcal{O}(N^{-\infty})
\end{equation}
where the constants in $\mathcal{O}(N^{-\infty})$ only depend on $c$, $\rho$, and $\chi$.
\end{lemm}
\begin{proof}
Applying the Poisson summation formula to the function
$e^{2\pi i ax/N}\chi(x/N)$, we write the left-hand side of~\eqref{e:mnp} as
$$
N\sum_{\ell\in\mathbb Z} \hat\chi(N\ell-a)
$$
where $\hat\chi$ is the Fourier transform of $\chi$:
$$
\hat\chi(\xi)=\int_{\mathbb R}\exp(-2\pi i x\xi)\chi(x)\,dx.
$$
Since $\hat\chi$ is rapidly decreasing, \eqref{e:mnp} follows;
here we use~\eqref{e:mnp-0} to handle the cases $\ell=0,1$.
\end{proof}

\subsection{Propagation of singularities}

Following~\eqref{e:chi-N}, for each $\varphi:[0,1]\to\mathbb R$, we define
$$
\varphi_N\in \ell^2_N,\quad
\varphi_N(j)=\varphi\Big({j\over N}\Big).
$$
The function $\varphi_N$ defines a multiplication operator on $\ell^2_N$,
still denoted $\varphi_N$.
We also use the corresponding Fourier multplier
$$
\varphi_{N}^{\mathcal F}=\mathcal F_N^*\varphi_N\mathcal F_N.
$$
The following theorems are analogues of propagation of semiclassical singularities
(that is, regions where $\varphi_N$ is not $\mathcal O(N^{-\infty})$)
in position and frequency space under quantum evolution.
A stronger statement (not needed here) is \emph{Egorov's theorem},
which describes symbols of propagated quantum observables. In the context
of quantum baker's maps it was proved in~\cite[Theorems~12,13]{DNW}
and~\cite[Proposition~4.15]{NonnenmacherZworskiOQM}.

We start with the case when we apply the open quantum map only once.
We use the map $\Phi$ defined in~\eqref{e:Phi} and the cutoff function $\chi$
which is part of~\eqref{e:oqm-triple}.
Note that due to the simple nature of the map $B_N$ we do not need
to impose any smoothness assumptions on the classical observables $\varphi,\psi$
below.
\begin{prop}[Propagation of singularities]
\label{l:egorov}
Assume that $\varphi,\psi:[0,1]\to[0,1]$ and for some $c>0$ and $0\leq\rho<1$,
\begin{equation}
\label{e:dist-eg-phi}
d\big(\Phi(\supp\psi\cap\Phi^{-1}(\supp\chi)),\supp\varphi)\geq cN^{-\rho}.
\end{equation}
Then
\begin{align}
\label{e:eg-space}
\|\varphi_N B_N \psi_N\|_{\ell^2_N\to\ell^2_N}&=\mathcal O(N^{-\infty}),\\
\label{e:eg-fourier}
\|\psi_N^{\mathcal F}B_N \varphi_N^{\mathcal F}\|_{\ell^2_N\to \ell^2_N}&=\mathcal O(N^{-\infty}),
\end{align}
where the constants in $\mathcal O(N^{-\infty})$ depend only on $c,\rho,\chi$. In particular, \eqref{e:eg-space} and \eqref{e:eg-fourier} hold when%
\footnote{The fact that~\eqref{e:dist-eg} implies~\eqref{e:dist-eg-phi} for a different
constant $c$ is not directly used in this paper, however its proof is a good
introduction to the proof of Proposition~\ref{l:egorov-long} below.}
\begin{equation}
\label{e:dist-eg}
d\big(\supp\psi,\Phi^{-1}(\supp\varphi)\big)\geq cN^{-\rho}.
\end{equation}
\end{prop}
\Remark The continuous analog of the operator $\varphi_N$ is
the multiplication operator~$\varphi(y)$, and the continuous analog
of $\varphi_N^{\mathcal F}$ is the Fourier multiplier $\varphi({h\over i}\partial_y)$.
Both of these are semiclassical pseudodifferential operators (see for instance~\cite[\S4.1]{e-z}),
with symbols given by $\tilde \varphi(y,\eta)=\varphi(y)$ and $\tilde \varphi^{\mathcal F}(y,\eta)=\varphi(\eta)$.
The condition~\eqref{e:dist-eg-phi} is equivalent to each of the following conditions
featuring the relation $\varkappa=\varkappa_{M,\mathcal A}$ defined in~\eqref{e:kappa-def}
and the symbol $\widetilde \chi(y,\eta)=\chi(\Phi(y))\chi(\eta)$:
\begin{align}
  \label{e:jack-1}
d\big(\varkappa(\supp\tilde \psi\cap \supp\widetilde\chi),\supp\tilde \varphi\big)&\geq cN^{-\rho},\\
  \label{e:jack-2}
d\big(\varkappa^{-1}(\supp\tilde \psi^{\mathcal F}\cap\varkappa(\supp\widetilde\chi)),\supp\tilde\varphi^{\mathcal F})&\geq cN^{-\rho}.
\end{align}
The continuous analogues of~\eqref{e:eg-space}, \eqref{e:eg-fourier} are expressed via
the operator from~\eqref{e:U-h}:
$$
\varphi\, \mathcal U_h\psi,\ \psi\Big({h\over i}\partial_y\Big)\mathcal U_h\,\varphi\Big({h\over i}\partial_y\Big)=\mathcal O(h^\infty)_{L^2(\mathbb R)\to L^2(\mathbb R)}
$$
In the case when $\varphi,\psi$ are smooth and $h$-independent and $\rho=0$, the latter two statements follow from~\eqref{e:jack-1},
\eqref{e:jack-2}, and the wavefront set statement
(see for instance~\cite[\S3.2]{nhp})
$$
\WFh(\mathcal U_h)\subset \{(x,\xi;y,\eta)\mid (x,\xi)=\varkappa(y,\eta),\quad
(y,\eta)\in\supp\widetilde\chi\}.
$$
\begin{proof}[Proof of Proposition~\ref{l:egorov}]
By~\eqref{e:B-N-matrix}, we have for all $u\in \ell^2_N$, $j\in\{0,\dots,N-1\}$,
$$
\begin{aligned}
\varphi_N B_N\psi_Nu(j)&=\sum_{a\in\mathcal A}\sum_{\ell=0}^{{N/ M}-1}
A^a_{j\ell} \,u\Big(\ell+a{N\over M}\Big),\\
A^a_{j\ell}&={\sqrt{M}\over N}\varphi\Big({j\over N}\Big)\exp\Big({2\pi i aj\over M}\Big)\chi\Big({\ell M\over N}\Big)
\psi\Big({\ell\over N}+{a\over M}\Big)\widetilde A_{j\ell},\\
\widetilde A_{j\ell}&=\sum_{m=0}^{N/M-1}
\exp\Big({2\pi i m(j-\ell M)\over N}\Big)\chi\Big({mM\over N}\Big).
\end{aligned}
$$
We write
$$
\widetilde A_{j\ell}=
\sum_{m=0}^{N-1}\exp\Big({2\pi ibm\over N}\Big)\chi_1\Big({m\over N}\Big),\quad
b:=j-\ell M,\quad
\chi_1(x)=\chi(Mx).
$$
We have $A^a_{j\ell}=0$ unless
\begin{equation}
  \label{e:supporter}
{j\over N}\in\supp\varphi,\quad
{\ell\over N}+{a\over M}\in\supp\psi,\quad
{\ell M\over N}=\Phi\Big({\ell\over N}+{a\over M}\Big)\in\supp \chi.
\end{equation}
By~\eqref{e:dist-eg-phi}, we see that~\eqref{e:supporter} implies
$$
d\Big({b\over N},0\Big)=d\Big({j\over N},{\ell M\over N}\Big)\geq cN^{-\rho}.
$$
Applying Lemma~\ref{l:non-st}, we see that
$$
\max_{a,j,\ell}|A^a_{j\ell}|=\mathcal O(N^{-\infty})
$$
and~\eqref{e:eg-space} follows.

To show~\eqref{e:eg-fourier}, we notice that $\mathcal F_N^*=\overline{\mathcal F_N}$
and thus
$$
\psi_N^\mathcal{F}B_N\varphi_N^\mathcal{F}=\mathcal{F}_N^\ast
(\overline{\varphi_N B_N\psi_N})^*\mathcal{F}_N
$$
where for any operator $A:\ell^2_N\to\ell^2_N$ its
complex conjugate $\overline A:\ell^2_N\to\ell^2_N$ is defined by $\overline A\,\bar u=\overline{Au}$,
$u\in\ell^2_N$.
Since both $\mathcal{F}_N$ and $\mathcal{F}_N^\ast$ are unitary, \eqref{e:eg-fourier} follows from \eqref{e:eg-space} and the facts that both operations of the complex conjugate and the adjoint preserve the operator norm.

Finally, we show that~\eqref{e:dist-eg} implies \eqref{e:dist-eg-phi} with a different $c$.
Indeed, assume \eqref{e:dist-eg} holds.
Then for any $x\in\supp\varphi$ and $y\in\supp\psi\cap\Phi^{-1}(\supp\chi)$, by \eqref{e:dist-phi-1}, we see that
$$
\begin{aligned}
d(x,\Phi(y))&\geq\min\{d(\Phi(y),0),M\cdot d(y,\Phi^{-1}(x))\}\\
&\geq\min\{d(\supp\chi,0),cMN^{-\rho}\}\geq c'N^{-\rho}
\end{aligned}
$$
which finishes the proof.
\end{proof}
Now we turn to the case when we iterate the open quantum map up to (almost) twice the Ehrenfest time%
\footnote{The map $\varkappa_{M,\mathcal A}$ has expansion rate equal to $M$,
and the semiclassical parameter is $h=(2\pi M^k)^{-1}$, therefore
propagation until the Ehrenfest time corresponds to taking $B_N$ to the power $k/2$.}
$k$.
\begin{prop}[Propagation of singularities for long times]
\label{l:egorov-long}
Assume that $\varphi,\psi:[0,1]\to[0,1]$ and for some $c>0$, $0\leq\rho<1$,
and an integer $\tilde k\in [1,k]$,
\begin{equation}
\label{e:dist-eg-l}
d\big(\supp\psi,\Phi^{-\tilde k}(\supp\varphi)\big)\geq cN^{-\rho}.
\end{equation}
Then
\begin{align}
\label{e:eg-space-l}
\|\varphi_N (B_N)^{\tilde k} \psi_N\|_{\ell^2_N\to\ell^2_N}&=\mathcal O(N^{-\infty}),\\
\label{e:eg-fourier-l}
\|\psi_N^{\mathcal F}(B_N)^{\tilde k} \varphi_N^{\mathcal F}\|_{\ell^2_N\to \ell^2_N}&=\mathcal O(N^{-\infty}),
\end{align}
where the constants in $\mathcal O(N^{-\infty})$ depend only on $c,\rho,\chi$.
\end{prop}
\begin{proof}
It suffices to construct a sequence of functions (see Figure~\ref{f:nested-cutoffs})
$$
\varphi^{(j)}:[0,1]\to [0,1],\quad \psi^{(j)}:=1-\varphi^{(j)},\quad
j=0,\dots,\tilde k,
$$
such that for some $c'>0$ depending only on $c,\chi$,
\begin{gather}
  \label{e:seqf-1}
\psi^{(0)}\varphi=0,\\
  \label{e:seqf-2}
\psi^{(\tilde k)}\psi=\psi,\\
  \label{e:seqf-3}
d\big(\Phi(\supp\psi^{(j+1)}\cap \Phi^{-1}(\supp\chi)),\supp\varphi^{(j)}\big)\geq c'N^{-\rho},\quad
j=0,\dots,\tilde k-1.
\end{gather}
Indeed, inserting $\indic=\varphi^{(j)}_N+\psi^{(j)}_N$ after the $j$-th factor $B_N$ and using~\eqref{e:seqf-1},
\eqref{e:seqf-2}, we write
$$
\begin{gathered}
\varphi_N(B_N)^{\tilde{k}}\psi_N
=\sum_{j=0}^{\tilde{k}-1}A'_{j}\big(\varphi^{(j)}_NB_N\psi^{(j+1)}_N\big)A''_{j},\\
A'_{j}=\varphi_N(B_N)^j,\quad
A''_{j}=(B_N\psi_N^{(j+2)})\cdots(B_N\psi_N^{(\tilde{k})})\psi_N.
\end{gathered}
$$
Clearly, 
\begin{equation*}
\|A'_{j}\|_{\ell^2_N\to\ell^2_N},\|A''_{j}\|_{\ell^2_N\to\ell^2_N}\leq 1.
\end{equation*}
Applying~\eqref{e:eg-space} and using~\eqref{e:seqf-3}, we get
\begin{equation*}
\|\varphi^{(j)}_NB_N\psi^{(j+1)}_N\|_{\ell^2_N\to\ell^2_N}=\mathcal{O}(N^{-\infty}).
\end{equation*}
This concludes the proof of \eqref{e:eg-space-l} as $\tilde{k}\leq k=O(\log N)$. The other estimate \eqref{e:eg-fourier-l} follows from \eqref{e:eg-fourier} by the same argument.

\begin{figure}
\includegraphics{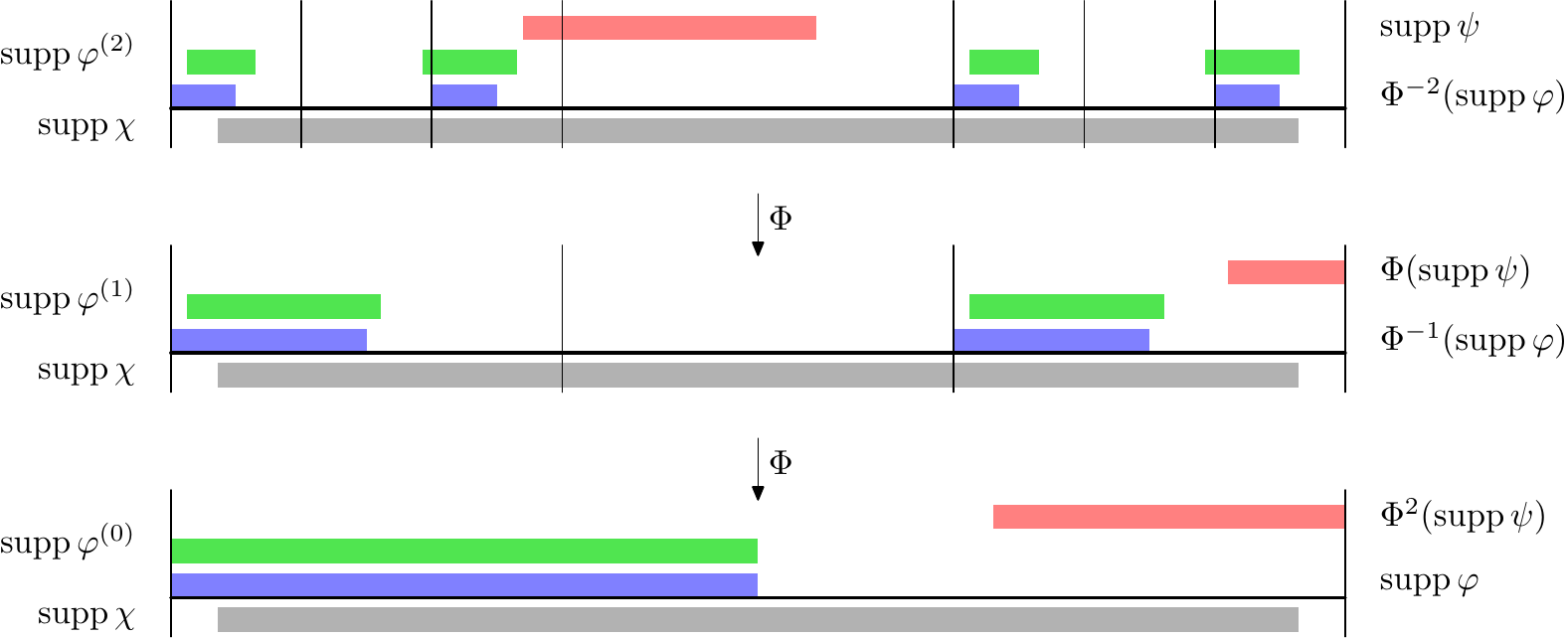}
\caption{An illustration of the proof of Proposition~\ref{l:egorov-long} for
$M=3$, $\mathcal A=\{0,2\}$, $\tilde k=2$. Note that $d(\Phi^2(\supp\psi),\supp\varphi)=0$,
explaining why we need to use the cutoff function $\chi$ in~\eqref{e:seqf-3}.}
\label{f:nested-cutoffs}
\end{figure}

We now construct the functions $\varphi^{(j)}$. Fix $c'>0$ depending only on $c,\chi$, to be chosen
in~\eqref{e:c-cond-1}, \eqref{e:c-cond-2} below.
Define
$$
\varphi^{(0)}(x)=\begin{cases}
1,& x\in\supp\varphi,\\
0,& \text{otherwise}.
\end{cases}
$$
Note that~\eqref{e:seqf-1} holds.
We next define inductively for $j=1,\dots,\tilde k$,
$$
\varphi^{(j)}(x)=\begin{cases}
1,& x\in\Phi^{-1}(\supp\chi)\quad \text{and} \quad d(\Phi(x),\supp\varphi^{(j-1)})\leq c'N^{-\rho},\\
0,& \text{otherwise}.
\end{cases}
$$
Then~\eqref{e:seqf-3} holds, so it remains to prove~\eqref{e:seqf-2}. The latter follows
from the fact that
\begin{equation}
  \label{e:seqf-4}
\supp\varphi^{(\tilde k)}\cap \supp\psi=\emptyset.
\end{equation}
To show~\eqref{e:seqf-4}, we note that
any point in $\supp\varphi^{(\tilde k)}$ is equal to
$x_{\tilde k}$ for some sequence of points
$x_0,\dots,x_{\tilde k}\in [0,1]$ such that
$$
x_0\in\supp\varphi;\quad
x_j\in\Phi^{-1}(\supp\chi),\quad
d(\Phi(x_j),x_{j-1})\leq c'N^{-\rho},\quad
j=1,\dots,\tilde k.
$$
By~\eqref{e:dist-eg-l} it then suffices to prove that
\begin{equation}
  \label{e:seqf-5}
d(x_{\tilde k},\Phi^{-\tilde k}(x_0))<cN^{-\rho}.
\end{equation}
We have for each $j=1,\dots,\tilde k$,
$$
\begin{aligned}
\min\big\{d(\supp\chi,0),M\cdot d(x_j,\Phi^{-j}(x_0))\big\}&\leq
\min\big\{d(\Phi(x_j),0),M\cdot d(x_j,\Phi^{-j}(x_0))\big\}\\
&\leq d(\Phi(x_j),\Phi^{-(j-1)}(x_0))\\
&\leq d(x_{j-1},\Phi^{-(j-1)}(x_0))+c'N^{-\rho}
\end{aligned}
$$
where on the second line we use~\eqref{e:dist-phi-1}. By induction on $j$,
we see that if $c'$ is small enough so that
\begin{equation}
  \label{e:c-cond-1}
{c'M\over M-1}\leq d(\supp\chi,0)
\end{equation}
then for all $j=0,\dots,\tilde k$ we have
$$
d(x_j,\Phi^{-j}(x_0))\leq c'N^{-\rho}\cdot {1-M^{-j}\over M-1}.
$$
This implies~\eqref{e:seqf-5} as long as
\begin{equation}
  \label{e:c-cond-2}
{c'\over M-1}\leq c,
\end{equation}
finishing the proof of the proposition.
\end{proof}

\subsection{Localization of eigenfunctions and reduction to fractal uncertainty principle}
  \label{s:localization}

Fix $\rho\in (0,1)$ and put
\begin{equation}
  \label{e:tilde-k}
\tilde k:=\lceil\rho k\rceil\in \{1,\dots,k\}.
\end{equation}
Define
$$
\mathcal X_{\rho}:=\{x\in [0,1]\colon d(x,\Phi^{-\tilde k}([0,1]))\leq N^{-\rho}\}.
$$
Using the Cantor set $\mathcal C_k$ defined in~\eqref{e:C-k}, we also put
\begin{equation}
  \label{e:X-rho}
X_{\rho}:=\bigcup \{\mathcal C_k+m\colon m\in \mathbb Z,\ |m|\leq 2N^{1-\rho}\}
\subset\mathbb Z_N
\end{equation}
where addition is carried in the group $\mathbb Z_N$. Then
\begin{equation}
  \label{e:kontainer}
\ell\in \{0,\dots,N-1\},\quad
{\ell\over N}\in\mathcal X_{\rho}\quad\Longrightarrow\quad
\ell\in X_{\rho}.
\end{equation}
Indeed, we have
$$
\Phi^{-\tilde k}([0,1])=\bigcup_{j'\in\mathcal C_{\tilde k}}
\Big({j'\over M^{\tilde k}},{j'+1\over M^{\tilde k}}\Big)
\subset\bigcup_{j\in\mathcal C_k}\Big({j-M^{k-\tilde k}\over N},
{j+M^{k-\tilde k}\over N} \Big)
$$
and~\eqref{e:kontainer} follows since $M^{k-\tilde k}\leq N^{1-\rho}$.

Taking
$\varphi\equiv 1$ and $\psi:=1-\mathbf 1_{\mathcal X_{\rho}}$ in Proposition~\ref{l:egorov-long},
we obtain
\begin{align}
  \label{e:concentration-1}
(B_N)^{\tilde k}&=(B_N)^{\tilde k}\indic_{X_{\rho}}+\mathcal O(N^{-\infty})_{\ell^2_N\to\ell^2_N},\\
  \label{e:concentration-2}
(B_N)^{\tilde k}&=\mathcal F_N^*\indic_{X_{\rho}}\mathcal F_N(B_N)^{\tilde k}+\mathcal O(N^{-\infty})_{\ell^2_N\to\ell^2_N},
\end{align}
where the constants in $\mathcal O(N^{-\infty})$ depend only on $\rho,\chi$.
\begin{figure}
\includegraphics{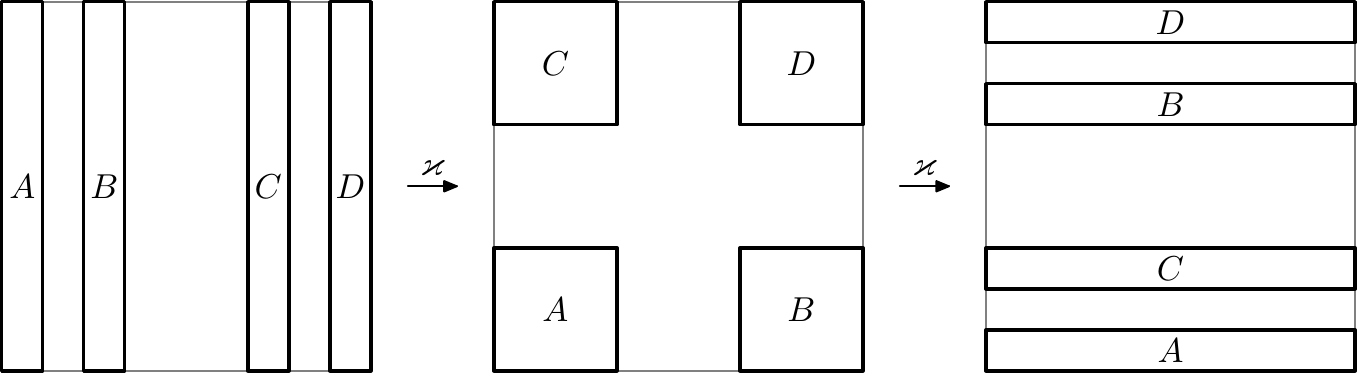}
\caption{The domain and range of the relation $\varkappa^2$
for $M=3$, $\mathcal A=\{0,2\}$.
See also Figure~\ref{f:basic-baker}.}
\label{f:longtime}
\end{figure}

The statements~\eqref{e:concentration-1}, \eqref{e:concentration-2} can be interpreted
in terms of the canonical relation $\varkappa=\varkappa_{M,\mathcal A}$ from~\eqref{e:kappa-def}
and the Fourier integral operator $\mathcal U_h$ from~\eqref{e:U-h}.
Indeed, define the \emph{incoming/outgoing tails}
and the \emph{trapped set}
\begin{equation}
  \label{e:trapped-set}
\Gamma_\pm = \bigcap_{\pm r\geq 0} \varkappa^{\,r}\big((0,1)^2\big),\quad
\Gamma = \Gamma_+\cap\Gamma_-
\end{equation}
They can be expressed in terms of the Cantor set $\mathcal C_\infty$ defined in~\eqref{e:C-infty}:
$$
\Gamma_+=\{x\in (0,1),\ \xi\in\mathcal C_\infty\},\quad
\Gamma_-=\{\xi\in (0,1),\ x\in\mathcal C_\infty\}.
$$
See Figure~\ref{f:longtime}. Then~\eqref{e:concentration-2} corresponds
to the statement that functions in the range of~$\mathcal U_h^{\tilde k}$
(that is, outgoing functions) are microlocalized $h^\rho$ close to $\Gamma_+$.
Similarly~\eqref{e:concentration-1} corresponds to the statement that
functions in the range of the adjoint of $\mathcal U_h^{\tilde k}$
(that is, incoming functions) are microlocalized $h^\rho$ close to $\Gamma_-$.

Applied to eigenfunctions of $B_N$, \eqref{e:concentration-1} and~\eqref{e:concentration-2} give the following
statement. (See~\cite[Lemma~4.6]{hgap} for its analog in the continuous setting of hyperbolic surfaces.)
\begin{prop}
  \label{l:key-eigenvalues}
Fix $\nu>0,\rho\in (0,1)$ and assume that for some $k\in\mathbb N,\lambda\in\mathbb C,u\in\ell^2_N$,
$$
B_N u=\lambda u,\quad
|\lambda|\geq M^{-\nu}.
$$
Then, with $X_\rho$ defined in~\eqref{e:X-rho},
\begin{gather}
  \label{e:concentration-3}
\|u\|_{\ell^2_N}\leq M^\nu|\lambda|^{-\rho k}\,\|\indic_{X_{\rho}}u\|_{\ell^2_N}+\mathcal O(N^{-\infty})\|u\|_{\ell^2_N},\\
  \label{e:concentration-4}
\|u-\mathcal F_N^*\indic_{X_\rho}\mathcal F_N u\|_{\ell^2_N}=\mathcal O(N^{-\infty})\|u\|_{\ell^2_N}
\end{gather}
where the constants in $\mathcal O(N^{-\infty})$ depend only on $\nu,\rho,\chi$.
\end{prop}
\begin{proof}
Define $\tilde k$ by~\eqref{e:tilde-k}. Then
$$
(B_N)^{\tilde k}u=\lambda^{\tilde k}u,\quad
|\lambda|^{-\tilde k}\leq M^{\nu} N^{\rho \nu}.
$$
Applying~\eqref{e:concentration-1} to $\lambda^{-\tilde k}u$ and using the bound $\|B_N\|_{\ell^2_N\to\ell^2_N}\leq 1$,
we obtain~\eqref{e:concentration-3}. Applying~\eqref{e:concentration-2} to $\lambda^{-\tilde k}u$,
we obtain~\eqref{e:concentration-4}.
\end{proof}
Finally, Proposition~\ref{l:key-eigenvalues} implies the following statement,
which proves the second part of Theorem~\ref{t:gap-detailed}.
\begin{prop}
  \label{l:fup-reduction}
With $r_k$ defined in~\eqref{e:r-k}, put
\begin{equation}
  \label{e:beta-limsup}
\beta:=-\limsup_{k\to\infty}{\log r_k\over k\log M}.
\end{equation}
Then
\begin{equation}
\label{e:spec-limsup}
\limsup_{N\to\infty}\max\{|\lambda|\colon \lambda\in\Sp(B_N)\}\ \leq\ M^{-\beta}.
\end{equation}
\end{prop}
\begin{proof}
Fix $\rho\in (0,1)$. Take $k\in\mathbb N$, $N:=M^k$, and assume that $\lambda\in \mathbb C$ is an eigenvalue
of $B_N$ such that $|\lambda|\geq M^{-\beta}$. (Note that $\beta$ is always finite, see~\eqref{e:JN-2}.)
Choose a normalized eigenfunction
$$
u\in\ell^2_N,\quad
B_Nu=\lambda u,\quad
\|u\|_{\ell^2_N}=1.
$$
Combining~\eqref{e:concentration-3} and~\eqref{e:concentration-4}, we obtain
\begin{equation}
  \label{e:fupuse-1}
\begin{aligned}
1=\|u\|_{\ell^2_N}&\leq M^\beta|\lambda|^{-\rho k}\|\indic_{X_\rho}\mathcal F_N^*\indic_{X_\rho}\mathcal F_N u\|_{\ell^2_N}+\mathcal O(N^{-\infty})\\
&\leq M^\beta|\lambda|^{-\rho k}\|\indic_{X_\rho}\mathcal F_N^*\indic_{X_\rho}\|_{\ell^2_N\to \ell^2_N}
+\mathcal O(N^{-\infty})
\end{aligned}
\end{equation}
where the constants in $\mathcal O(N^{-\infty})$ depend only on $\beta,\rho,\chi$.

Using~\eqref{e:same-shift},
the following corollary of~\eqref{e:X-rho}:
$$
0\leq \mathbf 1_{X_{\rho}}\leq \sum_{m\in\mathbb Z,\, |m|\leq 2N^{1-\rho}}\mathbf 1_{\mathcal C_k+m},
$$
and the triangle inequality for the operator norm, we estimate
\begin{equation}
  \label{e:fupuse-2}
\begin{aligned}
\|\indic_{X_\rho}\mathcal F_N^*\indic_{X_\rho}\|_{\ell^2_N\to \ell^2_N}
&=\|\indic_{X_\rho}\mathcal F_N\indic_{X_\rho}\|_{\ell^2_N\to \ell^2_N}\\
&\leq \sum_{|m|,|m'|\leq 2N^{1-\rho}}\|\indic_{\mathcal C_k+m}\mathcal F_N\indic_{\mathcal C_k+m'}\|_{\ell^2_N\to\ell^2_N}\\
&\leq 5N^{2(1-\rho)} r_k.
\end{aligned}
\end{equation}
Combining~\eqref{e:fupuse-1} and~\eqref{e:fupuse-2}, we obtain a bound on the
spectral radius of~$B_N$:
$$
\big(\max\{|\lambda|\colon \lambda\in\Sp(B_N)\}\big)^{\rho k}\leq \max\big\{M^{-\beta\rho k},
5M^\beta N^{2(1-\rho)}r_k
\big\}.
$$
Taking both sides to the power $1\over \rho k$ and taking the limit, we get
$$
\limsup_{N\to\infty}\max\{|\lambda|\colon \lambda\in\Sp(B_N)\}\leq\max\{M^{-\beta},M^{(2-2\rho-\beta)/\rho}\}.
$$
This is true for all $\rho\in (0,1)$; taking the limit $\rho\to 1$, we obtain~\eqref{e:spec-limsup}.
\end{proof}

\section{Fractal uncertainty principle}
  \label{s:fup}

In this section, we study bounds on the operator norms
\begin{equation}
  \label{e:op-norm-fup}
\|\indic_X\mathcal F_N\indic_Y\|_{\ell^2_N\to \ell^2_N}
\end{equation}
where $X,Y\subset\mathbb Z_N$. We will derive several general bounds
and apply them to the special case of Cantor sets $\mathcal C_k$
defined in~\eqref{e:C-k},
estimating the norm
\begin{equation}
  \label{e:r-k-again}
r_k:=\|\indic_{\mathcal C_k}\mathcal F_N\indic_{\mathcal C_k}\|_{\ell^2_N\to\ell^2_N}
\end{equation}
and
finishing the proof of Theorem~\ref{t:gap-detailed} (see the end of~\S\ref{s:improve-2}).
We also establish better bounds when $\delta$ is close to $1/2$ using additive
energy (see~\S\ref{s:ae-improvements}),
consider the special case when the fractal uncertainty principle
exponent $\beta$ defined in~\eqref{e:beta-fup} has the maximal possible value
(see~\S\ref{s:special-alphabets}), and give lower bounds on~$r_k$
(see~\S\ref{s:improve-all}).

First of all, we have the following basic estimates:
\begin{align}
  \label{e:0-fup}
\|\indic_X\mathcal F_N\indic_Y\|_{\ell^2_N\to\ell^2_N}&\leq 1,\\
  \label{e:pres-fup}
\|\indic_X\mathcal F_N\indic_Y\|_{\ell^2_N\to\ell^2_N}&\leq 
\sqrt{|X|\cdot |Y|\over N}.
\end{align}
The operator norm bound~\eqref{e:pres-fup} follows from the following
formula for the Hilbert--Schmidt norm:
\begin{equation}
  \label{e:hs}
\|\indic_X\mathcal F_N\indic_Y\|_{\HS}=\sqrt{|X|\cdot |Y|\over N}.
\end{equation}
For the case of Cantor sets $X=Y=\mathcal C_k$,
the bounds~\eqref{e:0-fup}, \eqref{e:pres-fup} yield
\begin{equation}
  \label{e:this-is-trivial}
r_k\leq\min\big(1,N^{\delta-{1\over 2}}\big)
\end{equation}
where $\delta\in (0,1)$ is defined in~\eqref{e:delta}. Therefore, the exponent
$\beta$ defined in~\eqref{e:beta-fup} satisfies
$$
\beta\geq \max\Big(0,{1\over 2}-\delta\Big).
$$
(We will show that the limit~\eqref{e:beta-fup} exists in Proposition~\ref{l:beta-limit} below.)
Also, if we define
$$
\supp u:=\{j\in \mathbb Z_N\mid u(j)\neq 0\},\quad
u\in\ell^2_N,
$$
then~\eqref{e:op-norm-fup} is equal to
\begin{equation}
\label{e:norm-ip}
\sup\bigg\{{|\langle \mathcal F_N u,v\rangle|\over \|u\|_{\ell^2_N}\|v\|_{\ell^2_N}}\colon
u,v\in\ell^2_N\setminus \{0\},\
\supp u\subset Y,\
\supp v\subset X\bigg\}.
\end{equation}
Finally, if $X,Y\subset\mathbb Z_N$, $j,k\in\mathbb Z_N$,
and the sets $X+j,Y+k$ are defined using addition in $\mathbb Z_N$, then
\begin{equation}
  \label{e:same-shift}
\|\indic_{X+j}\mathcal F_N\indic_{Y+k}\|_{\ell^2_N\to \ell^2_N}
=\|\indic_X\mathcal F_N \indic_Y\|_{\ell^2_N\to\ell^2_N}.
\end{equation}
Indeed, the circular shift gives a unitary operator on $\ell^2_N$, and it is conjugated by the Fourier
transform to a multiplication operator which commutes with $\indic_X,\indic_Y$.

In particular~\eqref{e:same-shift} implies that if $j\geq 1$ and $\mathcal A\subset \{0,\dots,M-j-1\}$,
then the pairs $(M,\mathcal A)$ and $(M,\mathcal A+j)$ have the same norms
$r_k$, and thus the same value of $\beta$ defined in~\eqref{e:beta-fup}.

\subsection{Submultiplicativity}
  \label{s:submul}

In this section, we assume that $N$ factorizes as $N=N_1N_2$, where
$N_1,N_2\in\mathbb N$.
The following lemma gives a way to reduce the Fourier
transform $\mathcal F_N$ to Fourier transforms $\mathcal F_{N_1},\mathcal F_{N_2}$,
using a technique similar to the one employed in fast Fourier transform (FFT) algorithms.
The resulting submultiplicative inequality~\eqref{e:submultiplicative}
is a crucial component of the proof of Theorem~\ref{t:gap-detailed},
making it possible to reduce bounds for large $N$ to bounds
for bounded $N$.
\begin{lemm}
  \label{l:split-apart}
Assume $u,v\in\ell^2_N$. For $\ell_1,j_2\in\mathbb Z_N$, define $\tilde u_{\ell_1}\in\ell^2_{N_2}$,
$\tilde v_{j_2}\in\ell^2_{N_1}$ by
$$
\begin{aligned}
\tilde u_{\ell_1}(\ell_2)=u(\ell_1+N_1\ell_2),&\quad
\ell_2\in \mathbb Z_{N_2};
\\
\tilde v_{j_2}(j_1)=v(j_2+N_2j_1),&\quad
j_1\in \mathbb Z_{N_1}.
\end{aligned}
$$
Then
\begin{equation}
  \label{e:split-apart}
\langle \mathcal F_N u,v\rangle=\sum_{\ell_1=0}^{N_1-1}\sum_{j_2=0}^{N_2-1}
\exp\Big(-{2\pi i j_2\ell_1 \over N}\Big)\mathcal F_{N_2}\tilde u_{\ell_1}(j_2)
\cdot \overline{\mathcal F_{N_1}^* \tilde v_{j_2}(\ell_1)}.
\end{equation}
\end{lemm}
\begin{proof}
We write
$$
\begin{aligned}
\langle \mathcal F_Nu,v\rangle&=
{1\over\sqrt{N}}\sum_{j,\ell=0}^{N-1}\exp\Big(-{2\pi i j\ell\over N}\Big)u(\ell)\overline{v(j)}\\
&={1\over\sqrt{N}}\sum_{j_1,\ell_1=0}^{N_1-1}
\sum_{j_2,\ell_2=0}^{N_2-1}\exp\Big(-{2\pi i (j_2+N_2j_1)(\ell_1+N_1\ell_2)\over N}\Big)\tilde u_{\ell_1}(\ell_2)\overline{\tilde v_{j_2}(j_1)}.
\end{aligned}
$$
Now, since $\exp(-2\pi i j_1\ell_2)=1$, we have for each $j_2,\ell_1$
$$
\begin{gathered}
{1\over\sqrt{N}}\sum_{j_1=0}^{N_1-1}
\sum_{\ell_2=0}^{N_2-1}\exp\Big(-{2\pi i (j_2+N_2j_1)(\ell_1+N_1\ell_2)\over N}\Big)\tilde u_{\ell_1}(\ell_2)\overline{\tilde v_{j_2}(j_1)}\\
=\exp\Big(-{2\pi i j_2\ell_1\over N}\Big)\mathcal F_{N_2}\tilde u_{\ell_1}(j_2)
\cdot\overline{\mathcal F_N^* \tilde v_{j_2}(\ell_1)}
\end{gathered}
$$
finishing the proof.
\end{proof}
As a corollary of Lemma~\ref{l:split-apart}
we get the following bound on the norm~\eqref{e:op-norm-fup}
when the sets $X,Y$ have special structure:
\begin{lemm}
  \label{l:norm-apart}
Assume that $X_1,Y_1\subset\{0,\dots,N_1-1\}$,
$X_2,Y_2\subset \{0,\dots,N_2-1\}$, and define
$X,Y\subset \{0,\dots, N-1\}\simeq \ell^2_N$ by
$$
\begin{aligned}
X&:=\{j_2+N_2j_1\mid j_1\in X_1,\ j_2\in X_2\},\\
Y&:=\{\ell_1+N_1\ell_2\mid \ell_1\in Y_1,\ \ell_2\in Y_2\}.
\end{aligned}
$$
Then
$$
\|\indic_X\mathcal F_N\indic_Y\|_{\ell^2_N\to\ell^2_N}
\leq
\|\indic_{X_1}\mathcal F_{N_1}\indic_{Y_1}\|_{\ell^2_{N_1}\to\ell^2_{N_1}}
\cdot
\|\indic_{X_2}\mathcal F_{N_2}\indic_{Y_2}\|_{\ell^2_{N_2}\to\ell^2_{N_2}}.
$$ 
\end{lemm}
\begin{proof}
We use~\eqref{e:norm-ip}. Assume that $u,v\in\ell^2_N$
and $\supp u\subset Y$, $\supp v\subset X$. To estimate
$\langle\mathcal F_N u,v\rangle$, we use~\eqref{e:split-apart}.
We have $\tilde u_{\ell_1}=0$ unless $\ell_1\in Y_1$,
and $\tilde v_{j_2}=0$ unless $j_2\in X_2$, therefore
the sum on the right-hand side of~\eqref{e:split-apart}
can be taken over $\ell_1\in Y_1,j_2\in X_2$. By Cauchy--Schwartz,
we then get
$$
|\langle\mathcal F_N u,v\rangle|^2\leq
\bigg(\sum_{\ell_1\in Y_1\atop j_2\in X_2}|\mathcal F_{N_2}\tilde u_{\ell_1}(j_2)|^2\bigg)
\bigg(\sum_{\ell_1\in Y_1\atop j_2\in X_2}|\mathcal F_{N_1}^*\tilde v_{j_2}(\ell_1)|^2\bigg).
$$
Since $\supp \tilde u_{\ell_1}\subset Y_2$, $\supp \tilde v_{j_2}\subset X_1$, we have
$$
\begin{aligned}
\sum_{\ell_1\in Y_1\atop j_2\in X_2}|\mathcal F_{N_2}\tilde u_{\ell_1}(j_2)|^2
\ &\leq\
\|\indic_{X_2}\mathcal F_{N_2}\indic_{Y_2}\|_{\ell^2_{N_2}\to\ell^2_{N_2}}^2
\cdot \|u\|_{\ell^2_N}^2,\\
\sum_{\ell_1\in Y_1\atop j_2\in X_2}|\mathcal F_{N_1}^*\tilde v_{j_2}(\ell_1)|^2
\ &\leq\
\|\indic_{Y_1}\mathcal F_{N_1}^*\indic_{X_1}\|_{\ell^2_{N_1}\to\ell^2_{N_1}}^2
\cdot \|v\|_{\ell^2_N}^2
\end{aligned}
$$
finishing the proof.
\end{proof}
In the case of Cantor sets~\eqref{e:C-k}, by putting $N_1=M^{k_1},N_2=M^{k_2}$,
$X_1=Y_1=\mathcal C_{k_1}$, $X_2=Y_2=\mathcal C_{k_2}$,
$X=Y=\mathcal C_k$, Lemma~\ref{l:norm-apart} implies
the following submultiplicative inequality on the norm $r_k$ defined in~\eqref{e:r-k-again}:
\begin{equation}
  \label{e:submultiplicative}
r_{k_1+k_2}\leq r_{k_1}r_{k_2},\quad
k_1,k_2\in\mathbb N.
\end{equation}
The sequence $\log r_k$ is then subadditive, which by Fekete's Lemma gives
\begin{prop}
  \label{l:beta-limit}
For $\beta$ defined in~\eqref{e:beta-limsup}, we have
\begin{equation}
  \label{e:beta-limit}
\beta=-\lim_{k\to\infty}{\log r_k\over k\log M}=-\inf_{k}{\log r_k\over k\log M}.
\end{equation}
\end{prop}
Proposition~\ref{l:beta-limit} implies that the inequality in~\eqref{e:beta-fup}
is proven once we show that
\begin{equation}
  \label{e:beta-little}
-{\log r_k\over k\log M}>\max\Big(0,{1\over 2}-\delta\Big)\quad\text{for some }k.
\end{equation}

\subsection{Improvements over the pressure bound}
\label{s:improve-1}

We start the proof of~\eqref{e:beta-little} by improving
over the pressure bound ${1\over 2}-\delta$. We rely on the following
general
\begin{figure}
\includegraphics[scale=0.75]{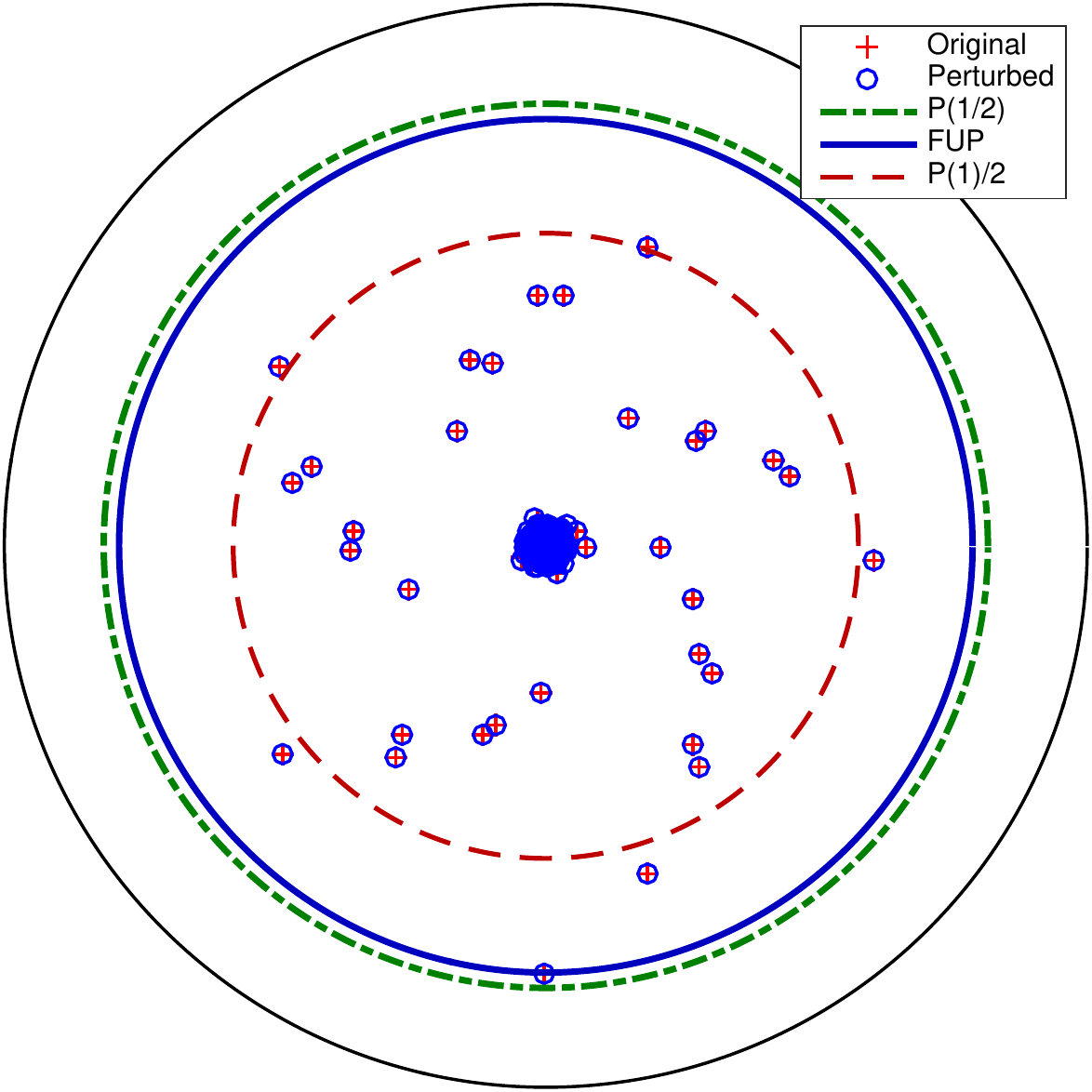}
\caption{The spectrum of $B_N$ for $M=6$, $\mathcal A=\{2,3\}$, $k=5$. The pressure bound on the spectral radius
is not sharp but the slightly smaller fractal uncertainty bound appears
to be sharp.}
\label{f:improved-pressure}
\end{figure}
%
\begin{lemm}
  \label{l:improve-pressure}
Assume $X,Y\subset \mathbb Z_N$ and there exist
\begin{equation}
  \label{e:off-beat}
j,j'\in X,\quad
\ell,\ell'\in Y,\quad
(j-j')(\ell-\ell')\notin N\mathbb Z.
\end{equation}
Then for some global constant $\mathbf K$,
\begin{equation}
  \label{e:improve-pressure}
\|\indic_X\mathcal F_N\indic_Y\|_{\ell^2_N\to\ell^2_N}\leq
\Big(1-{1\over \mathbf KN^4}\Big)\sqrt{|X|\cdot |Y|\over N}.
\end{equation}
\end{lemm}
\begin{proof}
Let $\sigma_1\geq\sigma_2\geq\dots\sigma_N\geq 0$ be the singular values
of $\indic_X\mathcal F_N\indic_Y$. Then
$$
\sigma_1=\|\indic_X\mathcal F_N\indic_Y\|_{\ell^2_N\to \ell^2_N}.
$$
Moreover, by~\eqref{e:hs}
\begin{equation}
  \label{e:sing-sum}
\sigma_1^2+\sigma_2^2\leq \sum_{r=1}^N \sigma_r^2=\|\indic_X\mathcal F_N\indic_Y\|_{\HS}^2
={|X|\cdot |Y|\over N}.
\end{equation}
It then remains to obtain a lower bound on $\sigma_1\sigma_2$.

By Weyl inequalities for products of operators~\cite[Theorem~3.3.16(d)]{HornJohnson}, the singular values of $\indic_X\mathcal F_N\indic_Y$ are greater
than or equal than those of
\begin{equation}
  \label{e:smallmat}
\indic_{X'}\mathcal F_N\indic_{Y'}=\indic_{X'}(\indic_X\mathcal F_N\indic_Y)\indic_{Y'},\quad
X':=\{j,j'\},\quad
Y':=\{\ell,\ell'\}.
\end{equation}
The matrix of~\eqref{e:smallmat} only has four nonzero elements,
and the absolute value of the determinant of the $2\times 2$ matrix composed of these is equal to
$$
{2\over N}\Big|\sin\Big({\pi\over N}(j-j')(\ell-\ell')\Big)\Big|\geq {2\over N}\sin\Big({\pi\over N}\Big)
$$
where in the last inequality we used~\eqref{e:off-beat}.
Therefore,
$$
\sigma_1\sigma_2\geq {2\over N}\sin\Big({\pi\over N}\Big)
$$
and by~\eqref{e:sing-sum} we have for $|X|\cdot |Y|\leq N$ and some global constant $\mathbf K$,
$$
\begin{aligned}
\sigma_1^2&={\sigma_1^2+\sigma_2^2+\sqrt{(\sigma_1^2+\sigma_2^2)^2-4\sigma_1^2\sigma_2^2}\over 2}\\
&\leq {|X|\cdot|Y|\over 2N}\Bigg(1+\sqrt{1-{16\over |X|^2\cdot |Y|^2}\sin^2\Big({\pi\over N}\Big)}\,\Bigg)\\
&\leq {|X|\cdot |Y|\over N}\Big(1-{1\over \mathbf KN^4}\Big)^2.
\end{aligned}
$$
For $|X|\cdot |Y|\geq N+1$, \eqref{e:improve-pressure} holds simply
by~\eqref{e:0-fup}; this finishes the proof.
\end{proof}
In the case of Cantor sets, Lemma~\ref{l:improve-pressure} implies
\begin{corr}
  \label{l:improve-pressure-cantor}
We have for some global constant $\mathbf K$ and $\beta$ defined in~\eqref{e:beta-limit}
\begin{equation}
  \label{e:improve-pressure-cantor}
\beta\geq -{\log r_2\over 2\log M}\geq{1\over 2}-\delta+{1\over \mathbf KM^8\log M}.
\end{equation}
\end{corr}
\begin{proof}
Take $a,a'\in\mathcal A$, $a<a'$,
and put $j=\ell=Ma+a$, $j'=\ell'=Ma+a'$.
Then $(j-j')(\ell-\ell')=(a'-a)^2\in (0,M^2)$, thus~\eqref{e:off-beat}
holds for $N=M^2$, $X=Y=\mathcal C_2$. The bound~\eqref{e:improve-pressure}
then gives~\eqref{e:improve-pressure-cantor}.
\end{proof}
\Remark The power of $M$ in~\eqref{e:improve-pressure-cantor} is most likely
not sharp. However, Proposition~\ref{l:lower-1/2} below gives
examples of alphabets for $\delta<1/2$ with a
power upper bound on the improvement $\beta-({1\over 2}-\delta)$.
See Table~\ref{b:worst-FUP} and Figure~\ref{f:improved-pressure}
for numerical evidence.

\begin{table}
\begin{tabular}{|l|l|l|l|l||l|l|l|l|l|}
\hline
$M$ & $|\mathcal A|$ & $\delta$ & $k$ & \vrule height12pt width0pt depth6pt$\beta_{\min}-\max(0,{1\over 2}-\delta)$ &
$M$ & $|\mathcal A|$ & $\delta$ & $k$ & $\beta_{\min}-\max(0,{1\over 2}-\delta)$
\\\hline
10 & 2 & 0.3010 & 12 & $5.4\cdot 10^{-3}$ & 
9 & 2 & 0.3155 & 12 & $6.5 \cdot 10^{-3}$
\\\hline
8 & 2 & 0.3333 & 12 & $9.5\cdot 10^{-3}$
& 
7 & 2 & 0.3562 & 12 & $1.3\cdot 10^{-2}$ 
\\\hline
6 & 2 & 0.3869 & 12 & $2\cdot 10^{-2}$ 
& 
5 & 2 & 0.4307 & 12 & $3.2\cdot 10^{-2}$ 
\\\hline
10 & 3 & 0.4771 & 7 & $3.7\cdot 10^{-2}$ 
& 
\textbf{4} & \textbf{2} & \textbf{0.5000} & \textbf{12} & $\mathbf{5.4\cdot 10^{-2}}$ 
\\\hline
\textbf{9} & \textbf{3} & \textbf{0.5000} & \textbf{7} & $\mathbf{4.1\cdot 10^{-2}}$ 
& 
8 & 3 & 0.5283 & 7 & $3.4\cdot 10^{-2}$ 
\\\hline
7 & 3 & 0.5646 & 7 & $2.1\cdot 10^{-2}$ 
& 
10 & 4 & 0.6021 & 6 & $7.3\cdot 10^{-3}$ 
\\\hline
6 & 3 & 0.6131 & 7 & $8.7\cdot 10^{-3}$
& 
3 & 2 & 0.6309 & 12 & $6.2\cdot 10^{-3}$ 
\\\hline
9 & 4 & 0.6309 & 6 & $3.6\cdot 10^{-3}$
& 
8 & 4 & 0.6667 & 6 & $1.1\cdot 10^{-3}$ 
\\\hline
5 & 3 & 0.6826 & 7 & $5.1\cdot 10^{-4}$ 
& 
10 & 5 & 0.6990 & 5 & $1.3\cdot 10^{-4}$
\\\hline
7 & 4 & 0.7124 & 6 & $4.8\cdot 10^{-5}$
& 
9 & 5 & 0.7325 & 5 & $4.6\cdot 10^{-6}$ 
\\\hline
6 & 4 & 0.7737 & 6 & $3.5\cdot 10^{-7}$ 
& 
8 & 5 & 0.7740 & 5 & $4.5\cdot 10^{-8}$ 
\\\hline
10 & 6 & 0.7782 & 4 & $4.1\cdot 10^{-9}$ 
& 
4 & 3 & 0.7925 & 7 & $3\cdot 10^{-7}$ 
\\\hline
9 & 6 & 0.8155 & 4 & $<10^{-12}$ 
& 
7 & 5 & 0.8271 & 5 & $1.9\cdot 10^{-11}$ 
\\\hline
10 & 7 & 0.8451 & 4 & $<10^{-12}$ 
& 
5 & 4 & 0.8614 & 6 & $<10^{-12}$ 
\\\hline
8 & 6 & 0.8617 & 4 & $<10^{-12}$ 
& 
9 & 7 & 0.8856 & 4 & $<10^{-12}$ 
\\\hline
6 & 5 & 0.8982 & 5 & $<10^{-12}$ 
& 
10 & 8 & 0.9031 & 4 & $<10^{-12}$ 
\\\hline
7 & 6 & 0.9208 & 4 & $<10^{-12}$ 
& 
8 & 7 & 0.9358 & 4 & $<10^{-12}$ 
\\\hline
9 & 8 & 0.9464 & 4 & $<10^{-12}$ 
& 
10 & 9 & 0.9542 & 3 & $<10^{-12}$ 
\\\hline
\end{tabular}
\smallskip
\caption{Numerically computed minimal values $\beta_{\min}$
of $-\log r_k/(k\log M)$ 
(which approximate the fractal
uncertainty exponents $\beta(M,\mathcal A)$ by~\eqref{e:beta-limit})
for fixed $M,|\mathcal A|$, sorted by $\delta$.
The alphabets achieving $\beta_{\min}$ were all arithmetic progressions,
mostly with difference $1$.
Note that the improvement over the
standard bound, $\beta_{\min}-\max(0,{1\over 2}-\delta)$, is the largest
when $\delta={1\over 2}$, and is very small when $\delta$ is close to 1,
already for $M=4$. This is in agreement with
the lower envelope of Figure~\ref{f:fupcloud}.}
\label{b:worst-FUP}
\end{table}

\subsection{Improvements over the zero bound}
\label{s:improve-2}

We next show that the left-hand side of~\eqref{e:beta-little}
is greater than 0 for some $k$. We rely on the following general
\begin{lemm}
  \label{l:improve-0}
Assume $X,Y\subset\mathbb Z_N$ and for some $L\in \{1,\dots, N-1\}$, the following
two conditions hold:
\begin{enumerate}
\item $|X|\leq L$;
\item $Y$ has a gap of size $L$, that is
there exists $j\in \mathbb Z_N$ such that
$j,\dots,j+L-1\notin Y$,
with addition carried in $\mathbb Z_N$.
\end{enumerate}
Then
\begin{equation}
  \label{e:improve-0}
\|\indic_X\mathcal F_N\indic_Y\|_{\ell^2_N\to \ell^2_N}\leq \sqrt{1-2^{-2N}}.
\end{equation}
\end{lemm}
\begin{proof}
By cyclically shifting $Y$ and using~\eqref{e:same-shift}
we may assume that
$$
Y\subset \{0,\dots,N-L-1\}.
$$
Assume that $u\in\ell^2_N$, $\|u\|_{\ell^2_N}=1$, $\supp u\subset Y$, and consider the polynomial
$$
p(z)=\sum_{\ell\in Y} u(\ell) z^\ell.
$$
Note that $p$ has degree at most $N-L-1$.
Denoting $\omega_N=\exp(-2\pi i/N)$, we have
$$
\mathcal F_N u(j)={1\over \sqrt{N}}\,p(\omega_N^j),\quad j\in\mathbb Z_N.
$$
Using that $\|\mathcal F_N u\|_{\ell^2_N}=1$, we compute
$$
\|\indic_X\mathcal F_N u\|_{\ell^2_N}^2=
1-{1\over N}\sum_{0\leq j<N\atop j\notin X}|p(\omega_N^j)|^2.
$$
This immediately shows that $\|\indic_X\mathcal F_N u\|_{\ell^2_N}<1$,
since otherwise the equation $p(z)=0$ has at least $N-|X|\geq N-L$ roots.

To get a quantitative bound, we use Lagrange interpolation: since
$p$ has degree at most $N-L-1<N-|X|$, we write
$$
p(z)=\sum_{0\leq j<N\atop j\notin X}p(\omega_N^j)\mathcal L_j(z),\quad
\mathcal L_j(z)=\prod_{0\leq m<N\atop m\notin X,\ m\neq j}{z-\omega_N^m\over \omega_N^j-\omega_N^m}.
$$
Differentiating at $z=1$ the polynomial
$$
z^N-1=\prod_{r=0}^{N-1} (z-\omega_N^r),
$$
we get for each $j\in\mathbb Z_N$
$$
N=\prod_{r=1}^{N-1}(1-\omega_N^r)=\prod_{0\leq m<N\atop m\neq j}|\omega_N^j-\omega_N^m|.
$$
Assume $|z|=1$; since $|z-\omega_N^r|\leq 2$, we get for $j\notin X$
$$
|\mathcal L_j(z)|\leq {2^{N-|X|}\over N}\prod_{m\in X}|\omega_N^j-\omega_N^m|
\leq {2^N\over N}.
$$
Therefore, by H\"older's inequality we have for $|z|=1$
$$
|p(z)|^2\leq {2^{2N}\over N}\sum_{0\leq j<N\atop j\notin X}|p(\omega_N^j)|^2
=2^{2N}\big(1-\|\indic_X\mathcal F_N u\|_{\ell^2_N}^2\big).
$$
Since 
$$
1=\|\mathcal F_N u\|_{\ell^2_N}^2={1\over N}\sum_{r=0}^N |p(\omega_N^r)|^2,
$$
we obtain $\|\indic_X\mathcal F_N u\|_{\ell^2_N}^2\leq 1-2^{-2N}$, finishing the proof.
\end{proof}
In the case of Cantor sets, Lemma~\ref{l:improve-0} implies
\begin{figure}
\includegraphics[scale=0.75]{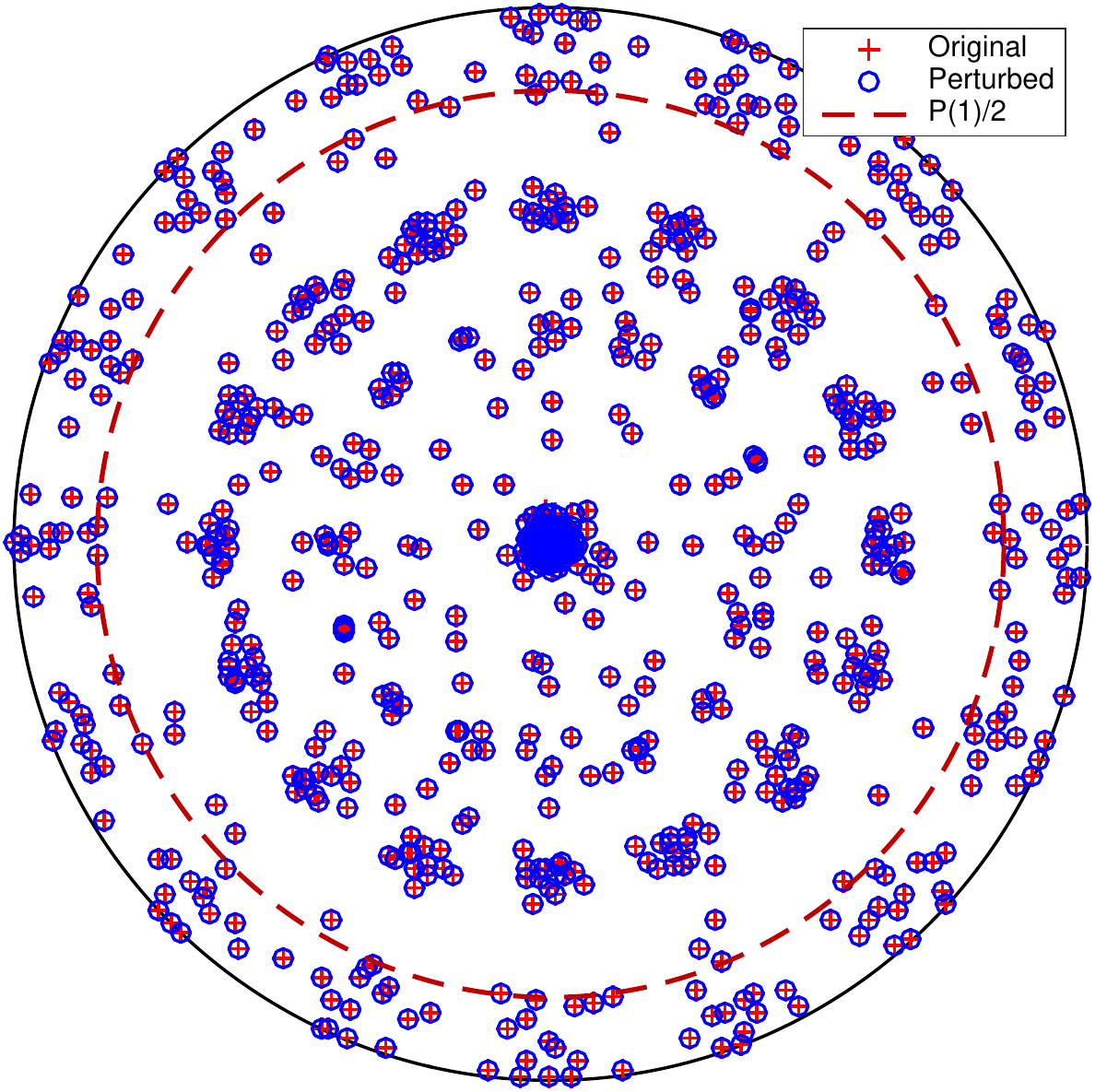}
\caption{The spectrum of $B_N$ for $M=7$, $\mathcal A=\{1,2,3,4,5\}$, $k=4$.
The numerically computed uncertainty principle exponent
is $\sim 10^{-11}$ and the spectral radius appears
to be very close to 1.}
\label{f:improved-0}
\end{figure}
%
\begin{corr}
  \label{l:improve-0-cantor}
For $\beta$ defined in~\eqref{e:beta-limit}, we have
\begin{equation}
  \label{e:improve-0-cantor}
\beta\geq -{\log r_k\over k\log M}\geq {2^{-2M^k}\over 2k\log M}>0\quad\text{for all }
k\geq {1\over 1-\delta}-{\log\lceil M^{1-\delta}-1\rceil\over(1-\delta)\log M}.
\end{equation}
\end{corr}
\begin{proof}
For each $a\in \mathcal A$, define
$$
L_a:=\max\{j\geq 0\colon a+\ell\notin\mathcal A\text{ for all }\ell\in\mathbb Z,\ 
1\leq \ell\leq j\}
$$
where addition is carried in $\mathbb Z_M$. We have
$$
\sum_{a\in\mathcal A}(L_a+1)=M.
$$
Therefore, there exists $a\in\mathcal A$ such that
$L_a\geq \lceil M^{1-\delta}-1\rceil$. Then the set $\mathcal C_k$
has a gap starting at $(a+1)M^{k-1}$ of size
$$
L=\lceil M^{1-\delta}-1\rceil M^{k-1}.
$$
By the condition on $k$, we have
$$
|\mathcal C_k|=M^{\delta k}\leq L.
$$
Applying Lemma~\ref{l:improve-0} with $N=M^k$, $X=Y=\mathcal C_k$,
we obtain~\eqref{e:improve-0-cantor}.
\end{proof}
\Remark For fixed $\delta\in (0,1)$ and large $M$, the bound on $k$ in~\eqref{e:improve-0-cantor}
is asymptotic to $k\geq \delta/(1-\delta)$. While we prove no upper bounds
on $\beta$, numerical evidence in Table~\ref{b:worst-FUP} and Figure~\ref{f:improved-0}
suggests that there exists some alphabets with $\delta>{1\over 2}$ for which
$\beta$ is very small and the spectral radius of $B_N$ is very close to 1.

We are now ready to finish the
\begin{proof}[Proof of Theorem~\ref{t:gap-detailed}]
The existence of the limit in~\eqref{e:beta-fup} follows
from Proposition~\ref{l:beta-limit}. The fact
that $\beta>\max(0,{1\over 2}-\delta)$ follows from Corollaries~\ref{l:improve-pressure-cantor}
and~\ref{l:improve-0-cantor}. Finally the asymptotic spectral radius bound~\eqref{e:gap-improves}
follows from Proposition~\ref{l:fup-reduction}.
\end{proof}

\subsection{Improvements using additive energy}
  \label{s:ae-improvements}
  
So far we have established lower bounds on the improvement $\beta-\max(0,{1\over 2}-\delta)$,
with $\beta$ defined in~\eqref{e:beta-limit}, which decay like 
a power of~$M$ for $\delta\leq {1\over 2}$ (see Corollary~\ref{l:improve-pressure-cantor})
and exponentially in~$M$ for $\delta\geq {1\over 2}$ (see Corollary~\ref{l:improve-0-cantor}).
However, Figure~\ref{f:fupcloud} and Table~\ref{b:worst-FUP} indicate that the value
of $\beta-\max(0,{1\over 2}-\delta)$ should be larger when $\delta\approx{1\over 2}$.
In this section we explain this observation
by establishing lower bounds on $\beta$ which decay like ${1\over\log M}$ when $\delta\approx {1\over 2}$.

Our bounds rely on the following general statement:
\begin{lemm}
  \label{l:fup-ae}
Assume $X,Y\subset\mathbb Z_N$. Then
$$
\|\indic_X\mathcal F_N\indic_Y\|_{\ell^2_N\to\ell^2_N}\leq {\widetilde E_A(X)^{1/8}|Y|^{3/8}\over N^{3/8}}
$$
where the quantity $\widetilde E_A(X)$, called \textbf{additive energy} of $X$, is defined by
$$
\widetilde E_A(X):=\big|\{(a,b,c,d)\in X^4\mid a+b=c+d\mod N\}\big|.
$$
\end{lemm}
\begin{proof}
The operator $A=(\indic_X\mathcal F_N\indic_Y)^*\indic_X\mathcal F_N\indic_Y
=\indic_Y\mathcal F_N^*\indic_X\mathcal F_N\indic_Y$
has the matrix
$$
A_{j\ell}={\mathbf 1_{Y}(j)\mathbf 1_{Y}(\ell)\over\sqrt{N}} \mathcal F_N(\mathbf 1_X)(\ell-j)
$$
where $\mathbf 1_{X},\mathbf 1_{Y}$ denote indicator functions. By Schur's and
H\"older's inequalities,
$$
\begin{aligned}
\|\indic_X\mathcal F_N\indic_Y\|_{\ell^2_N\to \ell^2_N}^2&
=\|A\|_{\ell^2_N\to\ell^2_N}\\
&\leq \max_{j\in Y}\sum_{\ell\in Y}{|\mathcal F_N(\mathbf 1_X)(\ell-j)|\over\sqrt N}\\
&\leq {|Y|^{3/4}\cdot \|\mathcal F_N(\mathbf 1_X)\|_{\ell^4_N}\over \sqrt{N}}.
\end{aligned}
$$
Now
$$
\|\mathcal F_N(\mathbf 1_X)\|_{\ell^4_N}^4
={1\over N^2}\sum_{\ell=0}^{N-1}\sum_{a,b,c,d\in X}
\exp\Big({2\pi i(a+b-c-d)\ell\over N}
\Big)={\widetilde E_A(X)\over N}
$$
finishing the proof.
\end{proof}
We remark that for all $X$
\begin{equation}
  \label{e:ae-basic-1}
|X|^2\leq \widetilde E_A(X)\leq |X|^3,
\end{equation}
where the first bound comes from considering quadruples of the form $(a,b,a,b)$ and
the second one, from the fact that $a,b,c$ determine $d$. Moreover,
by writing
$$
\widetilde E_A(X)=\sum_{j=0}^{N-1}\widetilde F_j(X)^2,\quad
\widetilde F_j(X):=\big|\{(a,b)\in X^2\mid a-b=j\mod N\}\big|,
$$
and using H\"older's inequality together with the identity
$\sum_j \widetilde F_j(X)=|X|^2$, we get
\begin{equation}
  \label{e:ae-basic-2}
\widetilde E_A(X)\geq {|X|^4\over N}.
\end{equation}
In the case of Cantor sets defined in~\eqref{e:C-k}, Lemma~\ref{l:fup-ae} immediately gives
\begin{corr}
  \label{l:improved-ae}
Assume that for some constants $C,\gamma\geq 0$ and all $k$,
\begin{equation}
  \label{e:ae-improved}
\widetilde E_A(\mathcal C_k)\leq CN^{3\delta-\gamma}.
\end{equation}
Then the exponent $\beta$ defined in~\eqref{e:beta-limit} satisfies
\begin{equation}
  \label{e:beta-ae}
\beta\geq {3\over 4}\Big({1\over 2}-\delta\Big)+{\gamma\over 8}.
\end{equation}
\end{corr}
Note that by~\eqref{e:ae-basic-1}, \eqref{e:ae-basic-2} we have
$$
\gamma\leq \min(\delta,1-\delta).
$$
Therefore, the bound~\eqref{e:beta-ae} cannot improve over the standard
bound $\max(0,{1\over 2}-\delta)$ unless $\delta$ is close to $1/2$,
specifically $\delta\in (1/3,4/7)$.
However, the advantage of~\eqref{e:beta-ae} over
the bounds in Corollaries~\ref{l:improve-pressure-cantor}, \ref{l:improve-0-cantor}
for $\delta\approx {1\over 2}$ is that the
exponent $\gamma$ from~\eqref{e:ae-improved} can be computed explicitly as follows:
\begin{lemm}
  \label{l:ae-computed}
Let $\rho(\mathcal A)$ be the spectral radius of the $2\times 2$ matrix
\begin{equation}
  \label{e:ae-matrix}
\mathcal M(\mathcal A)=\begin{pmatrix}
E_{M-1}(\mathcal A)+E_{M+1}(\mathcal A)&
2E_{M}(\mathcal A)\\
E_1(\mathcal A)&
E_0(\mathcal A)
\end{pmatrix}
\end{equation}
where (with the equality below in $\mathbb Z$ rather than $\mathbb Z/N\mathbb Z$)
\begin{equation}
  \label{e:E-ell}
E_\ell(\mathcal A):=\big|\{(a,b,c,d)\in\mathcal A^4\mid a+b-c-d=\ell\}\big|,\quad
\ell\in\mathbb Z.
\end{equation}
Then \eqref{e:ae-improved} holds for each $\gamma<\gamma_{\mathcal A}$ where
\begin{equation}
  \label{e:gamma-A}
\gamma_{\mathcal A}=3\delta-{\log\rho(\mathcal{A})\over\log M}.
\end{equation}
\end{lemm}
\begin{proof}
We use the standard addition algorithm, keeping track of the carry digits.
For $k\in\mathbb N$, consider the vector
$$
x^{(k)}=(x^{(k)}_0,x^{(k)}_1,x^{(k)}_2)\in\mathbb N_0^3
$$
defined as follows: $x^{(k)}_j$ is the number of quadruples
$(a,b,c,d)\in \mathcal C_k^4$ such that
$$
a+b=c+d+(j-1)M^k\quad\text{in }\mathbb Z.
$$
A direct calculation shows that if we put $x^{(0)}:=(0,1,0)$, then
for $k\in\mathbb N$
$$
x^{(k)}=\begin{pmatrix}
E_{M-1}(\mathcal A)&
E_{M}(\mathcal A)&
E_{M+1}(\mathcal A)\\
E_1(\mathcal A)&
E_0(\mathcal A)&
E_1(\mathcal A)\\
E_{M+1}(\mathcal A)&
E_M(\mathcal A)&
E_{M-1}(\mathcal A)
\end{pmatrix}x^{(k-1)},\quad
\widetilde E_A(\mathcal C_k)=\langle x^{(k)},(1,1,1)\rangle.
$$
Now, it is easy to see that $x^{(k)}_0=x^{(k)}_2$ for all $k$. In fact,
if $y^{(k)}=(2x^{(k)}_0,x^{(k)}_1)$, then
$$
y^{(k)}=\mathcal M(\mathcal A)^{k}\begin{pmatrix}0\\1\end{pmatrix},\quad
\widetilde E_A(\mathcal C_k)=\langle y^{(k)},(1,1)\rangle.
$$
It follows that~\eqref{e:ae-improved} holds for each $\gamma<\gamma_A$.
\end{proof}
The next statement, proved in Appendix~\ref{s:additive-combinatorics},
gives a positive lower bound on the exponent $\gamma_A$ from~\eqref{e:gamma-A}:
\begin{prop}
\label{l:ub-rho}
For any $\zeta>0$, there exists a constant $\varepsilon>0$ only depending on~$\zeta$ such that whenever $1<|\mathcal{A}|<\frac{2}{3}(1-\zeta)M$, the spectral radius $\rho(\mathcal{A})$ of the
matrix $\mathcal{M}(\mathcal{A})$ defined in~\eqref{e:ae-matrix} satisfies
\begin{equation}
\rho(\mathcal{A})\leq(1-\varepsilon)|\mathcal{A}|^3.
\end{equation}
Thus $\gamma_\mathcal{A}\geq c/\log M$ where $c=-\log(1-\varepsilon)>0$.
\end{prop}
Combining Corollary~\ref{l:improved-ae}, Lemma~\ref{l:ae-computed},
and Proposition~\ref{l:ub-rho}, we obtain
\begin{prop}
  \label{l:improved-ae-ultimate}
There exists a global constant $\mathbf K$ such that for all
$(M,\mathcal A)$ satisfying
$$
\Big|\delta-{1\over 2}\Big|\leq {1\over \mathbf K\log M}
$$
the fractal uncertainty exponent $\beta=\beta(M,\mathcal A)$ defined in~\eqref{e:beta-limit} satisfies
\begin{equation}
  \label{e:log-improved}
\beta\geq \max\Big(0,{1\over 2}-\delta\Big)+{1\over \mathbf K\log M}.
\end{equation}
\end{prop}
A family of alphabets with an upper bound on $\beta$ matching~\eqref{e:log-improved}
is presented in Proposition~\ref{l:lower-1/2} below.

\subsection{Special alphabets with the best possible exponent}
  \label{s:special-alphabets}

For any two nonempty sets $X,Y\subset \mathbb Z_N$, we have
the lower bound
\begin{equation}
  \label{e:JN-1}
\|\indic_X \mathcal F_N\indic_Y\|_{\ell^2_N\to \ell^2_N}\geq
\sqrt{\max(|X|,|Y|)\over N}.
\end{equation}
Indeed, to show the lower bound of $\sqrt{|X|/N}$, it is enough to apply
the operator $\indic_X\mathcal F_N$ to any element of $\ell^2_N$ supported
at one point of $Y$; taking adjoints, we obtain the lower bound $\sqrt{|Y|/N}$.

For $X=Y=\mathcal C_k$ defined in~\eqref{e:C-k}, and $r_k$ defined
in~\eqref{e:r-k-again}, \eqref{e:JN-1} gives $r_k\geq N^{(\delta-1)/2}$,
implying the following bound on the exponent $\beta$ from~\eqref{e:beta-limit}:
\begin{equation}
  \label{e:JN-2}
\beta\leq {1-\delta\over 2}.
\end{equation}
As expected from the numerical evidence in Figure~\ref{f:fupcloud} and Table~\ref{b:worst-FUP},
and proved in Proposition~\ref{l:not-the-best} below, for many alphabets
the actual value of $\beta$ is strictly below the upper bound~\eqref{e:JN-2}.
However, the next statement implies that there exist alphabets
for which~\eqref{e:JN-2} is an equality, that is $\beta$ takes
its largest allowed value:
\begin{prop}
  \label{l:special-alpha}
For an alphabet $\mathcal A\subset \mathbb Z_M$, define the 1-periodic function
\begin{equation}
  \label{e:G-A}
G_{\mathcal A}(x)={1\over \sqrt{M}}\sum_{a\in\mathcal A}\exp(-2\pi i a x),\
x\in\mathbb R.
\end{equation}
Assume that
\begin{equation}
  \label{e:special-alpha}
G_{\mathcal A}\Big({b-b'\over M}\Big)=0\quad\text{for all }
b,b'\in \mathcal A,\
b\neq b'.
\end{equation}
Then
\begin{equation}
  \label{e:special-r}
r_k=\Big({|\mathcal A|\over M}\Big)^{k/2}\quad\text{for all }k
\end{equation}
and thus $\beta=(1-\delta)/2$.
\end{prop}
\begin{proof}
Condition~\eqref{e:special-alpha} implies that any two different rows of the matrix
of the transformation $\indic_A\mathcal F_M\indic_A$ are orthogonal to each other.
Since each of these rows has $\ell^2_M$ norm equal to 0 or $\sqrt{|\mathcal A|/M}$, we have
$$
r_1=\sqrt{|\mathcal A|/M}.
$$
By~\eqref{e:submultiplicative} we obtain an upper bound on $r_k$
which matches the lower bound following from~\eqref{e:JN-1}.
This immediately implies~\eqref{e:special-r}.
\end{proof}
%
\begin{figure}
\includegraphics[scale=0.5]{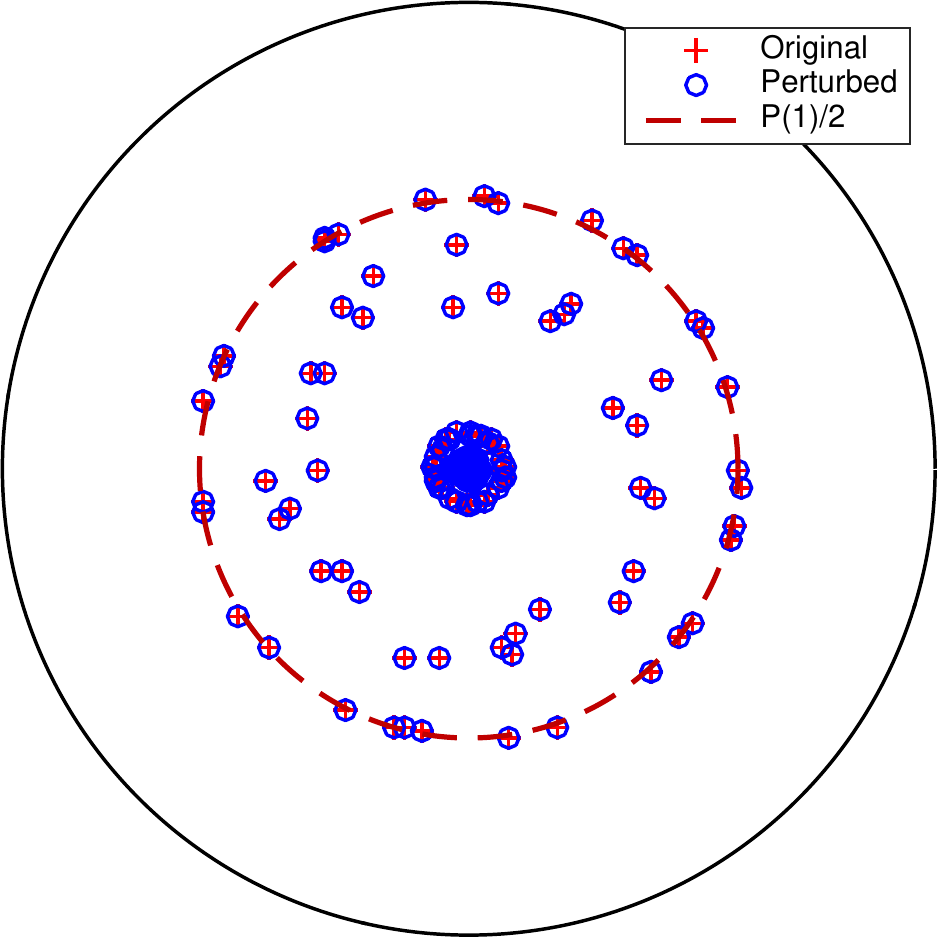}\quad
\includegraphics[scale=0.5]{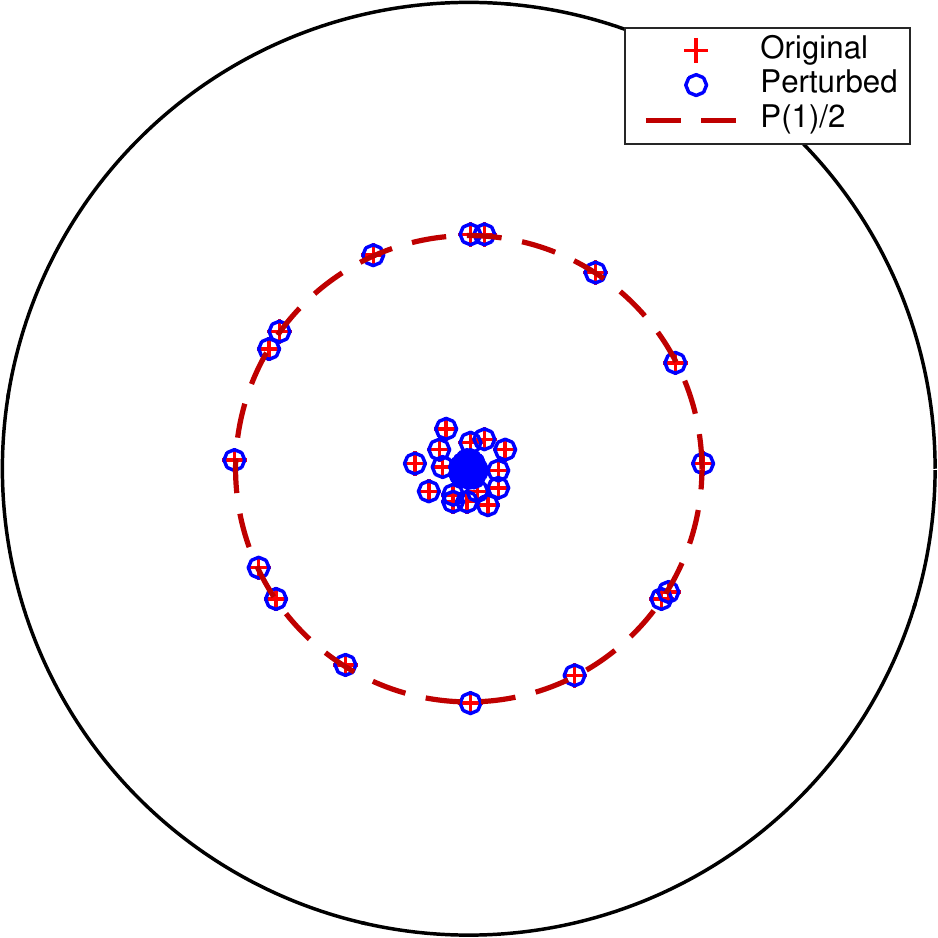}\quad
\includegraphics[scale=0.5]{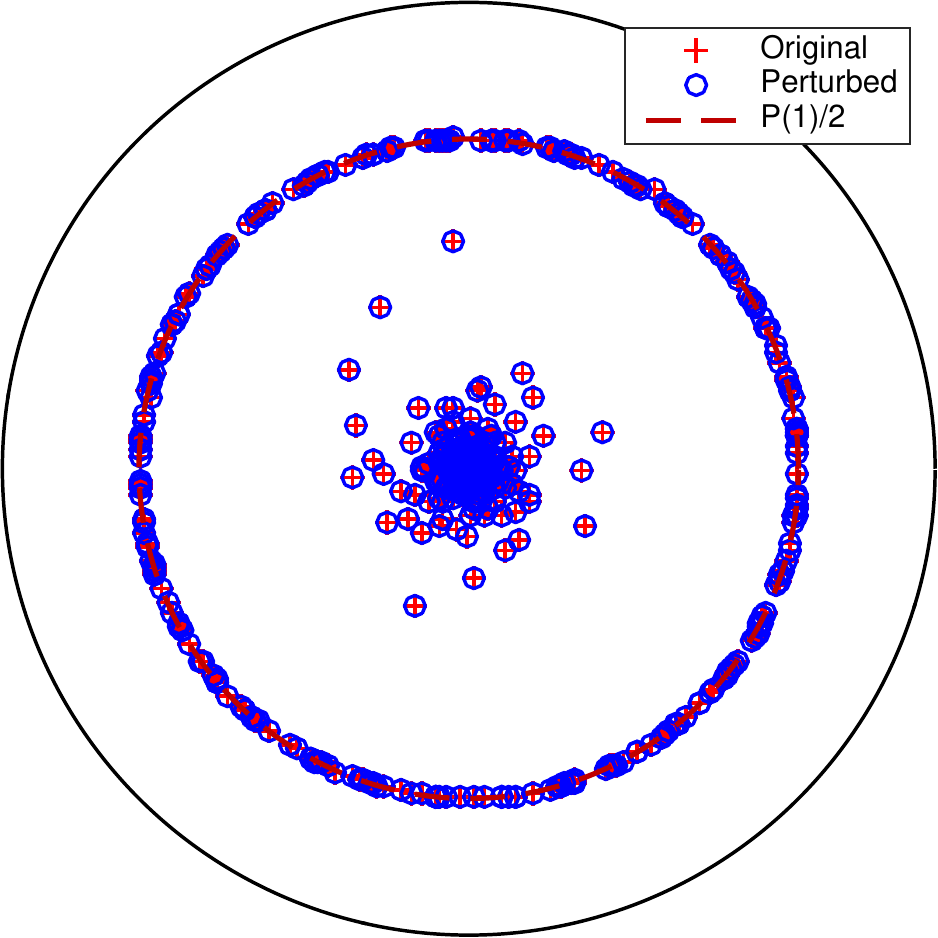}
\hbox to\hsize{\hss
\quad $M=6$, $\mathcal A=\{1,4\}$\hss\hss\qquad
$M=8$, $\mathcal A=\{2,4\}$\hss\hss\quad
$M=8$, $\mathcal A=\{1,2,5,6\}$
\hss}
\caption{The spectra of $B_N$ for three alphabets satisfying~\eqref{e:special-alpha},
with $k=5$ for $M=6$ and $k=4$ for $M=8$. Each case exhibits a band structure consistent with Conjecture~\ref{l:Weyl-great},
in particular the number of the eigenvalues near the outer circle is exactly $|\mathcal A|^k$.}
\label{f:Weyl-great}
\end{figure}
The alphabets satisfying~\eqref{e:special-alpha} are interesting
in particular because all nonzero singular values of the matrix
$\indic_{\mathcal C_k}\mathcal F_N\indic_{\mathcal C_k}$ are
equal to $(|\mathcal A|/M)^{k/2}$,
as follows from~\eqref{e:hs}. Therefore we expect
that as long as $0,M-1\notin\mathcal A$ and the cutoff~$\chi$ is equal to 1
near the Cantor set $\mathcal C_\infty$ (see Theorem~\ref{t:cutoff-dependence}),
many eigenvalues of the open quantum map $B_{N,\chi}$
will lie near the circle of radius $\sqrt{|\mathcal A|/M}$.
Indeed, if an eigenfunction $u$ with eigenvalue $\lambda$ satisfied the following stronger versions
of~\eqref{e:concentration-3}, \eqref{e:concentration-4}:
$$
\|u\|_{\ell^2_N}=|\lambda|^{-k}\cdot \|\indic_{\mathcal C_k}u\|_{\ell^2_N},\quad
u=\mathcal F_N^*\indic_{\mathcal C_k}\mathcal F_N u,
$$
then we would have $|\lambda|=\sqrt{|\mathcal A|/M}$.

The above heuristical reasoning is supported by numerical evidence.
In fact, in all cases of~\eqref{e:special-alpha} that we computed, eigenfunctions
exhibit a \emph{band structure} not unlike the one of~\cite{nhp,FaureTsujii1,FaureTsujii2,FaureTsujii3}~--
see Figure~\ref{f:Weyl-great}.
The outermost band concentrates strongly near the circle
of radius $\sqrt{|\mathcal A|/M}$ and has exactly $|\mathcal A|^k$ eigenvalues,
which prompts us to make the following
\begin{conj}
  \label{l:Weyl-great}
Assume that the alphabet $\mathcal A$ satisfies~\eqref{e:special-alpha},
$0,M-1\notin\mathcal A$, and $\chi=1$ near $\mathcal C_\infty$.
Then there exists $\nu_0>{1-\delta\over 2}$ such that for each $\varepsilon>0$
and $k$ large enough, we have the following:
\begin{itemize}
\item (Second gap) Every eigenvalue $\lambda$ of $B_{N,\chi}$ in $\{|\lambda|>M^{-\nu_0}\}$
lies in the thin annulus
\begin{equation}
  \label{e:annulus}
\Big\{\sqrt{|\mathcal A|\over M}-\varepsilon<|\lambda|<\sqrt{|\mathcal A|\over M}+\varepsilon\Big\},\quad
\sqrt{|\mathcal A|\over M}=M^{1-\delta\over 2};
\end{equation}
\item (Weyl law) The number of eigenvalues in the annulus~\eqref{e:annulus}
is exactly equal to
$$
|\mathcal A|^k=N^\delta.
$$
\end{itemize}
\end{conj}
\Remark Conjecture~\ref{l:Weyl-great} is true if we use the Walsh quantization
(see~\cite[\S5.1]{NonnenmacherZworskiOQM}) and put $\chi\equiv 1$.
Indeed, as explained in~\cite[Proposition~5.4]{NonnenmacherZworskiOQM}, in that case
the spectrum of the $k$-th power $B_N^k$ is computed explicitly as
$$
\Spec(B_N^k)=\big\{\lambda_1\cdots\lambda_k \mid \lambda_1,\dots,\lambda_k\in \Spec(\indic_{\mathcal A}\mathcal F_M\indic_{\mathcal A})\big\}.
$$
If~\eqref{e:special-alpha} holds, then the matrix $\indic_{\mathcal A}\mathcal F_M\indic_{\mathcal A}$
has $|\mathcal A|$ nonzero eigenvalues, which lie on the circle of radius $\sqrt{|\mathcal A|/M}$.
Therefore $B_N$ has $|\mathcal A|^k$ nonzero eigenvalues, which also lie on that circle.

There exist many solutions $(M,\mathcal A)$ to~\eqref{e:special-alpha}~-- see Table~\ref{b:special-cases}
for a complete list up to $M=24$. 
We do not give a classification of all solutions,
but we provide a few examples and properties (some of which were explained to the authors
by Bjorn Poonen):
\begin{table}
\begin{tabular}{|l|l||l|l||l|l|}
\hline
$M$ & $\mathcal A$ &
$M$ & $\mathcal A$ &
$M$ & $\mathcal A$ \\\hline
6 & $\{0,3\}$ &
6 & $\{0,2,4\}$ &
8 & $\{0,2\}$ \\\hline
8 & $\{0,1\}+\{0,4\}$ &
10 & $\{0, 5\}$ &
10 & $\{0, 2, \dots 8\}$ \\\hline
12 & $\{0, 4, 8\}$ &
12 & $\{0, 2, 4\}$ &
12 & $\{0,3,6,9\}$ \\\hline
12 & $\{0,1\}+\{0,6\}$ &
14 & $\{0,7\}$ &
14 & $\{0,2,\dots,12\}$ \\\hline
15 & $\{0,5,10\}$ &
15 & $\{0,3, \dots, 12\}$ &
16 & $\{0,2,4,6\}$ \\\hline
16 & $\{0,1\}+\{0,8\}$ &
18 & $\{0,9\}$ &
18 & $\{0,3\}$ \\\hline
18 & $\{0, 2, \dots, 16\}$ &
18 & $\{0,1,2\}+\{0,6,12\}$ &
20 & $\{0,5,10,15\}$ \\\hline
20 & $\{0,1\}+\{0,10\}$ &
20 & $\{0,2,\dots,8\}$ &
20 & $\{0,4,\dots,16\}$ \\\hline
20 & $\{0, 2, 4, 8, 16 \}$ &
21 & $\{0, 7, 14\}$ &
21 & $\{0, 3, \dots, 18\}$ \\\hline
22 & $\{0, 11\}$ &
22 & $\{0, 2, \dots, 20\}$ & 
24 & $\{0, 6\}$ \\\hline
24 & $\{0, 8, 16\}$ &
24 & $\{0, 4, 8\}$ &
24 & $\{0, 3\} + \{0, 12\}$ \\\hline
24 & $\{0, 1\} + \{0, 12\}$ &
24 & $\{0, 2, \dots, 10\}$ &
24 & $\{0, 2, 4, 6, 10, 20\}$ \\\hline
24 & $\{0, 2, 4, 8, 10, 18\}$ &
24 & $\{0, 2\}+\{0,8,16\}$ &
24 & $\{0, 1\}+\{0, 6, 12, 18\}$ \\\hline
24 & $\{0, 3, \dots, 21\}$ &
24 & $\{0, 1\}+ \{0, 4, \dots, 20\}$\\
\cline{1-4}
\end{tabular}
\smallskip
\caption{The complete list of solutions to~\eqref{e:special-alpha}
for $M\leq 24$ with $1<|\mathcal A|<M$.
We identify solutions which are related by~\eqref{e:transformer}
and use the following notation:
$a,a+k,\dots,a+\ell k$
denotes the arithmetic progression $\{a+jk\mid 0\leq j\leq \ell\}$
and $\mathcal A+\mathcal B$ denotes the set of all sums of elements
of $\mathcal A$ with elements of $\mathcal B$.}
\label{b:special-cases}
\end{table}
%
\begin{enumerate}
\item If $\mathcal A=\{0,\dots,M-1\}$ or $\mathcal A$ has only one element,
then $(M,\mathcal A)$ solves~\eqref{e:special-alpha},
though these degenerate alphabets are not allowed in the rest of this paper.
\item A basic example of a nondegenerate alphabet solving~\eqref{e:special-alpha}
is given by
$$
M=pq,\quad
\mathcal A=\{jq\mid j=0,\dots,p-1\},
$$
where $p$ is prime and $q>1$ is not divisible by~$p$.
This provides a family of examples with dimensions $\delta$ forming a dense set in $[0,1]$.
\item If $(M,\mathcal A)$ is a solution to~\eqref{e:special-alpha},
$d\in\mathbb N$ is coprime to $M$, and $q\in\mathbb Z$, then
\begin{equation}
  \label{e:transformer}
\big(M,(d\mathcal A+q)\bmod M\big)
\end{equation}
is also a solution. Indeed, the case of $d=1$
follows by direct calculation; it remains to consider the case $d>1,q=0$. In that
case we note that~\eqref{e:special-alpha} can be expressed as a system of polynomial
relations with integer coefficients (depending on~$\mathcal A$) on the root of unity
$\omega_M:=\exp(-2\pi i/M)$. Then the condition~\eqref{e:special-alpha}
for $(M,(d\mathcal A)\bmod M)$ is expressed as the same system of polynomial
relations on $\omega_M^{d^2}$. Since $d^2$ is coprime to $M$,
$\omega_M^{d^2}$ is a Galois conjugate of $\omega_M$ and these two numbers
solve the same polynomial equations with rational coefficients.
\item If $(M,\mathcal A)$ is a solution to~\eqref{e:special-alpha},
and $d\in\mathbb N$, then $(d^2M,d\mathcal A)$ is also a solution to~\eqref{e:special-alpha}.
\end{enumerate}
\Remark We say that a nonempty set
$\mathcal A\subset \mathbb Z_M$ is a \emph{discrete spectral set}
modulo $M$ if there exists $\mathcal B\subset \mathbb Z_M$,
called a \emph{spectrum} for $\mathcal A$,
such that $|\mathcal B|=|\mathcal A|$ and,
with $G_{\mathcal A}$ defined in~\eqref{e:G-A},
$$
G_{\mathcal A}\Big({b-b'\over M}\Big)=0\quad\text{for all }b,b'\in\mathcal B,\
b\neq b'.
$$
Clearly, an alphabet $\mathcal A$ satisfies~\eqref{e:special-alpha}
if and only if it is its own spectrum. One can define a more general
version of the open quantum map~\eqref{e:B-N-def}
which depends on two alphabets $\mathcal A,\mathcal B$ 
of the same size and a bijection $\mathcal A\leftrightarrow\mathcal B$.
If $\mathcal B$ is a spectrum for $\mathcal A$, then we expect
a spectral gap of size $(1-\delta)/2$ as in Proposition~\ref{l:special-alpha},
and Conjecture~\ref{l:Weyl-great} can be extended to these cases.
Regarding the structure of discrete spectral sets,
the following is a discrete analogue of a conjecture
made by Fuglede~\cite{Fuglede}, which
we verified numerically for all $M\leq 20$:
\begin{conj}
  \label{c:spectral-sets}
A set $\mathcal A\subset\mathbb Z_M$ is a discrete spectral set if and only
if it tiles $\mathbb Z_M$ by translations, that is there exists $T\subset \mathbb Z_M$
such that
$$
\mathbb Z_M=\bigsqcup_{t\in T} (\mathcal A+t)\bmod M.
$$
\end{conj}
\L aba~\cite{Laba} proved that if $M$ has at most two distinct prime factors
and $\mathcal A$ tiles $\mathbb Z_M$ by translations,
then $\mathcal A$ is a spectral set.
We refer the reader to~\cite{Laba,Dutkai-Lai,MaKo} for an overview of recent results on spectral sets.

\subsection{Upper bounds on the fractal uncertainty exponent}
\label{s:improve-all}

We finally present some asymptotic lower bounds on the norm $r_k$ from~\eqref{e:r-k-again}, or equivalently upper
bounds on the fractal uncertainty exponent $\beta$ defined in~\eqref{e:beta-limit}.
While an upper bound on $\beta$ does not imply a lower bound on the spectral
radius of $B_N$ (to prove the latter, one would have to show
existence of eigenvalues in a fixed annulus for a nonselfadjoint operator, which 
is notoriously difficult), numerics seem to indicate that
at least in some cases the value of $\beta$ gives a good approximation
to the spectral radius~-- see Figures~\ref{f:fup-howgood}, \ref{f:improved-pressure}, \ref{f:improved-0}, and~\ref{f:Weyl-great}.

Our bounds are based on an idea suggested by Hong Wang
(see also~\cite[\S6]{Wiener})
and use the following formula for the Fourier transform of the indicator function of~$\mathcal C_k$
in terms of the function $G_{\mathcal A}$ defined in~\eqref{e:G-A}:
\begin{equation}
  \label{e:fc-split}
\mathcal F_N(\mathbf 1_{\mathcal C_k})(j)=\prod_{s=1}^{k}G_{\mathcal A}\Big({j\over M^s}\Big).
\end{equation}
We first show that unless a slightly weaker version of~\eqref{e:special-alpha}
holds (we believe that this version is equivalent to~\eqref{e:special-alpha},
and we checked it numerically for all $M\leq 25$), the exponent
$\beta$ has to be strictly smaller than the upper bound~\eqref{e:JN-2}:
\begin{prop}
  \label{l:not-the-best}
Assume that there exist
\begin{equation}
  \label{e:not-the-best}
b,b'\in\mathcal A,\quad
b\neq b',\quad
G_{\mathcal A}\Big({b'-b\over M}\Big)G_{\mathcal A}\Big({b'-b\over M^2}\Big)\neq 0.
\end{equation}
Then $\beta<(1-\delta)/2$.
\end{prop}
\begin{proof}
Without loss of generality we assume $b'>b$.
(Note $G_{\mathcal A}(-x)=\overline{G_{\mathcal A}(x)}$.)
Put
$$
j_b:=\sum_{s=0}^{M-1} bM^s\in\mathcal C_k,
$$
and consider $u\in\ell^2_N$, $\supp u\subset \mathcal C_k$, defined by
$$
u(\ell)=\exp\Big({2\pi i j_b\ell\over N}\Big),\
\ell\in\mathcal C_k;\quad
\|u\|_{\ell^2_N}=|\mathcal A|^{k/2}.
$$
Then
\begin{equation}
  \label{e:glot}
r_k^2\geq
\bigg({\|\indic_{\mathcal C_k}\mathcal F_Nu\|_{\ell^2_N}\over \|u\|_{\ell^2_N}}\bigg)^2=\sum_{j\in\mathcal C_k}
{|\mathcal F_N(\mathbf 1_{\mathcal C_k})(j-j_b)|^2\over |\mathcal A|^k}
\end{equation}
and by~\eqref{e:fc-split},
\begin{equation}
  \label{e:produce}
{|\mathcal F_N(\mathbf 1_{\mathcal C_k})(j-j_b)|^2\over |\mathcal A|^k}={1\over |\mathcal A|^k}\prod_{s=1}^k\Big|G_{\mathcal A}\Big({j-j_b\over M^s}\Big)\Big|^2.
\end{equation}
We next estimate $|G_{\mathcal A}|^2$ from below. 
First of all, we have for all $s\in\mathbb N$,
\begin{equation}
  \label{e:never-the-best}
G_{\mathcal A}\Big({b'-b\over M^s}\Big)\neq 0.
\end{equation}
Indeed, the case of $s=1,2$ follows from~\eqref{e:not-the-best}. For $s\geq 3$
we have
$$
0 <{b'-b\over M^s}< {1\over 2M}.
$$
Recall that
$$
G_{\mathcal A}\Big({b'-b\over M^s}\Big)={1\over\sqrt M}\sum_{a\in\mathcal A} e^{-i\theta_a},\quad
\theta_a:=2\pi{a(b'-b)\over M^s}
$$
and the values $\theta_a$ lie in the interval $[0,\pi)$, implying that the sum of $e^{-i\theta_a}$
cannot be 0.

Next, we have as $s\to\infty$, uniformly in $x\in [0,M]$
$$
G_{\mathcal A}\Big({x\over M^s}\Big)={|\mathcal A|\over\sqrt M}+\mathcal O(M^{-s}).
$$
Therefore there exist constants $\rho>0$, $L\in\mathbb N$ depending on $\mathcal A$ such that
\begin{equation}
  \label{e:product-yay}
{1\over |\mathcal A|^\ell}\prod_{s=1}^\ell \Big|G_{\mathcal A}\Big({x\over M^s}\Big)\Big|^2\geq
\rho \Big({|\mathcal A|\over M}\Big)^\ell\quad\text{for all }
\ell\in\mathbb N,\
x\in [b'-b,b'-b+M^{-L}].
\end{equation}
For $q\in\mathbb N$, consider the following
subset of $\mathcal C_k$:
$$
\Omega_q=\bigg\{j_b+\sum_{p=1}^q (b'-b)M^{s_p}\colon
0\leq s_1,\dots,s_q<k,\
s_{p+1}-s_p\geq L+1\text{ for }1\leq p<q\bigg\}.
$$
Splitting the product in~\eqref{e:produce} into intervals $s\in [s_p+1,s_{p+1}]$
and using~\eqref{e:product-yay}, we see that
$$
{|\mathcal F_N(\mathbf 1_{\mathcal C_k})(j-j_b)|^2\over |\mathcal A|^k}
\geq \rho^q \Big(
{|\mathcal A|\over M}\Big)^k
\quad\text{for all }j\in\Omega_q.
$$
Moreover, for $q\leq {k\over 2L}$ and $k$ large enough the size of $\Omega_q$ is
$$
|\Omega_q|={\big(k-(q-1)L\big)!\over q! (k-qL)!}
\geq {1\over q!}\Big({k\over 2}\Big)^q.
$$
Therefore by~\eqref{e:glot}
$$
\begin{aligned}
r_k^2
&\geq \sum_{q=1}^{\lfloor{k\over 2L}\rfloor}
\sum_{j\in\Omega_q}
{|\mathcal F_N(\mathbf 1_{\mathcal C_k})(j-j_b)|^2\over |\mathcal A|^k}\\
&\geq \Big(
{|\mathcal A|\over M}\Big)^k\cdot \sum_{q=1}^{\lfloor{k\over 2L}\rfloor}{1\over q!}\Big({\rho k\over 2}\Big)^q\geq \Big({|\mathcal A|\over M}\Big)^k e^{\tilde\rho k} 
\end{aligned}
$$
for $\tilde\rho:={1\over 4}\min(\rho,L^{-1})>0$, where in the last inequality we used the following
corollary of Stirling's formula valid for sufficiently large $n\in\mathbb N$:
$$
\sum_{q=1}^n {1\over q!}n^q\geq {n^n\over n!}\geq e^{n/2}.
$$
It follows that
$$
\beta\leq {1-\delta\over 2}-{\tilde \rho\over 2\log M}<{1-\delta\over 2},
$$
finishing the proof.
\end{proof}
We finish this section by considering a particular family of alphabets
with $\delta\leq {1\over 2}+\mathcal O({1\over \log M})$
and fractal uncertainty exponent $\beta$ which is close to ${1\over 2}-\delta$.
This shows that the lower bound of Proposition~\ref{l:improved-ae-ultimate}
is sharp; see also the remark following Corollary~\ref{l:improve-pressure-cantor}.
\begin{prop}
  \label{l:lower-1/2}
Assume that $2\leq L\leq 2\sqrt{M}$, take $a_0\in [0,M-L]$, and consider the alphabet
$$
\mathcal A=a_0+\{0,1,\dots,L-1\}\subset\mathbb Z_M.
$$
Then for some global constant $\mathbf K$, the exponent $\beta$ defined in~\eqref{e:beta-limit}
is bounded above as follows:
\begin{equation}
  \label{e:lower-1/2}
\beta\leq {1\over 2}-\delta+{\mathbf KL^2\over M\log M}.
\end{equation}
\end{prop}
\begin{proof}
By shifting the alphabet $\mathcal A$ (see the remark following~\eqref{e:same-shift}),
we may assume that $a_0=0$.
We use~\eqref{e:fc-split}, calculating
$$
G_{\mathcal A}(x)={\exp(-2\pi iLx)-1\over \sqrt{M}(\exp(-2\pi ix)-1)},\quad
x\in\mathbb R.
$$
Put
$$
u:=\mathbf 1_{\mathcal C_k},\quad
\|u\|_{\ell^2_N}=L^{k/2}.
$$
Then
\begin{equation}
  \label{e:hodor}
r_k^2\geq \bigg({\|\indic_{\mathcal C_k}\mathcal F_N u\|_{\ell^2_N}\over \|u\|_{\ell^2_N}}\bigg)^2
={1\over L^k}\sum_{j\in\mathcal C_k}|\mathcal F_N(\mathbf 1_{\mathcal C_k})(j)|^2.
\end{equation}
Define the set $\Omega\subset\mathcal C_k$ by
$$
\Omega=\bigg\{\sum_{r=0}^{k-1}b_r M^r\colon b_0,\dots,b_{k-1}\in \Big\{0,\dots,\Big\lfloor {L\over 9}\Big\rfloor\Big\}\bigg\}.
$$
Since $L\leq 2\sqrt{M}$, we have for all $s\in \mathbb N$ and $j\in\Omega$,
$$
{j\over M^s}\in J+\Big[0,{1\over 2L}\Big]\quad\text{for some }J\in\mathbb Z.
$$
Then for some global constant $\mathbf K$, we have
$$
\Big|G_{\mathcal A}\Big({j\over M^s}\Big)\Big|\geq {L\over \mathbf K\sqrt M}\quad\text{for all }
s\in\mathbb N,\
j\in\Omega,
$$
and thus by~\eqref{e:fc-split}
$$
|\mathcal F_N(\mathbf 1_{\mathcal C_k})(j)|^2\geq {L^{2k}\over \mathbf K^{2k}M^k}\quad\text{for all }
j\in\Omega.
$$
It follows from~\eqref{e:hodor} that
$$
r_k^2
\geq {1\over L^k}\sum_{j\in\Omega}|\mathcal F_N(\mathbf 1_{\mathcal C_k})(j)|^2
\geq  {L^{2k}\over (3\mathbf K)^{2k}M^k}.
$$
This gives
$$
\beta\leq {1\over 2}-\delta+{\log(3\mathbf K)\over \log M}.
$$
This implies the bound~\eqref{e:lower-1/2} as long as $L\geq c\sqrt{M}$ for any
fixed $c>0$.

It remains to consider the case when $L/\sqrt{M}$ is small.
We compute
$$
G_{\mathcal A}(J+x)={L\over\sqrt{M}}(1+\mathcal O(Lx)),\quad
x\in\Big[0,{1\over 2L}\Big],\quad
J\in\mathbb Z
$$
with the constant in $\mathcal O(\cdot)$ independent of~$L,x,J$. Thus
$$
\Big|G_{\mathcal A}\Big({j\over M^s}\Big)\Big|\geq {L\over\sqrt{M}}\Big(1-{\mathbf KL^2\over M}\Big)
\quad\text{for all }s\in\mathbb N,\ j\in\mathcal C_k,
$$
implying by~\eqref{e:fc-split} and~\eqref{e:hodor}
$$
r_k^2\geq {L^{2k}\over M^k}\Big(1-{\mathbf K L^2\over M}\Big)^{2k}.
$$
This proves~\eqref{e:lower-1/2} when $L/\sqrt{M}$ is small enough depending on $\mathbf K$.
\end{proof}

\section{Weyl bounds}
  \label{s:Weyl}
  
In this section, we prove Theorem \ref{t:weyl} following the argument of~\cite{ifwl}.
We fix
$$
\tilde \nu > \nu > 0
$$
and put
$$
\Omega:=\{M^{-\tilde\nu}<|\lambda|<3\}\subset\mathbb C.
$$

\subsection{An approximate inverse}

Similarly to~\S\ref{s:localization}, fix
$$
\rho,\rho'\in (0,1),\quad
\tilde k:=\lceil\rho k\rceil,\quad
\tilde k':=\lceil\rho' k\rceil.
$$
Let the sets
\begin{equation}
  \label{e:X-rho-new}
X:=X_{\rho},\
X':=X_{\rho'}\ \subset\ \mathbb Z_N
\end{equation}
be defined in~\eqref{e:X-rho}.
We construct an approximate inverse for $B_N-\lambda$, similarly to~\cite[Proposition~2.1]{ifwl}:
\begin{lemm}
\label{l:para-id}
There exist families of operators $\mathcal{J}(\lambda),\mathcal{Z}(\lambda),\mathcal{R}(\lambda):\ell^2_N\to\ell^2_N$ holomorphic in $\lambda\in\Omega$, such that uniformly in $\lambda\in\Omega$
\begin{align}
\label{e:para-j}
\|\mathcal{J}(\lambda)\|_{\ell^2_N\to\ell^2_N}&\leq |\lambda|^{-\tilde{k}},\\
\label{e:para-z}
\|\mathcal{Z}(\lambda)\|_{\ell^2_N\to\ell^2_N}&=\mathcal{O}(N^{\tilde \nu(\rho+\rho')}),\\
\label{e:para-error}
\|\mathcal{R}(\lambda)\|_{\ell^2_N\to\ell^2_N}&=\mathcal{O}(N^{-\infty})
\end{align}
and the following identity holds:
\begin{equation}
\label{e:para-id}
\indic=\mathcal{J}(\lambda)\indic_{X}\mathcal{F}_N^\ast\indic_{X'}\mathcal{F}_N+\mathcal{Z}(\lambda)(B_N-\lambda)+\mathcal{R}(\lambda).
\end{equation}
\end{lemm}
\begin{proof}
We have the following identities:
\begin{align}
  \label{e:para-id-1}
\indic&=\lambda^{-\tilde k}(B_N)^{\tilde k}\indic_{X}+\mathcal{Z}_1(\lambda)(B_N-\lambda)+\mathcal{R}_1(\lambda),\\
  \label{e:para-id-2}
\indic&=\mathcal{F}_N^\ast\indic_{X'}\mathcal{F}_N+\mathcal{Z}_2(\lambda)(B_N-\lambda)+\mathcal{R}_2(\lambda)
\end{align}
where
$$
\begin{aligned}
\mathcal{Z}_1(\lambda)=&\;-\sum_{0\leq \ell<\tilde k}\lambda^{-1-\ell}(B_N)^\ell=\mathcal{O}(N^{\rho\tilde \nu})_{\ell^2_N\to\ell^2_N},\\
\mathcal{Z}_2(\lambda)=&\;-\sum_{0\leq\ell <\tilde k'}\lambda^{-1-\ell}(\indic-\mathcal F_N^*\indic_{X'}\mathcal F_N)(B_N)^\ell=\mathcal{O}(N^{\rho'\tilde \nu})_{\ell^2_N\to\ell^2_N}
\end{aligned}
$$
and by~\eqref{e:concentration-1}, \eqref{e:concentration-2}
\begin{equation*}
\begin{split}
\mathcal{R}_1(\lambda)=&\;\lambda^{-\tilde{k}}(B_N)^{\tilde{k}}(\indic-\indic_{X})
=\mathcal O(N^{-\infty})_{\ell^2_N\to\ell^2_N},\\
\mathcal{R}_2(\lambda)=&\;\lambda^{-\tilde{k}'}(\indic-\mathcal{F}_N^\ast\indic_{X'}\mathcal{F}_N)(B_N)^{\tilde{k}'}
=\mathcal O(N^{-\infty})_{\ell^2_N\to\ell^2_N}.
\end{split}
\end{equation*}
Then~\eqref{e:para-id} holds with 
$$
\begin{aligned}
\mathcal{J}(\lambda)&=\lambda^{-\tilde{k}}(B_N)^{\tilde{k}},\\
\mathcal{Z}(\lambda)&=\mathcal Z_1(\lambda)+
\lambda^{-\tilde{k}}(B_N)^{\tilde k}\indic_X\mathcal{Z}_2(\lambda),\\
\mathcal{R}(\lambda)&=\mathcal{R}_1(\lambda)+\lambda^{-\tilde{k}}(B_N)^{\tilde k}\indic_X\mathcal{R}_2(\lambda).\qedhere
\end{aligned}
$$
\end{proof}
We remark that Proposition~\ref{l:para-id} gives a resolvent
bound inside the spectral gap given by the uncertainty principle:
\begin{prop}
\label{l:resolvent-bound}
Let $\beta$ be defined in~\eqref{e:beta-limsup}. Then for each $\nu\in (0,\beta)$
we have
$$
\|(B_N-\lambda)^{-1}\|_{\ell^2_N\to \ell^2_N}\leq CN^{2\nu}\quad\text{when}\quad
M^{-\nu}\leq |\lambda|\leq 1,\quad
N\geq N_0,
$$
with the constants $N_0,C$ depending on $\nu$.
\end{prop}
\begin{proof}
Take $\rho'=\rho$. By~\eqref{e:fupuse-2},
\eqref{e:para-j}, and~\eqref{e:para-error} we have for large $N$
$$
\|\mathcal J(\lambda)\indic_{X}\mathcal F_N^*\indic_{X}\mathcal F_N+\mathcal R(\lambda)\|
\leq C|\lambda|^{-\rho k}N^{2(1-\rho)}r_k\leq CN^{\rho\nu+2(1-\rho)-\beta+}\leq
{1\over 2}
$$
assuming that $\rho\in (0,1)$ satisfies the inequality
$-\rho\nu+2(\rho-1)+\beta>0$. It follows from~\eqref{e:para-id}
and~\eqref{e:para-z}
that
$$
\|(B_N-\lambda)^{-1}\|_{\ell^2_N\to\ell^2_N}
\leq 2\|\mathcal Z(\lambda)\|_{\ell^2_N\to \ell^2_N}
\leq C N^{2\nu},
$$
finishing the proof.
\end{proof}

\subsection{Proof of Theorem~\texorpdfstring{\ref{t:weyl}}{3}}

Fix $\rho,\rho'\in (0,1)$.
Using Lemma~\ref{l:para-id}, define for $\lambda\in\Omega$
\begin{align}
\mathcal{B}(\lambda)&=\mathcal{J}(\lambda)\indic_{X}\mathcal{F}_N^\ast\indic_{X'}\mathcal{F}_N+\mathcal{R}(\lambda)=\indic-\mathcal{Z}(\lambda)(B_N-\lambda),\\
  \label{e:F-def}
F(\lambda)&=\det(\indic-\mathcal B(\lambda)^2).
\end{align}
Then we have
$$
F(\lambda)=\det(B_N-\lambda)\cdot \det \mathcal Z(\lambda)\cdot \det (1+\mathcal B(\lambda))
$$
and thus (with both sets counting multiplicity) 
\begin{equation}
\label{e:zero}
\Sp(B_N)\cap\Omega\ \subset\ \{\lambda\in\Omega\colon F(\lambda)=0\}.
\end{equation}
We have the following estimates on the determinant $F(\lambda)$:
\begin{lemm}
\label{l:det-estimate}
Fix $\rho,\rho'\in (0,1)$. Then there exists a constant $C$ such that
for all $N$ large enough,
\begin{align}
  \label{e:det-estimate-1}
\sup_{\lambda\in\Omega}|F(\lambda)|&\leq \exp(CN^{\widetilde m}),\\
  \label{e:det-estimate-2}
|F(2)|&\geq \exp(-CN^{\widetilde m})
\end{align}
where
\begin{equation}
  \label{e:tilde-m}
\widetilde m=\delta(\rho+\rho')+1-\rho-\rho'+2\rho\tilde\nu.
\end{equation}
\end{lemm}
\begin{proof}
We have by~\eqref{e:hs} and~\eqref{e:X-rho-new} the Hilbert--Schmidt norm bound
$$
\|\indic_X\mathcal F_N\indic_{X'}\|_{\HS}^2={|X|\cdot |X'|\over N}\leq 
CN^{\delta(\rho+\rho')+1-\rho-\rho'}.
$$
Therefore,
\begin{equation}
  \label{e:hs-gotit}
\sup_{\lambda\in\Omega} \|\mathcal B(\lambda)\|_{\HS}^2\leq CN^{\widetilde m}.
\end{equation}
We estimate the determinant using the trace norm, 
\begin{equation}
  \label{e:det-gotit}
|F(\lambda)|\leq \exp\big( \|\mathcal B(\lambda)^2\|_{\tr}\big)\leq\exp\big(\|\mathcal B(\lambda)\|_{\HS}^2\big)
\end{equation}
and~\eqref{e:det-estimate-1} follows.

Next, we have for $\lambda=2$ and sufficiently large $k$,
$$
\|\mathcal B(2)\|_{\ell^2_N\to\ell^2_N}
\leq \|\mathcal J(2)\|_{\ell^2_N\to\ell^2_N}+\mathcal O(N^{-\infty})
\leq 2^{-\rho k}+\mathcal O(N^{-\infty})\leq {1\over 2}.
$$
Therefore
$$
\|(\indic-\mathcal B(2)^2)^{-1}\|_{\ell^2_N\to\ell^2_N}\leq 2.
$$
We have
$$
\begin{gathered}
|F(2)|^{-1}=\big|\det\big((\indic-\mathcal B(2)^2)^{-1}\big)\big|\\
=\big|\det\big(\indic+\mathcal B(2)^2(\indic-\mathcal B(2)^2)^{-1}\big)\big|
\leq \exp(2\|\mathcal B(2)\|_{\HS}^2)
\end{gathered}
$$
thus~\eqref{e:det-estimate-2} follows from~\eqref{e:hs-gotit}.
\end{proof}
To pass from the estimates~\eqref{e:det-estimate-1}, \eqref{e:det-estimate-2}
to bounding the number of zeros of $F(\lambda)$, we
need the following general statement from complex analysis:
\begin{lemm}
  \label{l:jensen}
Assume that $z_0\in\mathbb C$, $\Omega\subset\mathbb C$ is a connected open set,
and $K\subset\mathbb C$ is a compact set such that
$$
z_0\ \in\ K\ \subset\ \Omega.
$$
Let $f(z)$ be a holomorphic function on $\Omega$ such that
for some constant $L>0$,
\begin{equation}
  \label{e:jensen-cond}
\sup_{z\in \Omega}|f(z)|\leq e^L,\quad
|f(z_0)|\geq e^{-L}.
\end{equation}
Then the number of zeros of $f(z)$ in $K$, counted with multiplicities, is bounded as follows:
\begin{equation}
  \label{e:jensen-bound}
\big|\{z\in K\colon f(z)=0\}\big|\ \leq\ CL
\end{equation}
where the constant $C$ depends only on $z_0,\Omega,K$.
\end{lemm}
\begin{proof} By splitting $K$ into a union of smaller sets and shrinking $\Omega$
accordingly, we may assume that $\Omega$ is simply connected. Using Riemann
mapping theorem, we then reduce to the case when $\Omega=\{|z|<1\}$ is the unit disk
and $z_0=0$. We may then assume that $K=\{|z|\leq \alpha\}$ is the
closed disk of some radius $\alpha\in (0,1)$.

We now use Jensen's formula (see for instance~\cite[\S3.61, equation (2)]{TitchFunctions}):
$$
\sum_{z\colon f(z)=0,\ |z|<r} \log \Big({r\over |z|}\Big)=
{1\over 2\pi}\int_0^{2\pi}\log |f(re^{i\theta})|\,d\theta-\log |f(0)|\quad\text{for all }r\in (0,1).
$$
where the roots of $f$ are repeated in the sum according to multiplicity.
Taking $r:={1+\alpha\over 2}$ and using~\eqref{e:jensen-cond},
we obtain~\eqref{e:jensen-bound} with $C=2/\log(r/\alpha)$.
\end{proof}
Applying Lemma~\ref{l:jensen} to the function $F$ defined
in~\eqref{e:F-def},
with $z_0=2$, $K=\{M^{-\nu}\leq |\lambda|\leq 2\}$, and $L=CN^{\widetilde m}$, and using Lemma~\ref{l:det-estimate}
and~\eqref{e:zero}, we get the following
bound on the counting function defined in~\eqref{e:N-k}:
$$
\mathcal N_k(\nu)\leq CN^{\widetilde m}.
$$
To prove Theorem~\ref{t:weyl}, it remains to show that by choosing
$\tilde\nu,\rho,\rho'$, we can make the constant
$\widetilde m$ defined in~\eqref{e:tilde-m} arbitrarily close
to
$$
m=\min(2\nu+2\delta-1,\delta).
$$
This follows from the following two statements:
$$
\begin{aligned}
\tilde\nu=\nu+\varepsilon,\quad
\rho=\rho'=1-\varepsilon\quad&\Longrightarrow\quad
\widetilde m\to 2\nu+2\delta-1\quad \text{as }\varepsilon\to 0+,\\
\tilde\nu=\nu+\varepsilon,\quad
\rho=\varepsilon,\quad \rho'=1-\varepsilon\quad&\Longrightarrow\quad
\widetilde m\to \delta\quad\text{as }\varepsilon\to 0+.\\
\end{aligned}
$$

\section{Independence of cutoff}
  \label{s:cutoff-dependence}
  
To prepare for the proof of Theorem~\ref{t:cutoff-dependence}, we introduce a family of open quantum baker's maps with different cutoffs on the physical side and the Fourier side.
More precisely, for $\chi,\chi'\in C_0^\infty((0,1);[0,1])$ we define the following generalization of~\eqref{e:B-N-def}:
\begin{equation}
B_{N,\chi,\chi'}=\mathcal F_N^*\begin{pmatrix}
\chi_{N/M}\mathcal F_{N/M}\chi'_{N/M} & \\
&\ddots&\\
 & &\chi_{N/M}\mathcal F_{N/M}\chi'_{N/M}
\end{pmatrix}I_{\mathcal A, M}.
\end{equation}
The advantage of this family is that it is bilinear in $\chi$ and $\chi'$.

Following the proof of Proposition~\ref{l:egorov-long}, we see
that propagation of singularities for long times also holds for powers of $B_{N,\chi,\chi'}$, or more generally for products of the form
\begin{equation}
B_{N,\chi_1,\chi_1'}B_{N,\chi_2,\chi_2'}\cdots B_{N,\chi_{\tilde{k}},\chi_{\tilde{k}}'}
\end{equation}
and the constants in $\mathcal O(N^{-\infty})$ are uniform as long as only finitely many different cut-offs appear in the product.

Let $\rho\in(0,1)$ and $\tilde{k}$, $\mathcal{X}_\rho$, $X_\rho$ be as in \S \ref{s:localization}, then we have the following slightly more general version of~\eqref{e:concentration-1} and~\eqref{e:concentration-2},
$$
(B_{N,\chi',\chi})^j(B_{N,\chi'})^{\tilde{k}-1-j}B_{N,\chi',\chi'-\chi}=(B_{N,\chi',\chi})^j(B_{N,\chi'})^{\tilde{k}-1-j}B_{N,\chi',\chi'-\chi}\indic_{X_\rho}+\mathcal{O}(N^{-\infty})_{\ell^2_N\to\ell^2_N}
$$
uniformly for $0\leq j\leq \tilde{k}-1$, and
$$
B_{N,\chi'-\chi,\chi'}(B_{N,\chi})^{\tilde{k}-1}=\mathcal{F}_N^\ast\indic_{X_\rho}\mathcal{F}_NB_{N,\chi'-\chi,\chi'}(B_{N,\chi})^{\tilde{k}-1}+\mathcal{O}(N^{-\infty})_{\ell^2_N\to\ell^2_N}.
$$
If $\supp(\chi'-\chi)\cap \mathcal C_\infty=\emptyset$ and
$\chi,\chi'$ are independent of $N$, then for $N$ large enough
$$
B_{N,\chi',\chi'-\chi}\indic_{X_\rho}=
\mathcal{F}_N^\ast\indic_{X_\rho}\mathcal{F}_NB_{N,\chi'-\chi,\chi'}=0.
$$
Therefore we have uniformly for $0\leq j\leq \tilde{k}-1$
\begin{align}
  \label{e:rapid-1}
(B_{N,\chi',\chi})^j(B_{N,\chi'})^{\tilde{k}-1-j}B_{N,\chi',\chi'-\chi}&=\mathcal{O}(N^{-\infty})_{\ell^2_N\to\ell^2_N},\\
  \label{e:rapid-2}
B_{N,\chi'-\chi,\chi}(B_{N,\chi})^{\tilde{k}-1}&=\mathcal{O}(N^{-\infty})_{\ell^2_N\to\ell^2_N}.
\end{align}
We are now ready to give
\begin{proof}[Proof of Theorem~\ref{t:cutoff-dependence}]
Let $u\in \ell^2_N$, $\|u\|_{\ell^2_N}=1$, be an eigenfunction
of $B_{N,\chi}$ with eigenvalue~$\lambda$.
We first show that $u$ is a quasimode for $B_{N,\chi',\chi}$. To see this, we write
$$
(B_{N,\chi',\chi}-\lambda)u=B_{N,\chi'-\chi,\chi}u
=\lambda^{1-\tilde{k}}B_{N,\chi'-\chi,\chi}(B_{N,\chi})^{\tilde{k}-1}u.
$$
Since $|\lambda|^{1-\tilde{k}}=\mathcal{O}(M^{\rho\nu k})=\mathcal{O}(N^{\rho\nu})$, we see
by~\eqref{e:rapid-2}
\begin{equation}
\label{e:quasi-1}
\|(B_{N,\chi',\chi}-\lambda)u\|_{\ell^2_N}=\mathcal{O}(N^{-\infty}).
\end{equation}
Next, put
$$
v:={w\over \|w\|_{\ell^2_N}},\quad
w:=(B_{N,\chi'})^{\tilde{k}}\,u.
$$
We show that $w$ is a quasimode for $B_{N,\chi'}$:
\begin{equation}
\label{e:quasi-2}
\|(B_{N,\chi'}-\lambda)w\|_{\ell^2_N}=\mathcal{O}(N^{-\infty}).
\end{equation}
To see this, we write
$$
(B_{N,\chi'}-\lambda)w=(B_{N,\chi'}-\lambda)(B_{N,\chi'})^{\tilde{k}}u=(B_{N,\chi'})^{\tilde{k}}(B_{N,\chi'}-\lambda)u.
$$
Replacing $\lambda u$ by $B_{N,\chi',\chi}u$
using~\eqref{e:quasi-1}, we obtain
$$
\|(B_{N,\chi'}-\lambda)w\|_{\ell^2_N}\leq \|(B_{N,\chi'})^{\tilde{k}}B_{N,\chi',\chi'-\chi}u\|_{\ell^2_N}+\mathcal{O}(N^{-\infty}).
$$
Now \eqref{e:quasi-2} follows from \eqref{e:rapid-1} with $j=0$.

It remains to give a polynomial lower bound for $\|w\|_{\ell^2_N}$. Multiplying $(B_{N,\chi',\chi})^{\tilde{k}}$ to the left of the telescoping identity
$$
(B_{N,\chi'})^{\tilde{k}}=(B_{N,\chi',\chi})^{\tilde{k}}
+\sum_{j=0}^{\tilde{k}-1}(B_{N,\chi'})^{\tilde{k}-j-1}B_{N,\chi',\chi'-\chi}(B_{N,\chi',\chi})^j,
$$
and using \eqref{e:rapid-1} (together with the fact that $\tilde{k}=\mathcal{O}(\log N)$), we get
\begin{equation*}
\begin{split}
(B_{N,\chi',\chi})^{\tilde{k}}(B_{N,\chi'})^{\tilde{k}}
=&\;(B_{N,\chi',\chi})^{2\tilde{k}}
+\sum_{j=0}^{\tilde{k}-1}(B_{N,\chi',\chi})^{\tilde{k}}(B_{N,\chi'})^{\tilde{k}-j-1}B_{N,\chi',\chi'-\chi}(B_{N,\chi',\chi})^j\\
=&\;(B_{N,\chi',\chi})^{2\tilde{k}}+\mathcal{O}(N^{-\infty})_{\ell^2_N\to\ell^2_N}.
\end{split}
\end{equation*}
Therefore we have
$$
\|w\|_{\ell^2_N}=\|(B_{N,\chi'})^{\tilde{k}}u\|_{\ell^2_N}
\geq \|(B_{N,\chi',\chi})^{\tilde{k}}(B_{N,\chi'})^{\tilde{k}}u\|_{\ell^2_N}\geq \|(B_{N,\chi',\chi})^{2\tilde{k}}u\|_{\ell^2_N}-\mathcal{O}(N^{-\infty}).
$$
Using \eqref{e:quasi-1}, we get for some $c>0$ independent of $N$,
$$
\|w\|\geq |\lambda|^{2\tilde{k}}-\mathcal{O}(N^{-\infty})
\geq cN^{-2\nu\rho}.
$$
This finishes the proof.
\end{proof}

\section{Remarks on numerical experiments}
  \label{s:numerics}

In this section we describe the numerical experiments used to produce
the figures in this paper. Our numerics and plots were made using MATLAB, version R2015b.

We use the cutoff function $\chi_\tau\in C_0^\infty((0,1);[0,1])$ depending on a
parameter $\tau\in (0,1/2]$ and defined by
$$
\chi_\tau(x)=F\Big({x\over\tau}\Big)F\Big({1-x\over\tau}\Big),\quad
F(x)=c\int_{-\infty}^{1.02\cdot x-0.01} \indic_{[0,1]}(t)\,e^{-{1\over t(1-t)}}\,dt
$$
where $c>0$ is chosen so that $F(x)=1$ for $x\gg 1$.
This function satisfies in particular
$$
\chi_{\tau}=1\quad\text{near }[\tau,1-\tau].
$$
We compute the eigenvalues of the matrices $B_{N,\chi}$ from~\eqref{e:B-N-def}
using the \texttt{eig()} function. To speed up computation
we remove the zero columns of the matrices~$B_{N,\chi}$ and the corresponding rows
(we call the resulting matrix \emph{trimmed});
this of course does not change the nonzero eigenvalues.
In Figures~\ref{f:fup-howgood}, \ref{f:Weyl}, \ref{f:improved-pressure}, \ref{f:improved-0}, and \ref{f:Weyl-great} we use $\tau=0.05$. In Figure~\ref{f:cutoff-dependence} we use the cutoffs
$\chi_j=\chi_{\tau_j}$ with $\tau_1=0.05,\tau_2=0.2,\tau_3=0.5$.

To test the stability of the eigenvalue computations, we also compute the
spectrum of the perturbed matrix
$B_{N,\chi}+\varepsilon Q$ where the entries of $Q$ are independent random variables
distributed uniformly in $[0,1]$ and $\varepsilon>0$ is chosen so that
$\|\varepsilon Q\|_{\ell^2_N\to \ell^2_N}=10^{-4}\cdot\|B_{N,\chi}\|_{\ell^2_N\to\ell^2_N}$.
(To speed up the computation we actually
perturb the trimmed matrix, see the previous paragraph.)
In Figures~\ref{f:fup-howgood}, \ref{f:improved-pressure}, \ref{f:improved-0}, \ref{f:Weyl-great} we plot the spectra of the original and the perturbed matrix,
noting that the two are very close to each other in the annuli of interest.
This indicates a lack of strong pseudospectral effects in annuli.

In Figures~\ref{f:fup-howgood}, \ref{f:cutoff-dependence},
\ref{f:improved-pressure}, \ref{f:improved-0}, and~\ref{f:Weyl-great} the outermost circle is the unit circle.
On each of these figures we also plot the circles of radii $M^{-\beta}$ for some of the following values of $\beta$:
\begin{enumerate}
\item[FUP:] the spectral radius bound of Theorem~\ref{t:gap-detailed}
where we replace the fractal uncertainty exponent $\beta$
with its approximation $\beta_k=-\log r_k/(k\log M)$ for the same
value of $k$ as used in the open quantum map;
\item[$P(1/2):$] the pressure bound, corresponding to
$\beta={1\over 2}-\delta$;
\item[$P(1)/2:$] the classical escape rate, corresponding
to $\beta={1-\delta\over 2}$.
\end{enumerate}
Here the norm $r_k$ from~\eqref{e:r-k} is computed numerically
using the \texttt{norm()} function. By~\eqref{e:beta-limit},
$\beta_k$ gives a lower bound for the limit $\beta$.
Moreover in the considered cases the sequence $\beta_k$
appears almost constant for $k\geq 3$, indicating that
$\beta_k$ is actually a reasonable approximation for $\beta$.

In Figure~\ref{f:fupcloud} we plot the points $(\delta,\beta_k)$
for all $M=3,\dots,10$ and all alphabets $\mathcal A$
with $1<|\mathcal A|<M$. Here for each $\mathcal A$
we take the largest $k$ such that
$|\mathcal A|^k\leq 5000$.


\appendix

\section{Proof of Proposition \texorpdfstring{\ref{l:ub-rho}}{3.11}}
\label{s:additive-combinatorics}

\subsection{Additive energies and portraits}

Fix $M\in\mathbb Z$, $M\geq 2$, and a set
$$
\mathcal A\subset \mathbb Z_M=\{0,\dots,M-1\}\subset\mathbb Z.
$$
We define the following additive quantities, the first
of which was considered in~\eqref{e:E-ell}:
$$
\begin{aligned}
E_\ell(\mathcal{A})&:=\big|\{(a,b,c,d)\in\mathcal{A}^4:a+b-c-d=\ell\}\big|,
\quad \ell\in\mathbb{Z},\\
\widetilde{E}_\ell(\mathcal{A})&:=\big|\{(a,b,c,d)\in\mathcal{A}^4:a+b-c-d=\ell\mod{M}\}\big|,\quad \ell\in\mathbb{Z}_M.
\end{aligned}
$$
We record some of their properties which will be used later:
$$
\begin{gathered}
E_\ell(\mathcal{A})=E_{-\ell}(\mathcal{A}),\\
E_\ell(\mathcal{A})=0,\;\;\;\; |\ell|\geq2M-1,\\
\widetilde{E}_\ell(\mathcal{A})=E_{\ell-2M}(\mathcal{A})
+E_{\ell-M}(\mathcal{A})+E_{\ell}(\mathcal{A})+E_{\ell+M}(\mathcal{A}),\quad
\ell\in\mathbb Z_M.
\end{gathered}
$$
In particular,
\begin{align}
\label{e:decom-e0}
\widetilde{E}_0(\mathcal{A})&=E_0(\mathcal{A})+2E_M(\mathcal{A}),\\
\label{e:decom-e1}
\widetilde{E}_1(\mathcal{A})&=E_1(\mathcal{A})+E_{M+1}(\mathcal{A})+E_{M-1}(\mathcal{A}).
\end{align}
Moreover, we have the trivial bound 
\begin{equation}
\label{e:ub-trivial-tildee}
\widetilde{E}_\ell(\mathcal{A})\leq|\mathcal{A}|^3
\end{equation}
since for any $a,b,c\in\mathcal{A}$, there is at most one $d\in\mathcal{A}$ such that $a+b-c-d=\ell\mod{M}$.

We next introduce the following quantities which we call ``additive portraits'':
$$
\begin{aligned}
F_j(\mathcal{A})&=\big|\{(a,b)\in\mathcal{A}^2\colon a-b=j\}\big|, \quad j\in\mathbb{Z},\\
\widetilde{F}_j(\mathcal{A})&=\big|\{(a,b)\in\mathcal{A}^2\colon a-b=j\mod{M}\}\big|,
\quad j\in\mathbb{Z}_M=\{0,\dots,M-1\}.
\end{aligned}
$$
They have the following properties:
$$
\begin{gathered}
\widetilde{F}_j(\mathcal{A})\leq|\mathcal{A}|,
\quad
\widetilde{F}_j(\mathcal{A})=F_j(\mathcal{A})+F_{j-M}(\mathcal{A}),\\
F_j(\mathcal{A})=0 \text{ for } |j|\geq M,\\
\widetilde{F}_0(\mathcal{A})=F_0(\mathcal{A})=|\mathcal{A}|, \quad F_j(\mathcal{A})=F_{-j}(\mathcal{A}).
\end{gathered}
$$
Moreover,
\begin{equation}
\label{e:sum-j-f}
\sum_{j\in\mathbb{Z}_M}\widetilde{F}_j(\mathcal{A})=\sum_{j\in\mathbb{Z}}
F_j(\mathcal{A})=|\mathcal{A}|^2.
\end{equation}
Finally, additive energies and additive portraits are related as follows:
\begin{align}
  \label{e:ae-port-1}
E_\ell(\mathcal{A})&=\sum_{j\in\mathbb{Z}}F_j(\mathcal{A})F_{j+\ell}(\mathcal{A}),\\
  \label{e:ae-port-2}
\widetilde{E}_\ell(\mathcal{A})&=\sum_{j\in\mathbb{Z}_M}\widetilde{F}_j(\mathcal{A})\widetilde{F}_{(j+\ell)\bmod M}(\mathcal{A}).
\end{align}
We recall from~\eqref{e:ae-matrix}
that we need to estimate
$\rho(\mathcal A)$, the spectral radius of the $2\times 2$ matrix
\begin{equation}
  \label{e:Mat-A}
\mathcal M(\mathcal A)=\begin{pmatrix}
E_{M-1}(\mathcal A)+E_{M+1}(\mathcal A)&
2E_{M}(\mathcal A)\\
E_1(\mathcal A)&
E_0(\mathcal A)
\end{pmatrix}.
\end{equation}

\subsection{Approximate group structure}

By~\eqref{e:decom-e0}, \eqref{e:decom-e1}, the sums of the columns
of the matrix $\mathcal M(\mathcal A)$ are given by $\widetilde E_0(\mathcal A)$,
$\widetilde E_1(\mathcal A)$. In this subsection we prove that
unless $|\mathcal A|\sim M$,
both of these quantities cannot be close to the maximal value $|\mathcal A|^3$:
\begin{prop}
  \label{l:ae-cyclic}
Fix a number
\begin{equation}
\label{e:ub-epsilon0}
0<\epsilon_0<{\big(1-\sqrt{2/3}\big)^2\over 2}.
\end{equation}
Then at least one of the following statements is true:
\begin{align}
  \label{e:altc-1}
|\mathcal{A}|&\geq\frac{2}{3}(1-\sqrt{\epsilon_0/2})M,\\
  \label{e:altc-2}
\widetilde{E}_0(\mathcal{A})&\leq(1-\epsilon_0)|\mathcal{A}|^3,\\
  \label{e:altc-3}
\widetilde{E}_1(\mathcal{A})&\leq 2\sqrt{2\epsilon_0}\,|\mathcal{A}|^3.
\end{align}
\end{prop}
\Remark The alternatives~\eqref{e:altc-1}--\eqref{e:altc-3} can be explained
by the following examples. If $|\mathcal A|$ is very close to $M$,
then only \eqref{e:altc-1} holds.
If $\mathcal A=\{0,\dots,|\mathcal A|-1\}$ and
$1\ll |\mathcal A|\ll M$,
then only \eqref{e:altc-2} holds.
Finally if $\mathcal A$ is a proper subgroup of $\mathbb Z_M$ then only \eqref{e:altc-3}
holds.

To start the proof, we define for $\alpha\in(0,1)$,
\begin{equation*}
V_\alpha:=\big\{j\in\mathbb{Z}_M:\widetilde{F}_j(\mathcal A)\geq(1-\alpha)|\mathcal{A}|\big\}.
\end{equation*}
By \eqref{e:sum-j-f}, we obtain an upper bound on the size of $V_\alpha$:
\begin{equation}
\label{e:ub-valpha}
|V_\alpha|\leq(1-\alpha)^{-1}|\mathcal{A}|.
\end{equation}
A lower bound is provided by
\begin{lemm}
Suppose that~\eqref{e:altc-2} is false. Then
for all $\alpha\in (0,1)$
\begin{equation}
\label{e:lb-valpha}
\sum_{j\in V_\alpha}\widetilde{F}_j(\mathcal{A})\geq\Big(1-{\epsilon_0\over \alpha}\Big)|\mathcal{A}|^2,\quad
|V_\alpha|\geq\Big(1-{\epsilon_0\over \alpha}\Big)|\mathcal{A}|.
\end{equation}
\end{lemm}
\begin{proof}
We split the sum~\eqref{e:ae-port-2} into two parts:
\begin{equation*}
\widetilde{E}_0(\mathcal{A})=\sum_{j\in\mathbb{Z}_M}\widetilde{F}_j(\mathcal{A})^2
=\sum_{j\in V_\alpha}\widetilde{F}_j(\mathcal{A})^2+\sum_{j\in\mathbb{Z}_M\setminus V_\alpha}\widetilde{F}_j(\mathcal{A})^2.
\end{equation*}
In the first sum we use the trivial bound $\widetilde{F}_j(\mathcal{A})\leq|\mathcal{A}|$ and in the second sum we use the definition of $V_\alpha$ to get
$$
\begin{aligned}
(1-\epsilon_0)|\mathcal{A}|^3&\leq \widetilde{E}_0(\mathcal{A})\leq |\mathcal{A}|\sum_{j\in V_\alpha}\widetilde{F}_j(\mathcal{A})+(1-\alpha)|\mathcal{A}|\sum_{j\in\mathbb{Z}_M\setminus V_\alpha}\widetilde{F}_j(\mathcal{A})
\\&=(1-\alpha)|\mathcal{A}|\sum_{j\in\mathbb{Z}_M}\widetilde{F}_j(\mathcal{A})
+\alpha|\mathcal{A}|\sum_{j\in V_\alpha}\widetilde{F}_j(\mathcal{A}).
\end{aligned}
$$
Finally we use \eqref{e:sum-j-f} to get
\begin{equation*}
(1-\epsilon_0)|\mathcal{A}|^3\leq(1-\alpha)|\mathcal{A}|^3
+\alpha|\mathcal{A}|\sum_{j\in V_\alpha}\widetilde{F}_j(\mathcal{A})
\end{equation*}
which gives the first part of~\eqref{e:lb-valpha} and the second part follows since $\widetilde{F}_j(\mathcal A)\leq |\mathcal A|$.
\end{proof}
We next exploit the additive structure of the sets $V_\alpha$ for small $\alpha$.
Henceforth in this subsection, we use addition in the group $\mathbb Z_M=\mathbb Z/M\mathbb Z$.
\begin{lemm}
  \label{l:double}
For $\alpha<1/2$, we have $V_\alpha+V_\alpha\subset V_{2\alpha}$.
\end{lemm}
\begin{proof}
Note that $j\in V_\alpha$ if and only if 
\begin{equation}
\label{e:valphaj}
|\mathcal{A}\cap(\mathcal{A}+j)|\geq(1-\alpha)|\mathcal{A}|.
\end{equation}
Now for any $j,k\in V_\alpha$, we have \eqref{e:valphaj} and
\begin{equation*}
|(\mathcal{A}+j)\cap(\mathcal{A}+j+k)|
=|\mathcal{A}\cap(\mathcal{A}+k)|\geq(1-\alpha)|\mathcal{A}|.
\end{equation*}
Therefore
\begin{equation*}
\begin{split}
|\mathcal{A}\cap(\mathcal{A}+j+k)|
\geq &\;|\mathcal{A}\cap(\mathcal{A}+j)|+|(\mathcal{A}+j)\cap(\mathcal{A}+j+k)|-|\mathcal{A}+j|\\
\geq &\; 2(1-\alpha)|\mathcal{A}|-|\mathcal{A}|=(1-2\alpha)|\mathcal{A}|
\end{split}
\end{equation*}
which implies $j+k\in V_{2\alpha}$.
\end{proof}
We now fix
$$
\alpha:=\sqrt{\epsilon_0/2}.
$$
Combining~\eqref{e:ub-valpha}, \eqref{e:lb-valpha}, and Lemma~\ref{l:double},
we see that
\begin{equation}
  \label{e:gotcha-add}
|V_\alpha+V_\alpha|\leq|V_{2\alpha}|\leq{|\mathcal{A}|\over 1-2\alpha}\leq\frac{|V_\alpha|}{(1-2\alpha)(1-\epsilon_0/\alpha)}
=\frac{|V_\alpha|}{(1-\sqrt{2\varepsilon_0})^2}.
\end{equation}
Note that~\eqref{e:ub-epsilon0} implies that the right-hand side is strictly
less than ${3\over 2}|V_\alpha|$.

We next recall the inverse Fre\u\i man theorem for Abelian groups~\cite[Corollary~5.6]{Tao-Vu}:
\begin{theo}
  \label{t:freiman-in-da-haus}
Let $A$ be a finite subset of an Abelian group $G$ and $|A+A|<\frac{3}{2}|A|$. Then there exists a subgroup $H$ of $G$
and $a\in G$ such that $A\subset a+H$ with $|H|<{3\over 2}|A|$.
\end{theo}
Combining Theorem~\ref{t:freiman-in-da-haus} with the previous observations,
we obtain
\begin{lemm}
  \label{l:in-a-subgroup}
Suppose that both~\eqref{e:altc-1} and~\eqref{e:altc-2} are false.
Then there exists a factorization $M=LL'$, $L,L'\in\mathbb N$, $L>1$,
such that $V_\alpha\subset L\mathbb Z/M\mathbb Z$.
\end{lemm}
\begin{proof}
Applying Theorem~\ref{t:freiman-in-da-haus} to $V_\alpha\subset\mathbb Z_M$
and using~\eqref{e:gotcha-add}, we see that
there exist $L,L'\in\mathbb N$ with $LL'=M$ and $a\in\mathbb Z_M$
such that $V_\alpha\subset a+L\mathbb Z/M\mathbb Z$ and
$L'<{3\over 2}|V_\alpha|$. Moreover since $0\in V_\alpha$
we may take $a=0$. Finally by~\eqref{e:ub-valpha}
and the fact that~\eqref{e:altc-1} is false, we have
$$
L'<{3\over 2}|V_\alpha|\leq {3\over 2(1-\alpha)}|\mathcal A|<M
$$
which gives $L>1$.
\end{proof}
We now finish the proof of Proposition~\ref{l:ae-cyclic}. Assume that
both~\eqref{e:altc-1} and~\eqref{e:altc-2} are false. We
will show that~\eqref{e:altc-3} holds.
We recall that by~\eqref{e:lb-valpha} and~\eqref{e:sum-j-f},
\begin{equation}
\label{e:ub-valphac}
\sum_{j\in\mathbb{Z}_M\setminus V_\alpha}\widetilde{F}_j(\mathcal{A})\leq\sqrt{2\epsilon_0}\,|\mathcal{A}|^2.
\end{equation}
By Lemma~\ref{l:in-a-subgroup}, 
we see that for any $j\in V_\alpha$, $j+1$ and $j-1$ are not in $V_\alpha$. Therefore we can write by~\eqref{e:ae-port-2}
\begin{equation*}
\widetilde{E}_1(\mathcal{A})
=\sum_{j\in\mathbb{Z}_M}\widetilde{F}_j(\mathcal{A})\widetilde{F}_{j+1}(\mathcal{A})
\leq\sum_{j\in\mathbb Z_M\setminus V_\alpha}\widetilde F_j(\mathcal A)\big(\widetilde{F}_{j-1}(\mathcal{A})+\widetilde{F}_{j+1}(\mathcal{A})\big).
\end{equation*}
Now we can use \eqref{e:ub-valphac} and the trivial bound $\widetilde{F}_j(\mathcal{A})\leq|\mathcal{A}|$ to get
\begin{equation}
\label{e:ub-tildee1}
\widetilde{E}_1(\mathcal{A})
\leq 2|\mathcal{A}|\sum_{j\in\mathbb{Z}_M}\widetilde{F}_j(\mathcal{A})
\leq 2\sqrt{2\varepsilon_0}\,|\mathcal{A}|^3,
\end{equation}
obtaining~\eqref{e:altc-3}.

\subsection{Rearrangement inequalities and upper bounds for \texorpdfstring{$E_\ell$}{E\_l}}

We first recall rearrangement inequalities,
following~\cite[Chapter~X]{HLP}. Let $(a)=\{a_j\}_{j\in\mathbb{Z}}$ be a sequence of non-negative numbers with only finitely many non-zero elements. We denote the set of such sequences to be $\ell_c^+$.

We say $(a')\in\ell_c^+$ is a rearrangement of $(a)\in\ell_c^+$ if there is a permutation function $\phi:\mathbb{Z}\to\mathbb{Z}$ which is the identity for large $|j|$ such that $a_j'=a_{\phi(j)}$.
We define the special rearrangements $(a^+)$ and $({}^+a)$ by requiring
$$
\begin{gathered}
a_0^+\geq a_1^+\geq a_{-1}^+\geq a_2^+\geq a_{-2}^+\geq\cdots\\
{}^+a_0\geq{}^+a_{-1}\geq{}^+a_1\geq{}^+a_{-2}\geq{}^+a_2\geq\cdots
\end{gathered}
$$
When $(a^+)=({}^+a)$, we denote $(a^\ast):=(a^+)=({}^+a)$ and
call the sequence $(a)$ \emph{symmetrical}. We recall the following
rearrangement inequalities, see~\cite[\S\S10.4, 10.5]{HLP}
\begin{theo}[Rearrangements of two sets] 
For any two sequences $(a),(b)\in\ell_c^+$
\begin{equation}
\label{ineq:rearr-2}
\sum_{r+s=0}a_rb_s=\sum_j a_jb_{-j}\leq \sum_j a_j^+\cdot{}^+b_{-j}=\sum_{r+s=0}a_r^+\cdot{}^+b_s.
\end{equation}
\end{theo}
%
\begin{theo}[Rearrangements of three sets]
For any $(a),(b),(c)\in\ell_c^+$ with $(c)$ symmetrical, we have
\begin{equation}
\label{ineq:rearr-3}
\sum_{r+s+t=0}a_rb_sc_t\leq\sum_{r+s+t=0}a_r^+\cdot{}^+b_s\cdot c_t^\ast
=\sum_{r+s+t=0}{}^+a_r\cdot b_s^+\cdot c_t^\ast.
\end{equation}
\end{theo}
We next obtain bounds on the individual terms of the matrix~\eqref{e:Mat-A}:
\begin{prop}
  \label{l:ab-individual}
Assume $|\mathcal A|>1$. Then we have the following inequalities:
\begin{gather}
  \label{e:abin-1}
\max\big(E_1(\mathcal A),2E_M(\mathcal{A}),E_{M+1}(\mathcal{A})+E_{M-1}(\mathcal{A})\big)
\leq E_0(\mathcal{A}),\\
  \label{e:abin-3}
E_0(\mathcal{A})\leq \frac{2}{3}|\mathcal{A}|^3+\frac{1}{3}|\mathcal{A}|
\leq {3\over 4}|\mathcal A|^3.
\end{gather}
\end{prop}
\begin{proof}
First of all, applying~\eqref{ineq:rearr-2} to $F_j(\mathcal A)$ and
$F_{j+1}(\mathcal A)$ and using~\eqref{e:ae-port-1}
and the fact that $F_j(\mathcal A)=F_{-j}(\mathcal A)$, we have
$E_1(\mathcal A)\leq E_0(\mathcal A)$. To show~\eqref{e:abin-1}
it remains to prove the inequality
\begin{equation}
  \label{e:abin-1.1}
E_{k+M}(\mathcal{A})+E_{k-M}(\mathcal{A})\leq E_0(\mathcal{A}),\;\;\; k\in\{0,1\}.
\end{equation}
To show~\eqref{e:abin-1.1}, we use that $F_j(\mathcal A)=0$ for $|j|\geq M$
and write by~\eqref{e:ae-port-1}
$$
\begin{gathered}
E_{k+M}(\mathcal{A})+E_{k-M}(\mathcal{A})=\sum_{j}F_j(\mathcal{A})\big(F_{j+k+M}(\mathcal{A})+F_{j+k-M}(\mathcal{A})\big)=\sum_{j}F_j(\mathcal A)b_{j+k},\\
b_j(\mathcal A):=\indic_{[-M,M-1]}(j)\cdot\big(F_{j+M}(\mathcal{A})+F_{j-M}(\mathcal{A})\big).
\end{gathered}
$$
Since $b_j(\mathcal A)$ is equal to $0$ for $j=0$, to $F_{j-M}(\mathcal A)$
for $1\leq j\leq M-1$, and to $F_{j+M}(\mathcal A)$
for $-M\leq j\leq -1$, we see that
$b_j$ is a rearrangement of $F_j(\mathcal A)$, thus~\eqref{e:abin-1.1}
follows from~\eqref{ineq:rearr-2}.

It remains to prove~\eqref{e:abin-3}.
For that we write
$$
E_0(\mathcal{A})=\sum_{a,r\in\mathcal{A},\ c,s\in-\mathcal{A}\atop a+r+c+s=0}1=\sum_{r+s+t=0}1_\mathcal{A}(r)1_{-\mathcal{A}}(s)F_t(\mathcal{A}).
$$
The sequence $F_t(\mathcal A)$ is symmetrical since $F_t(\mathcal A)=F_{-t}(\mathcal A)$
and $F_t(\mathcal A)\leq |\mathcal A|=F_0(\mathcal A)$. Applying
the first inequality in~\eqref{ineq:rearr-3},
we get
$$
E_0(\mathcal{A})\leq \sum_{r=-\lfloor(|\mathcal{A}|-1)/2\rfloor}^{\lceil(|\mathcal{A}|-1)/2\rceil}\ \sum_{s=-\lceil(|\mathcal{A}|-1)/2\rceil}^{\lfloor(|\mathcal{A}|-1)/2\rfloor}F_{-r-s}^\ast(\mathcal{A})
=\sum_{j=1-|\mathcal{A}|}^{|\mathcal{A}|-1}(|\mathcal{A}|-|j|)F_j^\ast(\mathcal{A}).
$$
The right-hand side can be written as
\begin{equation}
\sum_{j\in\mathbb{Z}}C_jF_j(\mathcal A)
=\sum_{r+s+j=0}1_{\mathcal{A}}(r)1_{-\mathcal{A}}(s)C_j.
\end{equation}
where $C_j$ is some rearrangement of the symmetrical sequence $\{\max(|\mathcal{A}|-|j|,0)\}$.
Applying~\eqref{ineq:rearr-3} again and using
that $C^\ast_j=\max(|\mathcal A|-|j|,0)$, we finally get
$$
E_0(\mathcal{A})\leq \sum_{r=-\lfloor(|\mathcal{A}|-1)/2\rfloor}^{\lceil(|\mathcal{A}|-1)/2\rceil}\ \sum_{s=-\lceil(|\mathcal{A}|-1)/2\rceil}^{\lfloor(|\mathcal{A}|-1)/2\rfloor}C_{-r-s}^\ast\\
=\sum_{j=1-|\mathcal{A}|}^{|\mathcal{A}|-1}(|\mathcal{A}|-|j|)^2,
$$
finishing the proof.
\end{proof}

\subsection{End of the proof}

We now prove Proposition~\ref{l:ub-rho}. Assume that
$1<|\mathcal A|<{2\over 3}(1-\zeta)M$ for some $\zeta>0$,
and choose $\varepsilon_0$ satisfying~\eqref{e:ub-epsilon0}
and such that $\varepsilon_0\leq 2\zeta^2$,
then~\eqref{e:altc-1} is false.
Define the normalized matrix $\mathcal M(\mathcal A)$ as follows:
$$
\widetilde{\mathcal{M}}(\mathcal{A}):=|\mathcal{A}|^{-3}\mathcal{M}(\mathcal{A}),
$$
Then Proposition~\ref{l:ub-rho} follows from Propositions~\ref{l:ae-cyclic} and~\ref{l:ab-individual}, \eqref{e:decom-e0}--\eqref{e:ub-trivial-tildee},
and
\begin{lemm}
Assume that $p,q,r,s\in\mathbb R$ satisfy for some $\epsilon_0\in (0,1/8)$,
\begin{gather}
  \label{e:lots-of-conditions-1}
0\leq p,q,r\leq s\leq {3\over 4},\quad
p+r\leq 1,\quad q+s\leq 1,\quad\text{and}\\
  \label{e:lots-of-conditions-2}
\text{either}\quad p+r\leq 2\sqrt{2\epsilon_0}\quad\text{or}\quad
q+s\leq 1-\epsilon_0.
\end{gather}
Then the matrix
$$
\mathcal M:=\begin{pmatrix} p & q \\ r & s \end{pmatrix}
$$
has spectral radius less than $1-\varepsilon$, with
$\varepsilon>0$ depending only on $\epsilon_0$.
\end{lemm}
\begin{proof}
Since the subset of $\mathbb R^4$ defined by~\eqref{e:lots-of-conditions-1},
\eqref{e:lots-of-conditions-2}
is compact, it suffices to show that
$\mathcal M$ has spectral radius $<1$.
Assume the contrary. The eigenvalues of $\mathcal M$ are
$$
\lambda_\pm={p+s\pm\sqrt{(p-s)^2+4qr}\over 2}\in\mathbb R,\quad
|\lambda_-|\leq\lambda_+.
$$
Since $0\leq r\leq1-p$ and $0\leq q\leq1-s$, we have
$$
\lambda_+\leq{p+s+\sqrt{(p-s)^2+4(1-p)(1-s)}\over 2}
=1.
$$
Since $\lambda_+\geq1$ by our assumption, it follows that
$qr=(1-p)(1-s)$, leading to the following two cases:
\begin{enumerate}
\item $p=1$ or $s=1$: this is impossible since $p\leq s\leq 3/4$;
\item $p+r=q+s=1$: this is impossible by~\eqref{e:lots-of-conditions-2}.\qedhere
\end{enumerate}
\end{proof}

\medskip\noindent\textbf{Acknowledgements.}
We would like to thank Maciej Zworski for introducing us to open quantum maps and
many helpful comments and encouragement, St\'ephane Nonnenmacher for several
discussions on the history of the subject,
Jens Marklof for suggesting the short proof of Lemma~\ref{l:non-st},
and Rogers Epstein, Vadim Gorin, Izabella \L aba, Bjorn Poonen, Peter Sarnak, Hong Wang,
and Alex Iosevich for many enlightening discussions.
We are also grateful to an anonymous referee for many useful suggestions
to improve this article.
This research was conducted during the period SD served as
a Clay Research Fellow and
LJ served as a postdoctoral research fellow in CMSA at Harvard University.
SD and LJ would also like to thank Yau Mathematical Sciences Center
at Tsinghua University for hospitality during their stay
where part of this project was completed.


\end{document}